\documentclass{amsart}[12pt]
\usepackage[margin=2.3cm]{geometry}
\usepackage[dvips]{color}
\usepackage{graphicx}
\usepackage{epsfig}
\usepackage{tikz}
\usepackage{enumerate}
\usepackage{color}
\usepackage[T1]{fontenc}
\usepackage{mathrsfs}%
\usepackage{verbatim}
\usepackage{color}
\usepackage{subfigure}
\usepackage{xcolor}

\usepackage{anysize}
\marginsize{3,5cm}{2,5cm}{2,5cm}{2,5cm}
\usepackage{amssymb}
\usepackage{amsmath,amsbsy,amssymb,amscd}
\usepackage{t1enc}
\usepackage[cp1250]{inputenc}
\usepackage[british]{babel}
\usepackage[all]{xy}
\usepackage{color}
\usepackage{amsfonts}
\usepackage{latexsym}
\usepackage{amsthm}
\usepackage{mathrsfs}
\usepackage{hyperref}

\usepackage{color}

\def\geq{\geqslant}
\def\leq{\leqslant}

\def\vep{\varepsilon}

\theoremstyle{plain}

\newtheorem{proposition}{Proposition}[section]
\newtheorem{theorem}{Theorem}

\newtheorem{lemma}[proposition]{Lemma}

\theoremstyle{remark}
\newtheorem{remark}[proposition]{Remark}
\newtheorem{definition}[proposition]{Definition}

\def\geq{\geqslant}
\def\leq{\leqslant}
\def\R{\mathbb{R}}

\def\Z{\mathcal{Z}}
\def\N{\mathbb{N}}

\newcommand{\bea}{\begin{eqnarray}}
  \newcommand{\eea}{\end{eqnarray}}
  \newcommand{\beab}{\begin{eqnarray*}}
  \newcommand{\eeab}{\end{eqnarray*}}

  \newcommand{\be}{\begin{equation}}
  \newcommand{\ee}{\end{equation}}

\newcommand{\cC}{\mathcal T}
\newcommand{\cP}{\mathcal P}

\newcommand{\cT}{\mathcal T}

\newcommand{\norm}[1]{\left\lVert#1\right\rVert}

\def\cC{\mathcal{T}}

\newcommand{\RL}[1]{Z_{#1}}

\newcommand{\RR}{\mathbb{R}}
\newcommand{\V}{\mathcal{V}}

\newcommand{\RLp}[1]{Z^{(#1)}}
\newcommand{\kei}[1]{{k(I)}}

\newcommand{\ud}{\mathrm{d}}

\newcommand{\BS}[2]{S_{#2}#1}

\usepackage{latexsym}
\usepackage{amssymb}
\usepackage{amsmath}
\usepackage{amsthm}
\usepackage{amscd}
\usepackage{hyperref}

\newcommand{\spa}{\operatorname{span}}

\newcommand{\bes}{\begin{equation*}}
\newcommand{\ees}{\end{equation*}}
\newcommand{\st}{\,  | \, \, }

\newcommand{\Log}[1]{\mathcal{L}\textrm{og} \left( #1 \right)}
\newcommand{\SymLog}[1]{\mathcal{S}\textrm{ym}\mathcal{L}\textrm{og} \left( #1 \right)}
\newcommand{\pLog}[1]{pure\mathcal{L}\textrm{og}
 \left( #1 \right)}
\newcommand{\pSymLog}[1]{pure\mathcal{S}\textrm{ym} \mathcal{L}\textrm{og}
\left( #1 \right)}
\newcommand{\SymLogC}[1]{\mathcal{S}\textrm{ym}\mathcal{L}\textrm{og}^\textrm{2+} \left( #1 \right)}

\theoremstyle{definition}
\newtheorem{rem}{Remark}[section]

\title[Singularity of the spectrum of smooth area-preserving flows]{Singularity of the spectrum of typical minimal smooth area-preserving flows in any genus}

\author[K.\ Fr\k{a}czek]{Krzysztof Fr\k{a}czek}
\address{Faculty of Mathematics and Computer Science, Nicolaus
Copernicus University, ul.\ Chopina 12/18, 87-100 Toru\'n, Poland}
\email{fraczek@mat.umk.pl}

\author[A.\ Kanigowski]{Adam Kanigowski}
\address{Department of Mathematics, University of Maryland, College Park, MD USA and Faculty of Mathematics and Computer Science, Jagiellonian University, Lojasiewicza 6, Krakow, Poland}
\email{akanigow@umd.edu}

\author[C.\ Ulcigrai]{Corinna Ulcigrai}
\address{Institut f\"ur Mathematik, Universit\"at Z\"urich, Winterthurerstrasse 190,
CH-8057 Z\"urich, Switzerland}
\email{corinna.ulcigrai@math.uzh.ch}


\begin{document}

\begin{abstract}
We consider smooth flows preserving a smooth invariant measure, or, equivalently, locally Hamiltonian flows on compact orientable surfaces, and show that almost every such locally Hamiltonian flow
with only simple saddles has singular spectrum. Furthermore, we prove that {for almost every pair of such flows, the {elements of the pair} are spectrally disjoint.}
{ More generally, the results  from which these statements are deduced are} singularity of the spectrum and pairwise {spectral} disjointness  for special flows over  full measure sets of interval exchange transformations  under a roof with symmetric logarithmic singularities.  {Spectral singularity  is proved using a criterion}
 based on tightness of Birkhoff sums with exponential tail decay.
{ The assumptions of the criterion are verified exploiting the}  cancellations proved by the last author to prove the absence of mixing in this class of flows, by showing that the latter can be combined with rigidity {by exploiting the local product structure of Rauzy-Veech induction}. {{Pairwise} spectral disjointness} then follows by producing mixing times (for the second flow), using a new mechanism for shearing based on { what we call \emph{resonant}} rigidity times.
\end{abstract}

\maketitle

\bigskip

This paper pushes  the investigation of the nature of the spectrum for typical smooth area-preserving flows on surfaces of higher genus. Area-preserving flows are one of the most basic examples of dynamical systems, studied since Poincar{\'e} at the  dawn  of the study of dynamical systems. Smooth surface flows preserving a smooth area form are equivalently known as locally Hamiltonian flows and have been studied in connection with the Novikov model in solid state physics as well as pseudo-periodic topology (see e.g.~\cite{No:the} and \cite{Zo:how}).

The study of their ergodic properties, in particular mixing properties, has been a very active area of research in the last  {few decades}. We mention for example results on mixing  \cite{Ul:mix,  Rav:mix, CW} or its absence \cite{Sch:abs, Ul:abs, FaKaZe}, weak mixing  \cite{Ul:wea, AFS, Fo:eue},  multiple mixing \cite{FaKa, KKU},  and deviation phenomena (see e.g.\ \cite{Fr-Ul24,Fr-Ki1} 
on deviations of ergodic averages) and refer the interested reader to the surveys \cite{Ul:ECM,Ul:ICM,Fr-Ul:Enc} 
and the {references} therein for more details.

{Questions concerning finer spectral properties and  spectral disjointness, on the other hand, remain largely open.} Although { enquiring} about the nature of the spectrum and, in particular, the \emph{spectral type} (see \S~\ref{sec:spectral} for definitions) of locally Hamiltonian flows is very natural ({the question} appears already in the survey \cite{KT} by Katok and Thouvenot, see \S~6.3.2, { see also \cite{FK:ICM} and \cite{KL2}}), only a few  results were known until now, only in genus one and two, see \S~\ref{sec:history}. Quoting from a recent survey \cite{L:Katok} by M.~Lema\'nczyk: '\emph{Full understanding of spectral theory of smooth flows on surfaces seems to be one of the main challenges in ergodic theory in forthcoming years}'.

Within this program, we show in this paper that the \emph{typical} minimal such flow (in the measure--theoretic sense, see \S~\ref{sec:measure}) on a surface of any genus $g\geq 2$ has \emph{singular} spectrum (see Theorem~\ref{thm:singsp} in \S~\ref{sec:mainthm}).
In particular, this answers Question 5 in the survey \cite{KL2},
and provides {the first result on the spectral type}, to the best of our knowledge, which holds for a class of typical locally Hamiltonian flows in \emph{every} genus (see Section~\ref{sec:open_dj} for a discussion of the challenges which are still  open).

{A \emph{strengthening} of spectral results is often provided through the notion of \emph{spectral disjointness}  (for example through the notion of \emph{spectral disjointness from mixing flows}, or through results on \emph{pairwise spectral disjointness}, } see \S~\ref{sec:otherdis}).
Disjointness phenomena {(namely results concerning spectral disjointness, or the weaker notion of disjointness in the sense of Furstenberg, see \S~\ref{sec:otherdis})}
appear to be an important feature of parabolic flows (see for example \cite{KLU, FF} for results on time-changes of horocycle flows, \cite{DKW} for unipotent flows  or \cite{FoKa} for Heisenberg nilflows).
{However, they remain poorly understood and constitute a largely open problem in the context of locally Hamiltonian flows.}
In particular, results on spectral disjointness (as well as Furstenberg disjointness, see \S~\ref{sec:otherdis})  are few and concern mostly flows in genus one.

In this paper, we  also prove that  minimal locally Hamiltonian flows with simple saddles are {spectrally} \emph{disjoint} from all mixing flows (Theorem~\ref{thm:mixing_disj}, see \S~\ref{sec:Dmix}), and furthermore that, \emph{within} {this class},   any two \emph{typically} chosen such flows are \emph{spectrally disjoint} (see Theorem~\ref{thm:main_dj} in  \S~\ref{sec:disjointness thm3} for a precise statement). This shows that from the spectral point of view, {this} class of dynamical systems is very rich.


In the next introductory Section~\ref{sec:mainresults}, we give some basic definitions in order to state these results
 and put them into context. In Section~\ref{sec:sf_results} we will then restate the results in the language of special flows. In \S~\ref{sec:outline} we give an outline of the main ideas in the proof and describe how the rest of the paper is organized (see \S~\ref{sec:organization}).

\section{Main results and background}\label{sec:mainresults}
Let us first introduce in \S~\ref{sec:background} the basic notions needed  to formulate the results in the setting of flows on surfaces. We will then state the result on the spectrum in \S\ref{sec:mainsingular} and on the {spectral} disjointness results in \S\ref{sec:maindisjointness}.

\subsection{Background material}\label{sec:background}
Let us give in \S~\ref{sec:locHamflows} the definition of locally Hamiltonian flows, and describe the types of fixed points and the notion of typicality for such flows, see \S~\ref{sec:fixedpoints} and \S~\ref{sec:measure} respectively. We then give basic spectral definitions (in \S~\ref{sec:spectral}) and summarize the classification of locally Hamiltonian flows  and their typical ergodic properties in \S~\ref{sec:classes}.
\subsubsection{Locally Hamiltonian flows.}\label{sec:locHamflows}
Throughout this paper, let $M$ denote a smooth, compact,  connected, orientable surface of genus $g\geq 1$.  Let  $(\varphi_t)_{t\in\mathbb R}$ be  a smooth flow preserving a smooth invariant measure.
These flows are also known in the literature as \emph{locally Hamiltonian flows},  i.e.\ \emph{locally} one can find coordinates $(x,y)$ on $M$ in which $(\varphi_t)_{t\in\mathbb{R}}$ is given by
 the solution to the  equations
  $$\dot{x}={\partial H}/{\partial y}, \dot{y}=-{\partial H}/{\partial x}$$ for some smooth  real-valued \emph{local}\footnote{A \emph{global}  Hamiltonian $ H$ cannot be in general defined (see \cite{NZ:flo}, \S1.3.4), but one can think of  $(\varphi_t)_{t\in\mathbb{R}}$ as globally given by a \emph{multi-valued} Hamiltonian function.} Hamiltonian function $H$.
Equivalently, locally Hamiltonian flows can be given and \emph{parametrized} by \emph{closed $1$-forms} as follows.
 Fix   a smooth volume form $\omega$ on $M$, normalized so that the associated measure gives area $1$ to $M$.
Given  a smooth closed real-valued differential $1$-form $\eta $, let $X$ be the vector field determined by $\eta = \imath_X \omega $ (where {$\imath_X\omega$ denotes the contraction of $\omega$ with $X$}) and consider the flow $(\varphi_t)_{t\in\mathbb{R}}$ on $M$ associated to $X$. Since $\eta$ is closed, the transformations $\varphi_t$, $t \in \mathbb{R}$, are  area-preserving (i.e.\ preserve the area form $\omega$ and the measure given by integrating it), so  $\varphi_t$, $t \in \mathbb{R}$ is a locally Hamiltonian flow. Conversely,
{ given a smooth flow preserving a smooth invariant measure,  without loss of generality  one can assume that the smooth invariant measure is the given volume form $\omega$   and then show that 
flow is obtained by this construction} (see \S~\ref{sec:reduction1} in Appendix~\ref{sec:meas}).


\smallskip
\subsubsection{Fixed points.}\label{sec:fixedpoints}
{Let us from now consider (in view of \S~\ref{sec:locHamflows}) locally Hamiltonian flows preserving a fixed smooth are form $\omega$.} When $g\geq 2$, the  (finite) set of fixed points of $(\varphi_t)_{t\in\mathbb{R}}$ is always non-empty.
 We will always assume that the $1$-form $\eta$  associated to $(\varphi_t)_{t\in\mathbb{R}}$  is \emph{Morse}, i.e.\ it is locally the differential of a Morse function.  Thus, zeros of $\eta$ are isolated and finite and all  correspond to either centers as in Figure~\ref{island}, or simple saddles as in  Figure~\ref{simplesaddle}, as opposed to degenerate \emph{multi-saddles} which have $2k$ separatrices for $k>2$, see Figure \ref{multisaddle}.
We assume in this paper that $(\varphi_t)_{t\in\mathbb R}$ has only \emph{non-degenerate} fixed points of saddle type {and has no saddle loops homologous to zero}. Thus, all fixed points are  \emph{simple saddles} (i.e.\ four-pronged saddles, with two incoming and two outgoing separatrices), as depicted in Figure~\ref{g3flow} for a surface of genus $3$.

\begin{figure}[h!]
 \subfigure[\label{island} Center]{
  \includegraphics[width=0.20\textwidth]{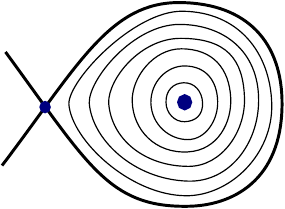} 	} \hspace{9mm}
\subfigure[Saddle\label{simplesaddle}]{ \includegraphics[width=0.14\textwidth]{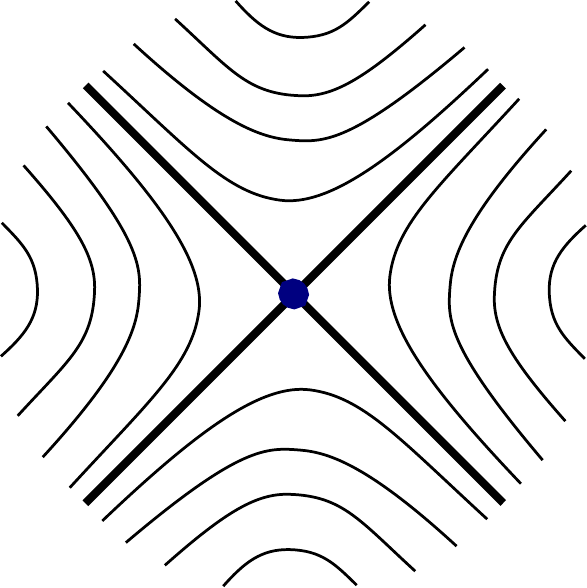}}\hspace{9mm}
\subfigure[Multisaddle\label{multisaddle}]{
\includegraphics[width=0.15\textwidth]{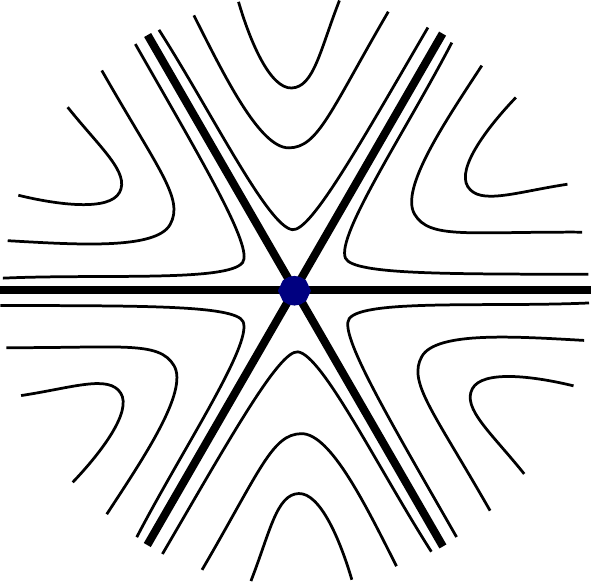} }	
 \caption{Types of fixed points for area-preserving flows.\label{saddles1}}
\end{figure}


\subsubsection{Measure class and typicality}\label{sec:measure}
{ When we fix the number and type of singularities, one has a natural \emph{measure-theoretic} notion of \emph{typical} locally Hamiltonian flow,} which
refers to a full measure set with respect to a natural measure class\footnote{We recall that two measures belong to the same measure class if they have the same sets of zero measure (and hence induce the same notion of full measure, or \emph{typical}); thus, a \emph{measure class} is uniquely identified by a collection of sets which have measure zero with respect to all measures in the class.}  known in the literature as \emph{Katok fundamental class} (introduced by Katok in \cite{Ka:inv}), see also \cite{NZ:flo}).
{
As we explain in Appendix~\ref{sec:meas}, it is enough to restrict our attention not only to locally Hamiltonian flows which preserve a fixed area form $\omega$ on the surface $M$, but also to those that have only non-degenerate saddles located in a \emph{fixed} finite set that we denote  $\Sigma\subset M$ (see \S~\ref{sec:reduction2} in Appendix~\ref{sec:meas}). Let $\mathcal{U}_{\min}(M,\omega,\Sigma)$ denote the set of such flows.

On $\mathcal{U}_{\min}(M,\omega,\Sigma)$ one can define the \emph{period map} as follows.
Let $n:=2g+\kappa-1$, where $g$ is the genus of $M$ and $\kappa$ is the cardinality of $\Sigma$. Then $n$ is the dimension of the relative $1$-homology space $H_1(M,\Sigma,\R)$. Using that flows have only simple saddles, one can also show that $n=4g-3$.
 Let $\gamma_1, \dots, \gamma_n\in H_1(M, \Sigma, \mathbb{Z})$ be a basis of $H_1(M, \Sigma, \mathbb{R})$.
The period map $Per:\mathcal{U}_{\min}(M,\omega,\Sigma)\to\R^n$ is then given by
\begin{equation}\label{def:Per}
Per((\psi_t)_{t\in\mathbb{R}})=\Big(\int_{\gamma_1}\eta,\ldots,\int_{\gamma_n}\eta\Big)\in\mathbb{R}^n,
\end{equation}
where $\eta$ is the closed $1$-form that gives the flow $(\psi_t)_{t\in\mathbb{R}}$.

A measure class on $\mathcal{U}_{\min}(M,\omega,\Sigma)$  is then}
 obtained considering  the pull-back $Per_* Leb$ of the Lebesgue measure class by the period map.
We say that a property is \emph{typical}  if it is satisfied for a set of locally Hamiltonian flows whose  complement has measure zero  for this measure class. 

\smallskip

 \begin{figure}[h!]
\includegraphics[width=0.6\textwidth]{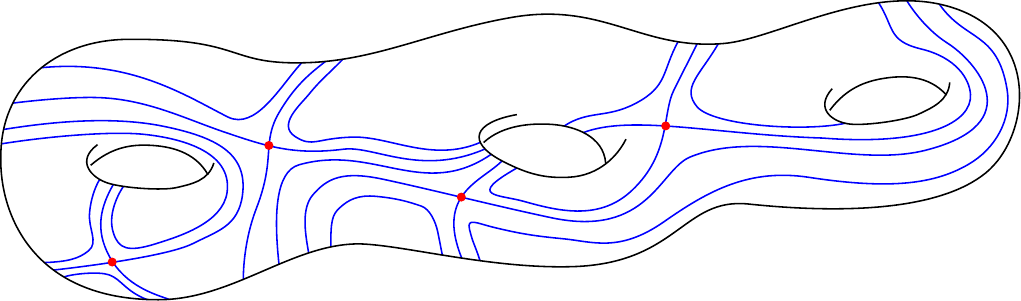}
 \caption{Trajectories of a locally Hamiltonian flow with four simple saddles on a surface of genus three.\label{g3flow}}
\end{figure}

\subsubsection{Spectral notions}\label{sec:spectral}
The spectrum and the spectral properties of a measure-preserving flow $(T_t)_{t \in \mathbb{R}}$ acting on a probability Borel space $(X,\mathcal{B},\mu)$ are defined in terms of the Koopman (unitary) operators associated to $(T_t)_{t \in \mathbb{R}}$. Let us recall that, for every $t\in\RR$, the \emph{Koopman} \emph{operator} associated to the automorphism  $T_t$, which, abusing notation,  we will denote also by $T_t$, is the operator
\[T_t:L^2(X,\mu)\to L^2(X,\mu)\quad \text{  given by \ \ }T_t(f)=f\circ T_t.\]
We denote by $L_0^2(X,\mu)$ the subspace of zero mean maps. We will sometimes restrict the Koopman representation to $L_0^2(X,\mu)$.

\medskip
\noindent {\it Spectral measures.} To every $g\in L^2(X,\mu)$ {and the flow $T=(T_t)_{t\in\R}$}, one can associate  a  \emph{spectral measure}  denoted by {$\sigma^T_g$}, i.e.\  the unique finite Borel measure on $\RR$ such that
\[\langle g\circ T_{t},g\rangle=\int_\RR e^{its}\,d\sigma^T_g(s)\quad\text{for every}\quad t\in \RR.\]
Let us denote by $\R(g)\subset  L^2(X,\mu)$ the \emph{cyclic subspace} generated by $g$ which is given by
\[\R(g):=\overline{\spa}\{T_t(g):t\in\R\}\subset  L^2(X,\mu)\]
By the spectral theorem (see e.g.~\cite{CFS:erg}) the Koopman $\R$-representation $(T_{t})_{t\in\R}$ restricted to $\R(g)$
is unitarily isomorphic to the $\R$-representation
\begin{equation}\label{def:Vt}
(V_t)_{t\in\R} \quad\text{on}\quad L^2(\R,\sigma_g)\quad \text{given by}\quad
V_t(h)(s)=e^{its}h(s).
\end{equation}

\noindent {\it Singular spectrum.} We say that a measure preserving flow $(T_t)_{t \in \mathbb{R}}$ on $(X,\mu)$ has (\emph{purely}) \emph{singular} spectrum iff for every $g\in L^2(X,\mu)$ the spectral measure $\sigma_g$ is singular with respect to the Lebesgue measure on $\R$.
Two measure preserving flows {$T=(T_t)_{t \in \mathbb{R}}$} on $(X,\mu)$ and {$S=(S_t)_{t \in \mathbb{R}}$} on $(Y,\nu)$ are \emph{spectrally {disjoint}} if the spectral measures $\sigma^T_f$ and $\sigma^S_g$ are pairwise orthogonal for all $f\in L^2_0(X,\mu)$ and $g\in L^2_0(Y,\nu)$.



\subsubsection{Classification of  locally Hamiltonian flows}\label{sec:classes}
Mixing properties of locally Hamiltonian flows are known to depend crucially on the type of singularities (i.e.~fixed points) of the flow and therefore this is expected to be the case also for spectral properties. We briefly recall in this subsection the subdivision of locally Hamiltonian flows in  three natural classes and the typical mixing properties within each, in order to put our main results into context and to later formulate open questions about spectral and disjointness properties (in  \S~\ref{sec:open} and \S~\ref{sec:open_dj} respectively). The reader may choose to skip this section and go directly to the results formulated in \S~\ref{sec:mainsingular} and \S~\ref{sec:maindisjointness}.

Flows which have \emph{degenerate} singularities, i.e.~saddles of higher order, with $2k$ prongs where {$k> 2$}, are known as \emph{{Ko\v{c}ergin} flows} and are usually those studied first since they are easier to study\footnote{Heuristically, we can explain this as follows. A key role in the study of mixing and spectral properties is played by quantitative estimates on the growth of the Birkhoff sums of the derivative of the roof function of the special representation, see e.g.\ \S~\ref{sec:strategy}.
In the case of Ko\v{c}ergin flows, the derivative is \emph{integrable},
while the derivative of flow with non-degenerate saddles is not in $L^1$. This makes  Birkhoff sums harder  to bound in the non-degenerate case.}, but are not generic (in the topological sense, with respect to the topology given by small perturbations by a closed $1$-form). Mixing for Ko\v{c}ergin flows in {any genus} was for example the very first mixing result proved by Ko\v{c}ergin \cite{Ko:mix}\footnote{Ko\v{c}ergin was indeed the first to study this class of flows in any genus, by studying special flows over interval exchange transformations under a roof with power-type of singularities, see \cite{Ko:mix}.}.

{The set of locally Hamiltonian flows with only \emph{non-degenerate} (or Morse-type) singularities is divided in two open (again with respect to small perturbations by a closed $1$-form) subsets $\mathcal{U}_{min}$ and $\mathcal{U}_{\neg min}$.} {In the open set $\mathcal{U}_{min}$ (defined by asking that there are no saddle loops homologous to zero)} the typical flow is \emph{minimal} {(here, this means that all orbits except fixed points are dense)} and (as shown by the last author) \emph{weakly mixing}, but not \emph{mixing} \cite{Ul:abs}   (although \emph{exceptional} flows which are \emph{mixing} 
can exist in $\mathcal{U}_{min}$,
as shown by  Chaika~and Wright  \cite{CW} who built an example in $g=5$). As in the case of the proof of singularity of the spectrum, the special case of flows on surfaces of genus $g=2$ with two isomorphic saddles is easier to treat and absence of mixing in this case was proven before, by Scheglov in \cite{Sch:abs}. 

In  the open set  $\mathcal{U}_{\neg min}$, on the other hand, the typical flow decomposes into (finitely many) subsurfaces, either foliated by closed orbits of the flow (known as \emph{periodic components}) or on which the restriction of the flow is minimal (the \emph{minimal components}, which are at most $g$, where $g$ is the genus), see e.g.~\cite{Ma:tra}.
The restriction of the typical flow in $\mathcal{U}_{\neg min}$ to each of its minimal components is \emph{mixing} (as shown by Ravotti \cite{Rav:mix}, building on work by the last author \cite{Ul:mix}) as well as mixing of all orders (see \cite{KKU}).
The simplest example of flow in $\mathcal{U}_{\neg min}$ is a flow on the torus ($g=1$) with one simple center and one simple saddle. As remarked already by Arnold in \cite{Ar:top}, the typical such flow has a  disk filled by periodic orbits and the center and is minimal when restricted to the complement of the disk. 
The restriction of the flow to this minimal component is nowadays called the \emph{Arnold flow}.  Mixing for Arnold flows was conjectured by Arnold in \cite{Ar:top} and proved by Sinai-Khanin in \cite{SK:mix}\footnote{See also further works by Ko\v{c}ergin \cite{Ko:03, Ko:04, Ko:04'} on flows on the torus with many minimal components.}. {Mixing of all orders for almost all  Arnold flows was proved in  \cite{FaKa} by B.~Fayad and the second author.} The same authors, together with Zelada, also showed in \cite{FaKaZe} the existence of an exceptional \emph{non}-mixing minimal component for a flow (on a genus two surface) in $\mathcal{U}_{\neg min}$.


\subsection{Singularity of the spectrum of locally Hamiltonian flows with simple saddles}\label{sec:mainsingular}
{In this paper we first prove a conjecture on the spectrum of locally Hamiltonian flows. When the flow has only simple saddles, the nature of the spectrum  was conjectured to be singular by}  A.~Katok\footnote{Private communication with the second author, A.~Kanigowski.}, see also the surveys \cite{L:Katok} and \cite{KL2} (in particular Question 5)\footnote{{Question 5 in \cite{KL2} states the problem of the nature of the spectrum of a typical flow in $\mathcal{U}_{min}$ as an open question.} The  singular nature of the spectrum proved in this paper was furthermore explicitly suggested in private communications.}.

\subsubsection{Spectral singularity of locally Hamiltonian flows}\label{sec:mainthm}
{ The first of our main results, that confirms the conjecture,} is the following.
\begin{theorem}[singular spectrum of typical higher genus flows with simple saddles]\label{thm:singsp}
For {every} genus $g\geq 2$, a \emph{typical} {minimal} locally Hamiltonian flow on a surface $M$ of genus $g$ with only simple saddles has \emph{purely singular} spectrum.
\end{theorem}
\noindent Theorem~\ref{thm:singsp} is, to the best of our knowledge, the first result on {the spectral type for typical minimal locally Hamiltonian flows in every genus}. We recall here below in \S~\ref{sec:history} the results previously known; { see also \S~\ref{sec:open}
 for  some open questions on spectra of locally Hamiltonian flows.} Theorem~\ref{thm:singsp} will be reduced to a result on special flows (Proposition~\ref{thm:dm_sf}) in \S~\ref{sec:reduction}. An outline of the proof, highlighting the difficulties and novelties is given after the special flows reformulation,  in \S~\ref{sec:strategy}.

\subsubsection{Previous results on spectra of locally Hamiltonian flows}\label{sec:history}
A special case of the main Theorem~\ref{thm:singsp} was proven by the  authors together with J.~Chaika for surfaces of genus $g=2$ in \cite{Ch-Fr-Ka-Ul}, where it is shown that a typical locally Hamiltonian flow on a surface of genus $g=2$ with two \emph{isomorphic}\footnote{We refer to  \cite[\S~2.4]{Ch-Fr-Ka-Ul} for the definition of isomorphic saddles.} \emph{simple saddles} has singular spectrum. While all minimal flows in genus two given by a non-degenerate (Morse type) local Hamiltonian have two \emph{simple} saddles, the assumption that the saddles are \emph{isomorphic} (which is necessary for the techniques in \cite{Ch-Fr-Ka-Ul} to be applicable) is an additional restriction on the class of flows.  Thus, Theorem~\ref{thm:singsp} generalizes to any genus $g\geq 2$ the result proved in \cite{Ch-Fr-Ka-Ul} in the special case of  genus two and also covers the more general case of locally Hamiltonian flows in genus two for which the saddles are not isomorphic. 
We comment on the differences between the {proofs} in  the special case of genus two and the general case in \S~\ref{sec:strategy}.

We remark that on a surface of genus \emph{one} no flow with only simple saddles can exist, due to Euler-characteristic restriction\footnote{This can be seen using Poincar{\'e}-Hopf theorem, which links the genus of the surface and the number and type of singularities: if $s$ is the number of simple saddles and $c$ of centers, while $g$ denotes the genus, $2g-2=s-c$.}. 
Nevertheless, a precursor of this and the previous genus two result in the language of special flows was proved in  an early work by the first author and Lema\'nczyk, see \cite{Fr-Le0} (where it is shown that
 special flows over rotations under a roof with a symmetric logarithmic singularity  
  have singular spectrum\footnote{This result is a corollary of  Theorem~12 in \cite{Fr-Le0}.} for a full measure set of rotation numbers, see \S~\ref{sec:main_sf} for definitions).
In the same paper \cite{Fr-Le0} \emph{examples} of locally Hamiltonian flows on surfaces of \emph{any} genus $\geq 2$ with  singular continuous spectrum  (see \cite[Theorem 1]{Fr-Le0}) are constructed. Although these have the merit of being the very first examples with known spectral type in higher genus, by the nature of the construction, they are highly non-typical and form a measure zero set in the space of locally Hamiltonian flows.

A recent spectral breakthrough in genus one, which goes in the opposite direction, was achieved by Fayad, Forni and the second author in \cite{FFK}, where it is shown that a
smooth flow with a sufficiently strong \emph{stopping point} on surfaces of \emph{genus one}  {has} absolutely continuous spectrum\footnote{More precisely, the authors show that these flows, which  can be represented as special flows over rotations, under roof functions with a polynomial singularity, have
\emph{countable Lebesgue spectrum}, as long as the power singularity is sufficiently \emph{strong}, i.e.~for a singularity of type $1/x^\alpha$, where $\alpha>1-\frac{1}{1000}$.} (see also \cite{FU, Sim} for results of similar type for time-changes of homogeneous flows). 
 We describe a unifying conjectural picture to reconcile these antipodal results in \S~\ref{sec:open} below, after recalling the classification of locally Hamiltonian flows.

\subsubsection{Open questions on spectra of  locally Hamiltonian flows} \label{sec:open}
Our main result (Theorem~\ref{thm:singsp}) settles the conjecture about the nature of the spectrum for typical flows in the open class $\mathcal{U}_{min}$ introduced in \S~\ref{sec:classes}.
{ For contrast,   the nature of the spectrum (for the restriction of a typical flow to a minimal component) in the complementary set $\mathcal{U}_{\neg min}$,
is a longstanding open question (see {Question 35} in the ICM Proceedings \cite{FK:ICM}). To the best of our knowledge there are currently no results, not
 even in genus one nor in other special cases.} These flows are indeed mixing, but with sub-polynomial rate {(see the work \cite{Rav:mix} by Ravotti, which provides logarithmic upper bounds for any genus and the recent work \cite{LL} by Lorenzo-Laguno  which provides different logarithmic lower bounds)} and therefore even whether to expect singularity or  absolute continuity of the spectrum is unclear.

{Very little is known on the nature of the spectrum in
also in presence of \emph{degenerate} singularities, or equivalently \emph{multi-saddles}. These types of degenerate singularities, i.e.~saddles as in Figure~\ref{multisaddle} (with $2n$ prongs, $n> 2$) produce indeed \emph{power-like} singularities in the special flow. The corresponding representations as special flows
(i.e.~\emph{Ko\v{c}ergin flows}) are} analogous to those that appear in the representation of flows on the torus with a stopping point, considered in the work \cite{FFK} by Fayad, Forni and the second author. A \emph{stopping point} can therefore be thought as \emph{degenerate} saddles, with only two separatrices (i.e.~$2n$ prongs, for $n=1$).
We remark that in \cite{FFK} the singularity is assumed to be of the form $1/x^\alpha$, where $\alpha$ is sufficiently close to $1$.
It might therefore be conjectured, from their result, that also in higher genus, the spectrum is absolutely continuous (and even countable Lebesgue), at least  in the presence of sufficiently strong \emph{degenerate} singular points, i.e.~when the number of prongs is sufficiently large.\footnote{The exponents of the power type singularities (here we call exponent the number $0<\alpha<1$ such that the singularity is of order $1/x^\alpha$) depend indeed on the number of prongs and approach $1$ (i.e.~gives stronger singularities) as the number of prongs grow, more precisely $\alpha=\frac{n-2}{n}$ if $2n$ is the number of prongs.}  Since
 the assumption $\alpha>1/2$ is clearly a necessary condition for the    methods of \cite{FFK} to work, the nature of the spectrum in presence of power singularities $1/x^\alpha$ with  $0<\alpha<1/2$ is, even conjecturally, less clear, and its investigation will require new methods.

{Finally}, another natural question in spectral theory is the so-called \emph{multiplicity} problem, namely what is the multiplicity of the spectrum (see e.g.~the survey \cite{L:Enc} for definitions and questions). The only result concerning multiplicity, to the best of our knowledge, is contained in the work  \cite{FFK} mentioned above about flows on the torus with \emph{stopping points}, in which it is proved not only that the spectrum is absolutely continuous, but also that it is countable Lebesgue (i.e.~Lebesgue with countable multiplicity). {Both in presence of non-degenerate saddles in higher genus and for special flows over rotations or IETs with only logarithmic singularities (symmetric or asymmetric)}, the multiplicity problem is  completely open.


\subsection{Spectral disjointness of locally Hamiltonian flows with simple saddles}\label{sec:maindisjointness}
 {  Several results in the literature on parabolic flows suggest that various type of disjointness phenomena may prove to be typical features of parabolic dynamics}, see e.g.~\cite{KLU, FF, DKW, FoKa}. { Few  of these results, though,} were proved in the context of locally Hamiltonian flows, { { mostly} for genus one surfaces (see below).
As examples of different type of}  disjointness phenomena, 
let us mention {Furstenberg or spectral} \emph{disjointness from mixing flows}, that can be seen as a strengthening of absence of mixing (see \S~\ref{sec:Dmix}), {
\emph{pairwise disjointness}  (i.e.~disjointness
of the flows belonging to  a \emph{typical} pair of elements of the class considered) in Furstenberg or spectral sense}   (see e.g.~\S~\ref{sec:disjointness thm3}), and finally
 {(Furstenberg or spectral)}  \emph{disjointness of rescalings} of the flow, see \S~\ref{sec:otherdis}.
While disjointness of rescalings { understood in the sense of Furstenberg}
  has seen a surge of interest in relation with work on the Sarnak conjecture (see e.g.~\cite{BSZ, L:Sarnak}), we stress that this type of disjointness is {one of the natural disjointness} features to investigate \emph{per se}, {as a candidate to be  one of the} distinguishing features of parabolic dynamics.


\subsubsection{{Spectral} disjointness from mixing flows}\label{sec:disjointness thm2}\label{sec:Dmix}
The first {spectral} disjointness result which we prove is that a typical minimal locally Hamiltonian flow in $\mathcal{U}_{min}$ is \emph{spectrally disjoint} from all mixing flows, in particular, from typical flows in $\mathcal{U}_{\neg min}$ and typical flows with degenerate fixed points. This {spectral} disjointness result
 is also a consequence of the type of criterion for   singularity of the spectrum that we use in this paper (see Theorem~\ref{thm:singcrit}). 
 For this reason, what we actually prove in this paper is the following stronger result:
\begin{theorem}\label{thm:mixing_disj}
For {every} genus $g\geq 2$, a \emph{typical} minimal  locally Hamiltonian flow on a surface $M$ of genus $g$ with only simple saddles is {spectrally disjoint from all mixing flows}.
\end{theorem}

{This theorem will therefore be proved simultaneously with the singularity of the spectrum (i.e., together with Theorem~\ref{thm:singsp}), since both are consequences of the same spectral singularity criterion.}



\subsubsection{{Typicality of pairwise spectral disjointness}\label{sec:disjointness thm3}}
The second and main {spectral} disjointness result concerns flows \emph{within} the class we consider, namely (minimal, non-degenerate) locally Hamiltonian flows in $\mathcal{U}_{min}$. A natural question to investigate is whether a \emph{typical} \emph{pair} of such flows is related to each other,
{where \emph{typical pair} should be intended as follows.

As in \S~\ref{sec:measure} (in view of Appendix~\ref{sec:meas}), we can restrict our attention to flows in $\mathcal{U}_{\min}(M,\omega,\Sigma)$, i.e.~flows without saddle loops that preserve a fixed $\omega$ and have singularities on  a fixed set $\Sigma\subset M$.
We then say that a result holds for a \emph{typical pair} if  the
subset of \emph{pairs} $((\varphi_t)_{t\in\R},(\psi_t)_{t\in\R})$ of flows  in the product space $ \mathcal{U}_{\min}(M,\omega,\Sigma)\times \mathcal{U}_{\min}(M,\omega,\Sigma)$ for which it holds  is \emph{measurable} and
has full measure with respect to the pullback via $Per\times Per$ of Lebesgue measure class on $\mathbb{R}^n\times \mathbb{R}^n$. We then also say that the result holds \emph{for almost every pair} of locally Hamiltonian flows in $\mathcal{U}_{min}$.


Using this measure-theoretic notion of \emph{typical} pairs,  the following
result proves that the flows in a typical pair of locally Hamiltonian flows in $\mathcal{U}_{min}$ are \emph{spectrally disjoint}:}
\begin{theorem}[{pairwise disjointness}]\label{thm:main_dj}
{For every genus $g\geq 2$,  for almost every pair $((\varphi_t)_{t\in\R},(\psi_t)_{t\in\R})$ of minimal locally Hamiltonian flows on a surface $M$ of genus $g$ with only simple saddles, the  flows  $(\varphi_t)_{t\in\R}$ and  $(\psi_t)_{t\in\R}$  are spectrally disjoint.}
\end{theorem}
{We will refer to this result as \emph{typical pairwise disjointness} of flows in $\mathcal{U}_{min}$.}
{
Theorem~\ref{thm:main_dj} illustrates the spectral richness of the class of minimal locally Hamiltonian flows with simple saddles. It complements Theorem~\ref{thm:singsp} by showing that, although the spectrum is typically singular, its singular type is not universal within the class: for a typical pair of such flows, the corresponding spectral types
 are mutually singular. Thus, from the spectral point of view, typical flows in this class are not only non-mixing and singular, but also pairwise unrelated.}


\subsubsection{Other results on disjointness in the literature}\label{sec:otherdis}
A notion of disjointness was introduced by Furstenberg \cite{Fur}: two flows are \emph{disjoint in the sense of Furstenberg} (or simply \emph{Furstenberg disjoint}) if they have no common \emph{joining} { other than the product measure}, see e.g.~\cite{KLU} or \cite{dlR} for definitions. {Recall that the observation that spectral disjointness
implies disjointness was already proven by Hahn-Parry in \cite{Ha-Pa}.} Disjointness (both spectral and in the sense of Furstenberg) has been investigated not only within a given class of dynamical systems, but  also between \emph{rescalings} of the flow; let us recall that given a flow $(\varphi_{t})_{t\in\mathbb{R}}$, and $\kappa\neq 1$, the $\kappa$-rescaling is the flow
 of the form
 $(\varphi_t^\kappa)_{t\in\mathbb{R}}:= (\varphi_{t\kappa})_{t\in\mathbb{R}}$ obtained by rescaling time.

  A phenomenon which has been proved for several classes of \emph{slowly chaotic} flows, possible only in entropy zero, is the so-called \emph{disjointness of rescalings}, namely that the rescalings $\varphi_\mathbb{R}^\kappa$ of $\varphi_\mathbb{R}$ are disjoint (spectrally, or in the sense of Furstenberg) for almost every $\kappa\in\mathbb{R}$, unless there is an algebraic reason which prevents it. Results in this direction include  for example \cite{KLU, FF} for time-changes of horocycle flows\footnote{Note that horocycle flows themselves have isomorphic rescalings (where the renormalization is given by the geodesic flow). Joining rigidity of (time-changes of) horocycle and unipotent flows was one of the key discoveries of Marina Ratner, see e.g.~\cite{Rat:Act, Rat:Inv}. See also \cite{FF, AFR, T1, T2} for recent results about rigidity of time-changes of horocycle and unipotent flows.
  However, for many non-trivial \emph{time-changes} (also called~time \emph{reparametrizations}) of horocycle flows, different \emph{rescaling}  (i.e.\ $p\neq q$) are disjoint, see \cite{KLU, FF}. On the other hand, rescalings of (time changes) of horocycle flows  have (countable) Lebesgue spectrum (see \cite{P} for the classical horocycle flow and \cite{FU,Ti} for results on its time-changes), so they are always spectrally isomorphic; in  \cite{KLU}  it is shown that they are nevertheless typically Furstenberg disjoint.},
 \cite{DKW} for disjointness in unipotent flows,
 \cite{CE} for (disjointness of powers of)  $3$-IETs, and  \cite{FoKa} for Heisenberg nilflows.




For special flows over rotations, when the roof is a symmetric logarithm,
 flows
(which can be thought of as a toy model for flows in $\mathcal{U}_{min}$),
P.~Berk and the second author in \cite{BK} proved  spectral disjointness of \emph{rescalings} of the form $(\varphi_t^\kappa)_{t\in\mathbb{R}}$ for any positive rational $\kappa\neq 1$.
 Only very few other results about \emph{spectral} disjointness are known\footnote{In the same paper \cite{BK}, also IETs and piecewise constant functions with fake discontinuity are shown to be typically disjoint; in \cite{ALL}, the authors give a class of rank one transformations for which all powers are spectrally disjoint.}.


Another class of surface flows for which disjointness was investigated are so called {\it von Neumann flows}\footnote{This class was introduced in \cite{vN}
as the first systems with continuous spectrum (weakly mixing systems).}. These are
{defined as} special flows over rotations under a piecewise linear roof; one can show that they arise from the study of linear flows on surfaces with boundary, see \cite{Fr-Le3,Co-Fr}.
In \cite{DK},
 Dong and the {second} author
 showed that \emph{any} two different rescalings of a von Neumann flow are disjoint, unless they are isomorphic. %


 Finally, in the specific setting of locally Hamiltonian flows, for typical Arnold flows, i.e.~typical flows in $\mathcal{U}_{\neg min}$ for $g=1$,
   Lema\'nczyk and the last two authors proved~in \cite{KLU} some disjointness properties, in particular disjointness of rescalings in the sense of Furstenberg  (see \cite[Theorems 1.2]{KLU}) and disjointness from smooth time-changes of the horocycle flows (see \cite[Theorem 1.3]{KLU}) for {Arnold flows}  (which were defined in \S~\ref{sec:classes}).
{ Finally, while this paper was being refereed, P.~Berk and the third author have posted a new preprint \cite{BU} announcing (Furstenberg) disjointness of \emph{rational} rescalings of a typical flow in $\mathcal{U}_{min}$ for any $g\geq 2$.}

\subsubsection{Open questions on disjointness}\label{sec:open_dj}
Since {few} results on disjointness are known, 
 many questions remain open. { In this paper, we answer (positively) the question of spectral disjointness between typical pairs of elements of the space $\mathcal{U}_{min}$ of}  locally Hamiltonian flows (i.e.~prove typicality of pairwise disjointness in the sense explained in \S~\ref{sec:disjointness thm3}), { or more generally special flows under roofs with logarithmic \emph{symmetric} singularities.}  It is hence natural to ask whether {an analogous result holds for $\mathcal{U}_{\neg min}$, namely whether one has  disjointness (spectrally, or at least in the sense of Furstenberg) also for  typical \emph{pairs} of elements of the space  $\mathcal{U}_{\neg min}$, or more generally for typical pairs of} special flows under roofs with logarithmic \emph{asymmetric} singularities.   

{ Another type of disjointness question concerns \emph{rescalings}.}
As mentioned in \S~\ref{sec:otherdis} above,  results on (Furstenberg) disjointness of \emph{rescalings} { are known for flows with both symmetric and asymmetric singularities over rotations (see \cite{BK} and \cite{KLU} respectively). One can also ask whether disjointness of rescalings is typical  for Ko\v{c}ergin flows, namely flows over rotations with power-type singularities (see \S~\ref{sec:classes}). }

{ For higher genus surfaces ($g\geq 2$), as mentioned in \S~\ref{sec:otherdis} above,
 (Furstenberg) disjointness of rescalings has recently been proved in \cite{BU} for typical flows in  $\mathcal{U}_{min}$. Disjointness in the sense of Furstenberg, though, does not imply spectral disjointness; thus, typical spectral disjointness of rescalings in  $\mathcal{U}_{min}$ remains an open (and possibly very hard) question.
Within  $\mathcal{U}_{\neg min}$, only
Furstenberg disjointness of rescalings
 for typical Arnold flows (i.e.~when $g=1$) was proved in \cite{KLU}. Furstenberg disjointness of rescalings for typical flows in $\mathcal{U}_{\neg min}$ is still an open problem as soon as $g>1$.
 \emph{Spectral} disjointness of rescalings is open in any genus, including $g=1$  (i.e.~Arnold flows) and seems a difficult problem (since, as recalled in \S~\ref{sec:open},  already the nature of the spectrum of a typical Arnold flow is a long-standing open question).} 

  Finally, it may be interesting to study the general problem of isomorphisms and rigidity of joinings in the class of locally Hamiltonian flows. Since locally Hamiltonian flows dynamically share many similarities with horocycle flows, in analogy to Ratner's theory for horocycle flows, one could ask whether two locally Hamiltonian flows are non-isomorphic (or even disjoint) \emph{unless} there is an  \emph{arithmetic}\footnote{Here the generalization of the notion of \emph{arithmetic} reason from rotation to IETs may involve the notion of Diophantine-like conditions in higher genus, see the survey \cite{Ul:ICM}.}
  reason on the base IET (or rotation). In this generality, this problem would need new tools especially in higher genus situation.


{
\section{Special flows results and reduction}
We now  state the main results that we prove for special flows (definitions are given below) over interval exchange transformations (IETs) under roofs with symmetric logarithmic singularities.}
A classical reduction (through a choice of a suitable Poincar{\'e} section) that we recall in this section  allows to study ergodic and spectral properties of locally Hamiltonian flows through their \emph{special flow} representation. {These representations are indeed special flows under roofs with symmetric logarithmic singularities, over a rotation when the genus is one, and more generally over IETs for any $g\geq 2$.  We stress, though, that results for special flows in this class are in principle \emph{more general}\footnote{{Indeed, the type of logarithmic singularities that arise from the representation of locally Hamiltonian flows in $\mathcal{U}_{min}$ have additional relations among the constants $C^\pm_i$ in the expression \eqref{eq:formlogsing} and form a special subclass of logarithmic symmetric singularities, sometimes referred to in the literature as logarithmic singularities of \emph{geometric type}, see e.g.~\cite{Fr-Ul}.}} than results for locally Hamiltonian flows and of independent interest.}

\subsection{Special flows over  IETs with symmetric logarithmic singularities}\label{sec:main_sf}
 Let us first recall the definitions of interval exchange {transformations} (\S~\ref{subsec:IETs}) and of special flows (\S~\ref{sec:sf}). Then we  can formulate our main result in the setting of special flows over IETs.
 \subsubsection{Interval exchange transformations.}\label{subsec:IETs}
An \emph{interval exchange transformation} (IET) of $d$ intervals $T:I\to I$ (also called {$d$-IET} for short) on an interval  $I=[0,\ell]$ (often we will consider the unit interval, i.e.~assume $\ell=1$)   \emph{with permutation} $\pi$ (on $\{1,\dots, d\}$)  and \emph{endpoints} (of the continuity intervals)
\[End(T):=\{ \beta_i, \ 0\leq i \leq d\}, \quad \text{where} \  0=:\beta_0 < \beta_1 < \dots \beta_{d-1}< \beta_d:= |I|, \]
is a piecewise isometry which sends the interval $I_i:=[\beta_{i-1},\beta_{i})$, for $1\leq i\leq d$,
by a translation, i.e.\ there exists a real $\delta_i$ such that
\be \label{IETdef}
T(x) = x + \delta_i,
\qquad \text{if }\ x \in I_i.
\ee
Thus, {an IET} is uniquely determined by a permutation $\pi$ on $d$ symbols and the lengths $\lambda_i:=|I_i|=\beta_{i}-\beta_{i-1}$ of the exchanged intervals, described by the length vector $\lambda= (\lambda_i)_{i} \in \mathbb{R}^d_{>0}$.
We will write  $T=(\pi, \lambda)$ for the interval exchange transformation with permutation $\pi$ and lengths $\lambda$.

\begin{rem}
Notice for later use that the \emph{displacements} $\delta_i$ in \eqref{IETdef} are a linear function of the lengths, namely if $\delta$ and $\lambda$ are the column vectors in $\mathbb{R}^d$ with entries respectively $\delta_i$ and $\lambda_i$ for $1\leq i\leq d$,
we can write $\delta = \Omega_\pi \lambda$ where, for $1\leq i,j\leq d$, 
\begin{equation*} \label{def:Omega}
(\Omega_\pi)_{ij}=
\left\{\begin{array}{cl} +1 & \text{ if
}i<j\text{ and }\pi(i)>\pi(j),\\
-1 & \text{ if }i>j\text{ and }\pi(i)<\pi(j),\\
0& \text{ in all other cases.}
\end{array}\right.
\end{equation*}
\end{rem}
\subsubsection{Full-measure sets of {IETs}}\label{sec:aeIET}
We will always assume that the permutation $\pi $ is \emph{irreducible}, i.e.~if the set $\{1,\dots, k\} $ is invariant under $\pi$, then $k=d$ (since this is a necessary condition for any $T = (\pi,\lambda)$ with permutation $\pi$ to be minimal).
We denote by $\mathfrak{S}_d^0$ the set of \emph{irreducible} permutations $\pi$ on $d$ elements.
For {every} number of intervals $d\geq 2$, we denote by $\mathcal{I}_d$ the space of all $d$-IETs on $[0,1]$ with irreducible permutations, which is therefore parameterized by $\mathfrak{S}_d^0\times \Delta_d$, where $
\Delta_d\subset \mathbb{R}^d_{>0}$ is the unit simplex of vectors $(\lambda_1,\dots, \lambda_d)\in \mathbb{R}^d_{>0}$ such that $\sum_{i=1}^d\lambda_i=1$.

We say that a result holds \emph{for almost every IET with permutation} $\pi$ if it holds for almost every choice of the length vector $\lambda$
with respect to the restriction of the Lebesgue measure on $\mathbb{R}^d$ to the simplex $\Delta_d$. Given $d\geq 2$, we say that a property holds for \emph{almost every} IET in $\mathcal{I}_d$ if, for  {every} (irreducible) permutation $\pi\in \mathfrak{S}_d^0$, the property holds for almost every IET with permutation $\pi$. Equivalently, if it holds for a set $\mathcal{F}\subset \mathcal{I}_d$ with $m(\mathcal{I}_d\setminus\mathcal{F})=0$, where $m$ is the product of the counting measure on $\mathfrak{S}_d^0$ and  the Lebesgue measure on $\Delta_d$.


 \subsubsection{Birkhoff sums and special flows.}\label{sec:sf}
Given a  positive, integrable roof function $f: I\to\mathbb{R}_{>0}$,  so that $\inf_{x\in I} f(x)>0$ let us denote by $S_n(f)(x)$ the  \emph{Birkhoff sum} defined by
\[S_n(f)(x)=\left\{
\begin{array}{rcl}
\sum_{0\leq i<n}f(T^ix)&\text{if}& n\geq 0\\
-\sum_{n\leq i<0}f(T^ix)&\text{if}& n< 0.
\end{array}
\right.\]
{The \emph{special flow} over $T: I \to I$ under the roof function $f$ is the vertical unit-speed flow on the suspension space
\[
I^f := \{ (x,y) \in I \times \mathbb{R} : 0 \leq y < f(x) \},
\]
where each point $(x,f(x))$ on the graph of $f$ is identified with $(T(x),0)$, as shown in Figure~\ref{symlog}.}
More precisely,  $(T^f_t)_{t\in\RR}$  acts
on $I^f$
so that
\[T^f_t(x,r)=(T^nx,r+t-S_n(f)(x)),\]
where $n=n(t,x) \in\mathbb{Z}$ is a unique integer number with $S_n(f)(x)\leq r+t<S_{n+1}(f)(x)$.

 \subsubsection{Roofs with logarithmic singularities}\label{sec:roofs}
We consider special flows under a \emph{roof function} chosen in a class of (positive) functions which have \emph{logarithmic singularities} at the discontinuities $\beta_i$. This is the type of singularities that arise in the special flow representation of locally Hamiltonian flows with simple (i.e.~non-degenerate) saddles. 
The prototype of a (right-side) logarithmic singularity at zero is given by {the model function
\begin{equation} \label{eq:positive}
\mathrm{Log}(x): =
 \begin{cases}|\log x |=
-\log (x) & \text{if}\ 0<x\leq 1, \\   0 & \text{if}\ x\leq 0.
\end{cases}
\end{equation}
Notice that the minus sign is added  so that $\textrm{Log} \geq 0$.

Let $T$ be an IET with endpoints $End(T)=\{0=\beta_0<\beta_1<\dots< \beta_d=|I|\}$
 (see ~\S\ref{sec:main_sf}). Let us introduce the auxiliary function $\textrm{Log}_i^+(x):= \textrm{Log} (x-\beta_i)$  (which describes a logarithmic singularity as  $x\to \beta_i^+$, i.e.\ as $x$ approaches $\beta_i$ from the right) for $0\leq i< d$ and $\textrm{Log}_i^-(x):= \textrm{Log}(\beta_i-x)$ for $1\leq i\leq d$ (which describes a left logarithmic singularity as $x\to \beta_i^-$, i.e.\ as $x$ approaches $\beta_i$ from the left).
We say that a function $f $ has \emph{pure logarithmic singularities} at the endpoints $\beta_i$ of $T$ and write  {$f\in \pLog{T}$} (where the dependence on $T$ is only through the location of the endpoints $\beta_i$ of $T$) if $f$  is a positive real--valued function, defined on the disjoint union $\sqcup_{i=0}^{d-1} (\beta_i,\beta_{i+1})$ of the form
\begin{equation}\label{eq:formlogsing}
f(x)=\sum_{0\leq i<d} C_i^+\mathrm{Log}^+_i(x) + \sum_{1\leq i\leq d}
C_{i}^-\mathrm{Log}_{i}^-(x),
\end{equation}
where $C^\pm_i \geq 0$ are non-negative constants,  
not all simultaneously zero. 
We say that $f$ has pure \emph{symmetric} logarithmic singularities if furthermore  
 $\sum_{i=0}^{d-1}C_i^+= \sum_{i=1}^{d}C_i^-$.
{
For brevity of notation, in the symmetric case we will often denote by:
\begin{equation*}\label{def:Csum}
C_{sum}:=\sum_{i=0}^{d-1}C^+_i+\sum_{i=1}^{d}C^-_i.
\end{equation*}}
Finally, we say that $f$ has  \emph{symmetric logarithmic singularities} and write 
 $f \in  \SymLog{T}$ 
if $f$ can be written as $f=f^{\operatorname{p}}+g$ where {$f^{\operatorname{p}}  \in \pSymLog{T}:=\pLog{T}\cap  \SymLog{T}$} has pure logarithmic symmetric singularities and $g=g_f:I\to\RR$ is a function which is absolutely continuous on {every}  interval $(\beta_i,\beta_{i+1})$, $0\leq i<d$. We say that a roof function \(f\in \SymLog{T}\) is of \emph{geometric origin} if there are exactly two zero coefficients $C^+_{i_+}$ and $C^-_{i_-}$ for some $0\leq i_+<d$ and $1\leq i_-\leq d$, and the remaining coefficients \(C^\pm_i\) are positive. It follows from the proof of Theorem~6.1 in \cite{Fr-Ul}, applied to the function identically equal to $1$ on the entire surface, that the roof function arising in the special flow representation of any minimal locally Hamiltonian flow on a surface satisfies this condition.

For technical reasons (since in some parts we will need additional regularity of the roof to work with second derivatives), we will also consider the subset $\SymLogC{T}$ of the set $\SymLog{T}$, such roof functions $f=f^{\operatorname{p}}+g$  that $g$ is extended to a $C^2$-map on {all intervals} $[\beta_i,\beta_{i+1}]$, $0\leq i<d$ and $f''\geq 0$. Clearly, $f^{\operatorname{p}}\in\SymLogC{T}$.

 \begin{figure}[h!]
\includegraphics[width=0.5\textwidth]{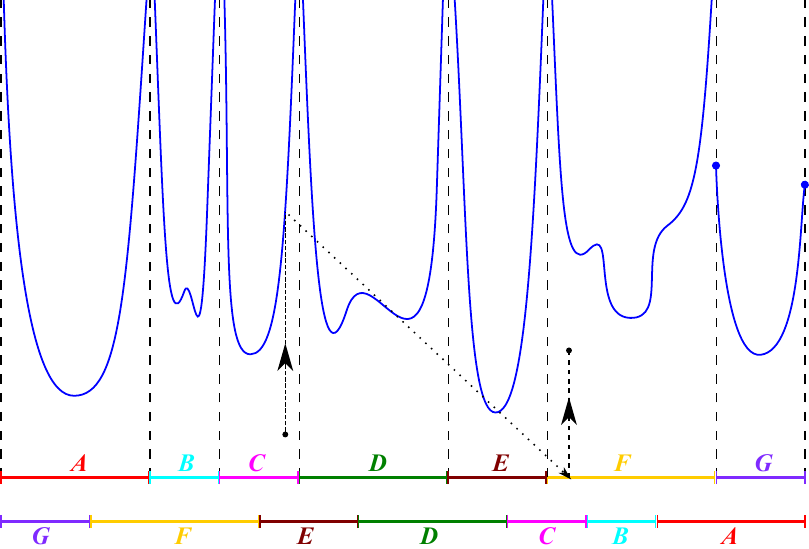}
 \caption{ A special flow over a $7$-IET  under a roof {$f \in \SymLog{T}$} which is a special representation of a locally Hamiltonian flow with two saddles of multiplicity $2$ on a surface of genus $3$.  \label{symlog}}
\end{figure}

\subsection{Spectral disjointness results for special flows}\label{sec:sf_results}
The main spectral result that we prove in the setting of special flows is the following.
{\begin{theorem}\label{thm:dm_sf}
For any $d\geq 2$, let $\pi$ be any irreducible permutation of $d$ symbols. For almost every IET $T$ with permutation $\pi$,
 for {every} $f\in \SymLog{T}$, the special flow $(T_t^f)$ over $T$ under $f$ has singular spectrum and is spectrally disjoint from all mixing flows.
\end{theorem}}

 The main {spectral} disjointness result stated as Theorem~\ref{thm:main_dj} follows, in view of the reduction, from the  following theorem:

{
\begin{theorem}\label{thm:sfdisjointness}
Let $d\geq 2$, and
let $\pi, \tau$ be  irreducible permutations of $d$ symbols.
For almost every IET $T$ with permutation $\pi$,  there exists a full--measure set $\mathcal{D}_T$ of IETs with permutation $\tau$ such that for every $S\in \mathcal{D}_T$ and for every pair of roof functions $f\in \SymLog{T}$ and $g\in \SymLogC{S}$ of geometric origin, the special  flows $(T^f_t)$ and $(S^g_t)$ are spectrally disjoint.
\end{theorem}
}

Notice that the full measure set $\mathcal{D}$ \emph{depends} on the choice of the (typical) IET $T$. The reason for this will be clear from the proof.  We recall furthermore that a similar result on \emph{pairwise disjointness} was proved by Jon Chaika for IETs themselves, namely Chaika in \cite{Cha} proved that for almost every IET $T$ there exists a full measure set of IETs $S$ 
such that $T$ and $S$ are spectrally disjoint.  Thus, our result can be seen as a strengthening of Chaika's work, in the same way that the (typical) absence of mixing for flows with symmetric logarithmic singularities can be seen as a far-reaching strengthening of the absence of mixing for IETs proved by Katok \cite{Ka:int}.

\subsection{Reduction to the special flows results}\label{sec:reduction}
It is well known that locally Hamiltonian flows with only simple saddles can be represented as special flows over IETs with symmetric logarithmic singularities. We here briefly explain how, in view of this  representation, Theorems~\ref{thm:singsp} and~\ref{thm:mixing_disj} follow from Theorem~\ref{thm:dm_sf}, and how  Theorem~\ref{thm:main_dj} follows from Theorem~\ref{thm:sfdisjointness}. 
Since the reduction of results on locally Hamiltonian flows to these special flows is by now classical, we provide only a sketch and refer  the interested reader for example to \cite{Fr-Ul,Rav:mix} for details.

\begin{proof}[Proof of Theorem~\ref{thm:singsp} and Theorem~\ref{thm:mixing_disj}]
Let $(\varphi_t)_{t\in \mathbb{R}}$ be any locally Hamiltonian flow with only  simple  saddles on a surface $M$ of genus $g$. Then,  {considering a Poincar{\'e} section}, one can show that $(\varphi_t)_{t\in \mathbb{R}}$  is metrically isomorphic to a special flow over an interval exchange $T$ (if the parameterization of the  Poincar{\'e} section is standard), see for example \cite[Section 4.4]{Yo06}.
Thus, {both flows have} the same ergodic and spectral properties.
By Theorem~\ref{thm:dm_sf}, for almost every choice of the lengths $\lambda_i:= \beta_{i}-\beta_{i-1}$, $1\leq i\leq d$,  of $T$, such special flow has purely singular spectrum and is spectrally disjoint from all mixing flows. By the definition of the Katok fundamental class and a Fubini argument, this gives a  full measure   set of locally Hamiltonian flows with simple saddles  (with respect to the Katok measure class, see \S~\ref{sec:measure}) with  singular spectrum and with spectral disjointness from all mixing flows.
\end{proof}
\noindent The other reduction, that we describe now, is very similar, {but requires also some attention to the  \emph{measurability} of the set of \emph{pairs} of spectrally disjoint locally Hamiltonian flows, which follows by a Fubini argument from  results proved in Appendix~\ref{sec:meas}.}

\begin{proof}[Proof of Theorem~\ref{thm:main_dj}.]
{ Let us recall from \S~\ref{sec:disjointness thm3} that the notion of \emph{typical pair} of flows in $\mathcal{U}_{min}(M, \omega, \Sigma)\times \mathcal{U}_{min}(M, \omega, \Sigma)$ is defined in terms of the measure class $(Per\times Per)_\ast (Leb\times Leb)$.  Consider the set} {  $\underline{\mathcal{SD}}\subset \mathbb{R}^n\times \mathbb{R}^n$ of all pairs $(\bar{v},\bar{w})\in \mathbb{R}^n\times \mathbb{R}^n$  {of periods} such that \emph{for every pair} of flows $((\varphi_t)_{t\in\R},(\psi_t)_{t\in\R})\in (Per\times Per)^{-1}(\{(\bar{v},\bar{w})\})$ the flows $(\varphi_t)_{t\in\R}$ and  $(\psi_t)_{t\in\R}$  are spectrally disjoint.
{It is proved in Appendix~\ref{sec:meas} (see Lemma~\ref{lem:measSD} in  \S~\ref{sec:measurability}) that the set $\underline{\mathcal{SD}}$ is Lebesgue measurable. }
Hence, to prove that the set of spectrally disjoint flows in   $\mathcal{U}_{min}(M, \omega, \Sigma)\times \mathcal{U}_{min}(M, \omega, \Sigma)$  has full measure, { one has only to prove that the set $\underline{\mathcal{SD}}$}  has full Lebesgue measure in $\mathbb{R}^n\times \mathbb{R}^n$.  Using a standard Fubini argument, it is enough to prove that
\begin{gather}\label{Fubini}
\begin{aligned}
\text{there exists a full Lebesgue measure set $V\subset \R^n$ such that}\qquad\qquad\quad\\
\text{for every $\bar{v}\in V$ there exists a full Lebesgue measure set $W_{\bar{v}}\subset \R^n$ such that}\\
\bar{w}\in W_{\bar{v}}\Longrightarrow (\bar{v},\bar{w})\in\underline{\mathcal{SD}}.\qquad\qquad\qquad\qquad\qquad\qquad
\end{aligned}
\end{gather}}
We will now use the reduction and Theorem~\ref{thm:sfdisjointness} to  show that \eqref{Fubini} holds. Given two locally Hamiltonian flows  $(\varphi_t)_{t\in \mathbb{R}}$ and $(\phi_t)_{t\in \mathbb{R}}$ with only  simple  saddles,
consider a representation of each as special flows, so that $(\varphi_t)_{t\in \mathbb{R}}$ is isomorphic to the special flow $T^f_t$ over an IET $T$ and under $f\in \SymLog{T}$, and  $(\phi_t)_{t\in \mathbb{R}}$ is isomorphic to $S^h_t$, where $h\in \SymLog{S}$. In fact, with some additional arguments, we can also show that for almost every $(\phi_t)_{t\in \mathbb{R}}$ we can find its special representation $S^h_t$ for which $h\in \SymLogC{S}$. Indeed, for any standard choice of the  Poincar{\'e} section $J$, the Poincar{\'e} map $S:J\to J$ is an IET and the roof function $h:J\to \mathbb{R}_{>0}$ is in $\SymLog{S}$ with $h=h^{\operatorname{p}}+g_{h}$. Moreover, for a.e.\ IET $S$, by Theorem~6.1 in \cite{Fr-Ul} and Theorem~8.9 in \cite{Fr-Ul24}, $g_h:J\to\mathbb{R}$ is cohomologous to a piecewise linear map $\widetilde g_h:J\to\mathbb{R}$, linear over all intervals exchanged by $S$. Let us consider $\widetilde{h}=h^{\operatorname{p}}+\widetilde g_h$ which is cohomologous to $h$ and belongs to $\SymLogC{S}$, as $\widetilde h''= (h^{\operatorname{p}})''\geq 0$. Unfortunately, the map $\widetilde{h}$ may not be positive. To solve this issue, we take advantage of the unique ergodicity for almost every IET $S$. As the integrals of $\widetilde{h}$ and $h$ are the same, so positive, by the unique ergodicity of $S$, we can find a small enough subinterval $\widehat J\subset J$ such that the induced map $S_{\widehat J}:\widehat J\to\widehat J$ is an IET exchanging the same number of intervals as the IET $S$ and the induced cocycle $\widetilde h_{\widehat J}:\widehat J\to\mathbb{R}$ is positive.  Moreover, as $\widetilde{h}\in \SymLogC{S}$ is cohomologous to $h$, we also have that $\widetilde{h}_{\widehat J}\in \SymLogC{S_{\widehat J}}$ is cohomologous to $h_{\widehat J}$. It follows that the special flows $(S_{\widehat J})^{\widetilde h_{\widehat J}}$ and $(S_{\widehat J})^{h_{\widehat J}}$ are isomorphic. As $(S_{\widehat J})^{h_{\widehat J}}$ is a special representation of the locally Hamiltonian flow $(\phi_t)_{t\in \mathbb{R}}$ given by a Poincar{\'e} section $\widehat J\subset J$, also the flow $(S_{\widehat J})^{\widetilde h_{\widehat J}}$ is a special representation of  $(\phi_t)_{t\in \mathbb{R}}$, for which the roof function $\widetilde{h}_{\widehat J}$ belongs as desired to  $\SymLogC{S_{\widehat J}}$ and is of geometric origin.

{Since spectral disjointness is an isomorphism invariant, in view of Theorem~\ref{thm:sfdisjointness}, we have that for almost every IET $T$ there exists a full--measure set $\mathcal{D}_T$ of IETs such that for every $S\in \mathcal{D}_T$ and for all roof functions $f\in \SymLog{T}$ and $g\in \SymLog{S}$ of geometric origin, the special  flows $(T^f_t)$ and $(S^g_t)$ are spectrally disjoint.
This gives \eqref{Fubini}, and, {in view of the arguments presented at the beginning, concludes the proof.}
}\end{proof}
\noindent  {In view of the reduction in this Section, the remainder of the paper will be devoted to the proof of Theorems~\ref{thm:dm_sf} and \ref{thm:sfdisjointness}.}

\section{Strategy and outline}\label{sec:outline}
In this section we explain  the main ideas and techniques used in the proofs. First, in \S~\ref{sec:strategy}, we sketch the ideas behind the proof of singularity of the spectrum {(Theorem~\ref{thm:dm_sf})},  highlighting the difficulties and differences with respect to the genus-one and genus-two cases (see \S~\ref{sec:g1g2}). In \S~\ref{sec:strategy2} we then explain the strategy for proving {spectral} disjointness of typical pairs, { i.e.~Theorem~\ref{thm:sfdisjointness}}. We then provide in \S~\ref{sec:organization} a reading guide to the rest of the paper.


\subsection{Strategy of the proof of singularity.}\label{sec:strategy}
{The proof of Theorem~\ref{thm:dm_sf}, i.e.~of the main result on singularity for special flows (as well as the proof of}  Theorem~\ref{thm:sfdisjointness}} on  disjointness of spectra of special flows, discussed in \S~\ref{sec:strategy2}) is based on a careful analysis of \emph{limit measures} arising from the \emph{distribution of Birkhoff sums} of the roof function, { as we now explain}. 

\subsubsection{Role of Birkhoff sums}\label{sec:BS}
Consider the special flow representation of a minimal locally Hamiltonian flow with only simple saddles as a special flow over an IET $T$ under a roof $f\in \SymLog{T}$ introduced in \S~\ref{sec:main_sf}.
Mixing as well as spectral results for locally Hamiltonian flows can be deduced from a fine understanding of the distribution of Birkhoff sums $S_n(f)= \sum_{k=0}^{n-1} f\circ T^k$ of the roof function $f$.  It is well known (see for example the surveys \cite{Ul:slo} or \cite{Fr-Ul:Enc}) that the growth of these Birkhoff sums provides qualitative information on the speed of \emph{shearing} of transverse arcs, which is the key phenomenon which produces mixing in these type of  flows (see e.g.~the heuristic explanation in \cite{Ul:slo}).
{ It is also well known that}      to prove  \emph{absence} of mixing { for special flows representing locally Hamiltonian flows} it is sufficient to prove \emph{lack of growth}, in the form of   \emph{tightness}\footnote{Let us recall that a sequence of real-valued random variables  $(X_n)_n$ on a common probability space  $(\Omega,\mathcal{A}, \mathcal{P})$ is called \emph{tight} if for every $\epsilon >0$ there exists $M>0$ such that $\mathcal{P}(|X_n|<M)\geq 1-\epsilon$ for all $n\in \N$. Thus, there is no escape of mass to infinity for distributions of random variables $(X_n)_n$. Here, the definition is applied to the Birkhoff sums $S_n f$ (after some centering) considered as random variables  with $x$ taken at random with respect to the Lebesgue measure on $[0,1]$.} of Birkhoff sums along a sequence of \emph{rigid} or \emph{partially rigid} times $(h_n)_{n\in \mathbb{N}}$ in the base. {Indeed,} tightness along partially rigid times is the key ingredient
for many criteria for absence of mixing, starting from  Katok~\cite{Ka:int} and Ko\v{c}ergin~\cite{Ko:abs} seminal works, up to later higher genus results \cite{Sch:abs, Ul:abs}.
The study of spectral questions is also based on the analysis of Birkhoff sums, but
requires  a {much} more delicate understanding of \emph{weak limits} of Birkhoff sums, as we now explain.

\subsubsection{The {spectral} singularity criterion}
The prototype for the criterion for singular spectrum that we prove in this paper is provided by the work \cite{Fr-Le0} of Lema\'nczyk with the first author,
who showed that if, in addition to \emph{tightness}, one can also control the \emph{tails} of the (centred) distributions of the
Birkhoff sums $S_{h_n}f$ at special times $(h_n)_n$, 
one can then deduce (using technical tools such as joinings and Markov operators) spectral disjointness from mixing flows. 
The  precise \emph{criterion}\footnote{The criterion, which follows essentially by the work \cite{Fr-Le0} where a slightly weaker criterion was proved,
 was already
 stated and proved in \cite{Ch-Fr-Ka-Ul}, where it was used to prove singularity of the spectrum for
   typical flows in $\mathcal{U}_{min}$  in the special case of surfaces of $g=2$ and flows with isomorphic saddles. } that we use for proving \emph{singularity of the spectrum} of special flows  (stated in \S~\ref{sec:sc}, see Theorem~\ref{thm:singcrit})
{states  that if} one can find a sequence of \emph{rigidity times} $(h_n)_n$ (so that in particular $T^{h_n}$ converges in measure to identity as $n\to \infty$, see the beginning of \S~\ref{sec:sc} for the precise definition) and
 centralizing constants $(c_n)_{n\in\mathbb{N}}$ in $\mathbb{R}$ such that the sequence of centred Birkhoff sums
 $(S_{h_n} f-c_n)_{n\in \mathbb{N}}$
 has \emph{exponential tails} (in the sense of \eqref{eq:expdecay} in the statement of Theorem~\ref{thm:singcrit}), then the flow $(T^f_{t})_{t\in\RR}$ has singular spectrum.

One can think of this criterion
 as a \emph{strengthening} of the criteria for absence of mixing  via tightness; 
with the additional information of exponential tails being the extra ingredient needed
 to prove singularity. It is important to remark though that,
while for
absence of mixing tightness is only needed at \emph{partial rigidity}  \emph{times} (i.e.\ along a sequence $(h_n)_n$ of times such that $T^{h_n}$ converges to identity (see the definition at the beginning of \S~\ref{sec:sc}) on subsets $E_n $ of measure 
 bounded below), 
  for the spectral conclusion  it is \emph{crucial}
  that the times $(h_n)$ are \emph{global} rigidity times (i.e.\ the measure of the sets $E_n$ tends to $1$ as $n$ grows). { This requirement is one of the key difficulties that we had to overcome in our proof.}



\subsubsection{Rigidity and tightness mechanisms}\label{sec:mechanisms}
Rigidity times are easy to construct for IETs: the existence and abundance of rigidity times were for example {shown} in the seminal work by Veech {\cite{Ve:spe}}, in which the renormalization algorithm known as  Rauzy-Veech induction (see \S~\ref{sec:noninvertibleRV}) plays a key role. This algorithm allows to produce
 a sequence of representations of an IET of $d\geq 2$ intervals as $d$ Rohlin-towers (see  Definition~\ref{def:tower} for the notion of Rohlin tower by intervals). Using ergodicity of the algorithm, one can show for example that for a.e.~IET one of the towers has area that tends to $1$ and gives a rigidity set of measure tending to $1$.

\emph{Tightness} of Birkhoff sums along  \emph{partial} rigidity sets
(for special flows over a.e.~$T$ with any $d\geq 2$,  under any $f\in \SymLog{T}$) 
was proved by the last author in   \cite{Ul:abs} (to show absence of mixing for flows in $\mathcal{U}_{min}$). The key technical result in   \cite{Ul:abs} is  control of the Birkhoff sums $S_{h_n}(f')$ of the \emph{derivative}\footnote{In  \cite{Ul:abs} it is shown (exploiting \emph{cancellations} between positive and negative contributions), that the \emph{trimmed} Birkhoff sums  $\widetilde{S}_{h_n}(f')$ obtained removing from  $S_{h_n}(f')$  the contributions of the closest visits to each singularity (see \S~\ref{sec:trimmedBS} for definitions) are controlled on \emph{each} tower corresponding to good times
  given by Rauzy-Veech induction (special times at which one has  delicate control of the equidistributions of orbits).} $f'$ of the roof (which is a non-integrable function, with singularities of type $1/x$) after \emph{trimming}, i.e.~removing the contribution of the closest visits to the singularities.
  The type of bounds proven, that we call \emph{linear bounds on trimmed derivatives} (see Definition~\ref{def:bounded}  in \S~\ref{sec:linearbounds}
  and Proposition~\ref{boundSf'growthprop}) give, \emph{after centering},  tightness of Birkhoff sums on each tower separately (since here the centering constant \emph{depends} on the tower);  moreover, the contribution of the closest visits can be used to produce  exponential control of the tails of Birkhoff sums (see \S~\ref{sec:sc via Bs}).
  Unfortunately (since the towers used to control Birkhoff sums have measure bounded away from $1$), this does \emph{not} give the desired tightness and exponential tails on a set of large measure, but only on each tower separately\footnote{Essentially this happens because the centering constants also grow and display \emph{polynomial deviations},
  so tightness after centering in each tower separately does not imply global tightness with a unique centering constant.}.

\subsubsection{Coexistence of rigidity and tightness  in genus one and  two}\label{sec:g1g2}
To combine the  two mechanisms explained in the above \S~\ref{sec:mechanisms} for rigidity on one side and for tightness and exponential tails on the other  (and thus to prove the assumptions of the criterion for singular spectrum), one would like to find a sequence of times where there is a \emph{large} Rohlin tower (given by Rauzy-Veech induction) of measure tending to $1$, on which the linear bounds on trimmed derivatives hold.  We refer to these times {as} \emph{coexistence} times, since they are times in which rigidity {\emph{coexists}} with tightness (and exponential tails). For special flows over {rotations} as well as IETs which correspond to flows on surfaces of genus $g=2$, tightness can be proved by ad-hoc mechanisms 
are \emph{compatible} with rigidity { and we now explain}.

In  the case of  \emph{irrational rotations},  tightness for  functions of bounded variation is provided by the classical \emph{Denjoy-Koksma inequality}\footnote{The Denjoy-Koksma inequality implies that Birkhoff sums of functions of bounded variation are (uniformly) bounded at the rigidity times (more precisely along the sequence $(q_n)_{n\in \mathbb{N}}$, where $q_n$ are the denominators of the continued fraction expansion of the rotation number).}.   Derivatives of functions with symmetric logarithmic singularities have \emph{unbounded} variation, but a truncation and approximation argument can be combined with Denjoy-Koksma to prove tightness and linear bounds on trimmed derivatives on the whole space (see for example \cite{BK,Fr-Le0}).
We remark that the Denjoy-Koksma inequality \emph{fails} for IETs with $d\geq 4$ and this difference  is at the heart of the many drastic differences of behaviour\footnote{The absence of Denjoy-Koksma inequality is indeed related to the obstructions to solve the cohomological equation \cite{Fo:coh,Fr-Ki2,Fr-Ki3}, 
to the phenomenon of polynomial deviations of ergodic averages for IETs \cite{Zo:dev, Fo:dev} and, more generally, to the lack of a KAM theory for non-linear IETs in higher genus (see e.g.~\cite{MMY} or \cite{SU} and the references therein).} between rotations and IETs in higher genus.

When a locally Hamiltonian flow on a surface of genus two has two simple isomorphic saddles, the underlying  surface has an intrinsic \emph{symmetry}\footnote{Each minimal locally Hamiltonian flow can be seen as a time-change of a linear flow on a translation surface; here by underlying surface we mean this translation surface. In the case of a flow on a surface of genus two with two simple saddles, the translation surface belongs to the stratum $\mathcal{H}(1,1)$ and has a hyperelliptic involution.}, the so-called \emph{hyperelliptic involution}, which provides a simpler mechanism\footnote{The geometric mechanism to prove  linear bounds on trimmed derivatives in presence of this involution exploits the remark that for any symmetric (of equal backward and forward length) trajectory from a fixed point of the hyperelliptic involution, there are perfect cancellations for Birkhoff sums of the derivative of the roof function, see \cite{Ch-Fr-Ka-Ul} for details.} for proving tightness.  The action of this involution is  implicitly\footnote{In  Scheglov's work \cite{Sch:abs}, cancellations for the (trimmed) Birkhoff sums of the  derivative were proved through a careful combinatorial analysis of the substitutions arising from the action of Rauzy-Veech induction on symmetric permutations; implicitly, though, the combinatorial properties exploited by Scheglov (in particular the palindrome nature of some of the words) are a combinatorial trace of the action of the hyperelliptic involution.} also behind the first proof of absence of mixing in genus two by Scheglov \cite{Sch:abs} (shortly before \cite{Ul:abs}). This mechanism is compatible with a  geometric form of rigidity, which exploits cylinders on the translation surface rather than Rohlin towers (see \cite{Ch-Fr-Ka-Ul} for details), and was exploited by the authors together with Chaika in \cite{Ch-Fr-Ka-Ul} to prove singularity of the spectrum
 in the genus two special case; unfortunately locally Hamiltonian flows with only simple saddles (i.e.~in $\mathcal{U}_{min}$) have an underlying hyperelliptic symmetry \emph{only} in genus two\footnote{Translation surfaces with a hyperelliptic involution (or more precisely \emph{hyperelliptic strata} of translation surfaces) exist in any genus $g\geq 1$. The results on absence of mixing \cite{Sch:abs} or singularity of the spectrum \cite{Ch-Fr-Ka-Ul} at the level of \emph{special flows} can be proved for a.e.~IET $T$ of any $d\geq 2$ and any roof $f\in \SymLog{T}$, but unfortunately these abstract special flows represent a locally Hamiltonian flow only when the IET exchanges $d=5$ intervals. Indeed, for $g\geq 3$, IETs that belong to hyperelliptic components appear only when the singularities are degenerate and hence give rise to power singularities of the roof, rather than logarithmic singularities.}.

\subsubsection{Mechanism for coexistence in $g\geq 2$}\label{sec:mechanismhigherg}
To generalize the result to higher genus, in this paper we use Rauzy-Veech induction. 
The times at which linear bounds on trimmed derivatives hold are constructed in \cite{Ul:abs} by considering a (balanced) \emph{acceleration} of Rauzy-Veech induction (see \S~\ref{sec:derivativesviaRV}, in particular Proposition~\ref{prop:accelerationset}).
Using Rauzy-Veech induction,
 one can  construct, for {every} $\epsilon>0$, an open set of IETs which {contains} a rigidity tower (see Definition~\ref{def:tower}) which occupies  a proportion $1-\epsilon$ of space.
  Coexistence of these two behaviours  is made possible by the local product structure and the powerful machinery provided by Rauzy-Veech induction: an essential property of (natural extension of) Rauzy-Veech induction which we crucially exploit is that it has a \emph{local product structure}, i.e.~the past is in a sense independent of the future; more precisely, we exploit the coordinates $(\pi, \lambda, \tau)$ introduced by Veech and the independence of the \emph{future} Rauzy-Veech induction iterates from $\tau$ and of the \emph{past} ones from $\lambda$, see \S~\ref{sec:statisticalRV}. Thus, $\epsilon$-rigidity can be produced by controlling only $\lambda$, while the times at which the linear bounds on trimmed derivatives hold can be shown to depend on $\tau$ only (this is explained in the Appendix~\ref{sec:app}, by revisiting the proof of \cite{Ul:abs}).  
  Then, times where a fixed $\epsilon$-rigidity coexists with trimmed derivative bounds can be deduced from ergodicity of Rauzy-Veech induction by considering visits to product sets; taking a sequence $\epsilon_k$ tending to zero with $k$, we then get the desired coexistence times.

{
\subsubsection{Summary of the strategy to prove singularity.}\label{sec:stratsing}
Thus, to summarize the strategy,
the main singularity result (Theorem~\ref{thm:singsp}) follows directly from the singularity criterion (Theorem~\ref{thm:singcrit})  developed from the work \cite{Fr-Le0} by Lema\'nczyk and the second author. To match the requirements of the singularity criterion one needs \emph{coexistence} of global
rigidity times together with tightness and exponential tail estimates for the centered
Birkhoff sums.
 Tightness and exponential tails of centered Birkhoff sums are obtained  by arguments in the spirit of the last author's work  \cite{Ul:abs}, which exploit
balanced times of Rauzy-Veech induction and linear bounds on trimmed derivatives, which are shown to depend on the \emph{past} of
the Rauzy--Veech induction. The additional requirement that a single rigidity tower occupies
almost the entire space is instead controlled by the \emph{future} of Rauzy-Veech induction, given  by the length parameters only. Thus,  the
key novelty in our paper is perhaps the production of times where these two features coexist, using the local
product structure of Rauzy-Veech induction.}

\subsection{Strategy of the proof of the {spectral} disjointness result}\label{sec:strategy2}
Let us now present the main ideas behind the proof
that the special flow representations of a typical pair of locally Hamiltonian flows is spectrally disjoint,  { i.e.~of  Theorem~\ref{thm:sfdisjointness}}.

\subsubsection{{Spectral} disjointness mechanism}\label{sec:disjointness strat}
 To prove Theorem~\ref{thm:main_dj}, we prove and exploit a {spectral} disjointness criterion (see Theorem~\ref{thm:spectorth}) that roughly says that {spectral} disjointness of the two flows can be proved if one finds a sequence of (common) times which have very different dynamical behaviour for the two flows. More precisely, for the first flow, we exploit the rigid times with exponential tails produced to prove singularity of the spectrum. To show {spectral} disjointness, it then suffices to find such times for the first flow, that are also \emph{mixing} times for the second flow (see Theorem~\ref{thm:spectorth} for the statement of the criterion).
 The proof of the criterion that we use to deal with the exponential tails (presented in \S~\ref{sec:disjointness_criterium}) exploits the study of weak limits of Birkhoff sums distributions (using Prokhorov compactness and Markov operators)  and follows closely techniques previously developed by Lema\'nczyk and the first author, see \cite{Fr-Le0,Fr-Le1,Fr-Le2}.

This way to proving {spectral} disjointness is reminiscent of the mechanism exploited by J.~Chaika in \cite{Cha} to show that almost every pair of interval exchange {transformations, the two transformations are disjoint.} In that case, Chaika shows that one can find rigid times for one IET that are mixing times for the other IETs (whose existence is guaranteed by the fact that IETs are typically weakly mixing, see \cite{AF:wea}, i.e.~mixing along a density one set of times). We stress though that while the strategy of Chaika is to find rigidity times in (any) \emph{given} sequence of mixing times, we exploit a reversed strategy, starting from a given sequence of rigid times (with exponential tails) for the first flow (a type of behaviour that is \emph{rare} and delicate to locate, as explained in the previous sections), and then finding among those a subsequence of mixing times for the second flow.

\subsubsection{Mixing of resonant times}\label{sec:mixingstrategy}
Since a typical locally Hamiltonian flow  is weakly mixing, although not mixing (as recalled in \S~\ref{sec:classes}),  
mixing times are abundant.
Since we need to find mixing times for the second flow that are a \emph{subsequence} of the given sequence of rigid times with exponential tails for the first flow (as explained in \S~\ref{sec:disjointness strat}), however we need to  be able to control the \emph{location} of mixing times.

For this, we use a novel way to produce shearing (and then mixing times), which is based on what we call \emph{resonant rigidity}.
The basic mechanism that we use to produce mixing, as for many classes of entropy zero flows, is the so-called \emph{mixing via shearing}\footnote{Mixing in several classes of entropy zero flows is produced through this mechanism:  mixing follows by showing that large families of transversal arcs \emph{shear} under the flow, i.e.~become more and more tilted in the flow direction and shadow a flow trajectory, which is typically uniformly distributed. {This mechanism} has been used since the work of Ko\v{c}ergin \cite{Ko:mix} in the context of surface flows  (see \cite{SK:mix, Ko:mix,  Ul:mix, Rav:mix}), as well as for several other entropy zero flows: for {horocycle flows}, shearing was used in seminal work by Marcus \cite{Marcus}; for time-changes of horocycle flows, by Forni and the last author in \cite{FU}; analytic time-changes of flows on $\mathbb{T}^3$ were built using shearing by Fayad in \cite{Fa:ana}; mixing via shearing was also used to prove mixing of time-changes of nilpotent flows, see \cite{AFU, AFRU, Rav2}.}.
In the case of a roof with asymmetric logarithmic singularities, the  shearing required for this mixing mechanism is created by the asymmetry of the roof (see \cite{SK:mix, Ko:03, Ko:04, Ko:04'} for rotations and in  \cite{Ul:mix, Rav:mix} for IETs). For IETs, equidistribution of orbits is controlled through \emph{balanced} renormalization (Rauzy-Veech induction) times to then prove the shearing estimates. In our setting, the logarithmic singularities are symmetric and times with very well controlled equidistribution of orbits are \emph{not} mixing (as shown in \cite{Ul:abs}).

The mixing times that we use  in this paper are close to \emph{multiples} of a \emph{rigid} time,  that we call  \emph{resonant} rigid times (see  Definition~\ref{def:resonantrigidity_t}).
Resonance produces many close visits to the singularities and hence amplifies their contribution, producing a shearing which comes from the asymmetry of the orbit visits rather than the roof\footnote{Another example in which mixing (at all times) is deduced in \cite{CW}, where shearing is also a consequence of the non-uniformity of the orbits distribution. In \cite{CW} the required estimates are more delicate, so can only be proved for very symmetric IETs which are cyclic covers of rotations. The mechanism used here, instead, is very simple and robust and  simply exploits resonance of a large rigid tower.}, and hence prove mixing (see Proposition~\ref{prop:mix'} and its proof in \S~\ref{sec:orthogonalityproof} for details).
We believe that this is a new shearing mechanism, which may prove useful also in different settings.


\subsubsection{Localization of resonant rigid times}
Resonant rigidity times for the second flow for which one can produce shearing (and hence mixing) as just described (in the above \S~\ref{sec:mixingstrategy}) 
  turn out to be \emph{frequent}; in particular, for a full measure set of IETs, it is possible to control where they occur, so that they can be chosen to coincide (infinitely often) with a given sequence given by the first flow (this is  the content of Proposition~\ref{prop:3}).
 The proof of this result exploits heavily Rauzy-Veech induction and the probabilistic framework provided by the Markovian coding of the induction.  Inspiration for this part is taken from the work of Chaika \cite{Cha}, as well as Avila-Gou\"ezel-Yoccoz \cite{AGY} for the Markovian formalism and distortion control. The crucial ingredients are the distortion estimate proved by Kerckhoff's lemma and the Chung-Erd{\"o}s inequality, which allow us to prove a zero-one law and deduce full measure of the desired set of IETs (for which we can synchronize resonant rigidity times with given times).

\subsection{Organization of the rest of the paper}\label{sec:organization}
{Let us briefly describe the content of the following Sections and Appendixes.
In Section~\ref{sec:singcrit}
we state and prove two criteria: one for {spectral} singularity (Theorem~\ref{thm:singcrit}) and one for {spectral} disjointness (Theorem~\ref{thm:spectorth}).

In Section~\ref{sec:singularity}, we prove  singularity of the spectrum using the {spectral} singularity criterion, assuming the existence of coexistence times at which one has rigidity and linear bounds on the derivative of trimmed Birkhoff sums (i.e.~assuming Proposition~\ref{prop:cohexistence}). In Section~\ref{sec:orthogonalityproof}, we exploit the {spectral} disjointness criterion to deduce  the {spectral} disjointness result (Theorem~\ref{thm:sfdisjointness}) from mixing along resonant rigid times (Proposition~\ref{prop:mix'}) and  from a result which allows us to \emph{synchronize}  resonant mixing times with a given sequence for a full measure set of IETs (Proposition~\ref{prop:3}). In the same Section~\ref{sec:orthogonalityproof},  we also prove Proposition~\ref{prop:mix'}, which gives mixing along resonant rigid times.

The remaining three sections are devoted to the two Propositions left to prove, namely  Proposition~\ref{prop:cohexistence} and Proposition~\ref{prop:3}. To prove these last two crucial technical  results, we need to introduce Rauzy-Veech induction and its natural extension, and recall several technical results related to the induction. We do this in Section~\ref{sec:backgroundRV}. We then prove Proposition~\ref{prop:cohexistence}, which provides  rigid times with trimmed derivative bounds, in Section~\ref{sec:coexistence}. Proposition~\ref{prop:3}, which allows us to locate resonant rigid times, is proved in Section~\ref{sec:resontanttimes2}.




In the four Appendices, we explain how some results used in the previous sections can be deduced from existing works. In Appendix~\ref{app:deviations}, we deduce a result on deviations of ergodic integrals for locally Hamiltonian flows from recent work of the first and last authors \cite{Fr-Ul24}. In Appendix~\ref{app:Roth}, we derive the growth condition for Roth-type IETs from the equivalent characterization proved by Marmi--Moussa--Yoccoz in \cite{MMY}. In Appendix~\ref{sec:app}, we prove a technical strengthening of a result of the last author \cite{Ul:abs}. The final Appendix~\ref{sec:meas} is devoted to the issue of measurability of the set of pairs of spectrally disjoint flows.}


\section{Criteria for {spectral} singularity and {spectral} disjointness}\label{sec:singcrit}
In this section we present criteria
that we will use to prove the {spectral} singularity as well as the {spectral} disjointness results. In \S~\ref{sec:sc}  we  state the criterion for {spectral} singularity  (Theorem~\ref{thm:singcrit}), while in \S~\ref{sec:disjointness_criterium}   
we state and prove the criterion for spectral disjointness (Theorem~\ref{thm:spectorth}).


\subsection{{Spectral} singularity criterion for special flows}\label{sec:sc}
 In this section we present a criterion  that guarantees that $(T^f_t)_{t\in\RR}$ has singular spectrum,  based on rigidity of the base $T:X\to X$ combined with exponential tails of the Birkhoff sums for the roof function $f:X\to\mathbb{R}_{>0}$.
As explained in \S~\ref{sec:strategy} of the introduction, the criterion  stems from the work of Lema\'nczyk and the first author \cite{Fr-Le0,Fr-Le1}, where it was formulated  in a slightly less general form.

 \smallskip
Let  us recall  first the notion of \emph{rigidity}:
\begin{definition}[rigidity]
An automorphism $T$ of a probability Borel space $(X,\mathcal{B},\mu)$ is called \emph{rigid} if there exists an increasing sequence $(h_n)_{n\in\N}$
of natural numbers  such that
\[\lim_{n\to+\infty}\mu(A\triangle T^{h_n}A)=0\quad\text{for every}\quad A\in\mathcal{B}.\]
The sequence $(h_n)_{n\in\N}$ is then a \emph{rigidity sequence} for $T$.
\end{definition}


We can now formulate the {spectral} singularity criterion for special flows:
\begin{theorem}[{spectral} singularity Criterion via rigidity and exponential tails, see Theorem~3.1 in \cite{Ch-Fr-Ka-Ul}]\label{thm:singcrit}
Let $f:X\to\mathbb{R}_{>0}$ be an integrable roof function with $\inf_{x\in X}f(x)>0$.
Suppose that there exist a {rigidity sequence} $(h_n)_{n\in\N}$ for $T$, a sequence $(A_n)_{n\in\N}$  of Borel sets with
$$\mu(A_n)\to 1\ \text{as} \ n\to+\infty,$$
 and a  sequence of real numbers $(a_n)_{n\in\N}$ (centralizing constants)  such that the sequence of centered Birkhoff sums
 $$(S_{h_n}(f)(x)-a_n)_{n\in \mathbb{N}}$$
 has exponential tails, i.e.\
 there exist two positive constants $C$ and $b$ such that
\begin{equation}\label{eq:expdecay}
\mu(\{x\in A_n:|S_{h_n}(f)(x)-a_n|\geq t\})\leq Ce^{-bt}\text{ for all }t\geq0\text{ and }n\in\N.
\end{equation}
Then the flow $(T^f_{t})_{t\in\RR}$ has singular spectrum  and is spectrally disjoint  from all mixing flows.
\end{theorem}
Both conclusions of this criterion were already proved in \cite{Ch-Fr-Ka-Ul}, see \cite[Theorem~3.1]{Ch-Fr-Ka-Ul},  for the
 spectral singularity,  and \cite[Remark~3.2]{Ch-Fr-Ka-Ul},   for the spectral {disjointness} to mixing flows.
The proof 
given in  \cite{Ch-Fr-Ka-Ul} in turn uses as a tool a result from \cite{Fr-Le2}, which is a version of Prokhorov weak compactness of tight sequences along rigidity sets.
We remark that
the \emph{exponential tails} assumption, i.e.~\eqref{eq:expdecay}, along rigidity sets implies in particular that the sequence of centred Birkhoff sums  $(S_{h_n}(f)(x)-a_n)_{n\in \mathbb{N}}$ is \emph{tight}. 


\subsection{Spectral {disjointness} criterion for special flows}\label{sec:disjointness_criterium}
We now state the spectral {disjointness} criterion that we will use to prove spectral disjointness: this follows from the existence
of a sequence of times where one flow is \emph{rigid with exponential tails} and the other \emph{mixing}.
Mixing \emph{along a subsequence} is defined as follows:

\begin{definition}[mixing sequence]\label{def:mixingseq}
A measure-preserving flow  $(U_t)_{t \in \mathbb{R}}$ on $(Y,\nu)$ is \emph{mixing along a sequence} $(a_n)_{n\geq 1}$ of real numbers converging to $+\infty$ if
\[\lim_{n\to\infty}\langle U_{a_n}g_1,g_2\rangle=\langle g_1,1\rangle\langle 1,g_2\rangle\quad\text{ for all }\quad g_1,g_2\in {L^2}(Y,\nu),\]
\end{definition}
Equivalently, the sequence of Koopman operators associated to $U_{a_n}$
 converges to the zero operator in the weak operator topology on $L^2_0(Y,\nu)$.
The criterion for {spectral} disjointness that we use is the following.

\begin{theorem}\label{thm:spectorth}
Assume that $(T^f_t)_{t\in\R}$ on $(X^f,\mu^f)$ is a special flow as in the assumptions of Theorem~\ref{thm:singcrit},
such that there exist 
 $(A_n)_{n\in\N}$ of Borel sets with $\mu(A_n)\to 1$ as $n\to+\infty$,
and  $(a_n)_{n\in\N}$ 
such that $(S_{h_n}(f)(x)-a_n)_{n\in \mathbb{N}}$ has exponential tails. 

If $(U_t)_{t\in\R}$ is a measure-preserving flow on $(Y,\nu)$ which is mixing along the sequence $(a_n)_{n\geq 1}$, then the flows $(T^f_{t})_{t\in\R}$ and $(U_t)_{t\in\R}$ are spectrally {disjoint}.
\end{theorem}
\noindent The rest of this section is dedicated to the proof of this criterion.

\begin{proof}[Proof of Theorem~\ref{thm:spectorth}]
 In the proof, we  need the notion of \emph{integral operator}: for every probability Borel measure $P$ on $\R$ denote by $\int_\R T^f_t\,dP(t):L^2(X^f,\mu^f)\to L^2(X^f,\mu^f)$
the operator such that
\[\langle\int_\R T^f_t\,dP(t)(h_1),h_2\rangle=\int_\R\langle T^f_t(h_1),h_2\rangle\,dP(t)\quad \text{ for all }\quad h_1,h_2\in L^2(X^f,\mu^f).\]
From the assumption, we know that the sequence of random variables $S_{h_n}(f)-a_n:A_n\to\R$, $n\geq 1$ is \emph{tight} and has exponential tails.  This, by classical weak-compactness arguments, implies (as shown for example in one of the steps of the proof of Theorem~3.1 in \cite{Ch-Fr-Ka-Ul})
that there  exists a probability Borel measure $P$   (obtained as a weak limit distribution)
such that
\begin{equation}\label{eq:tails}
P((-\infty,-t)\cup(t,+\infty))\leq Ce^{-bt}\quad\text{for all}\quad t\geq 0.
\end{equation}
(i.e.~the limit probability also has exponential tails) and, by Theorem~6 in \cite{Fr-Le2}, we have
\begin{equation}\label{eq:Tfan}
T^f_{a_n}\to \int_\R T^f_{-t}\,dP(t)\quad\text{in the weak operator topology on}\quad L^2(X^f,\mu^f).
\end{equation}
\smallskip
Suppose, contrary to the claimed conclusion, that the flow  $(U_t)_{t\in\R}$ on $(Y,\nu)$, which is mixing along $(a_n)_{n\geq 1}$, is not spectrally {disjoint from} $(T^f_t)_{t\in\R}$. Then there exist non-zero $g_1\in L^2_0(X^f,\mu^f)$ and $g_2\in L^2_0(Y,\nu)$ such that their spectral measures are equal, i.e.\ $\sigma^{T^f}_{g_1}=\sigma^{U}_{g_2}=\sigma\neq 0$. By assumption, $U_{a_n}\to 0$  in the weak operator topology on $L_0^2(Y,\nu)$. It follows that the unitary representation $(V_t)_{t\in \R}$ on $L^2(\R,\sigma)$, defined by \eqref{def:Vt}, satisfies
\begin{equation}\label{eq:Van}
V_{a_n}\to 0 \quad\text{in the weak operator topology on}\quad L^2(\R,\sigma).
\end{equation}
Hence $\sigma$ is a continuous measure. Indeed, suppose, contrary to our claim, that $s_0\in\R$ is an atom of $\sigma$. Let $h\in L^2(\R,\sigma)$ be the indicator function of the singleton $s_0$. As $V_th(s)=e^{its}h(s)$, we have $\langle V_{t}h,h\rangle=e^{is_0t}\sigma(\{s_0\})$,
so $|\langle V_{a_n}h,h\rangle|=\sigma(\{s_0\})$ for all $n\geq 1$ contrary to \eqref{eq:Van}.
By \eqref{eq:Tfan}, we also have
\[V_{a_n}\to \int_\R V_{-t}\,dP(t) \quad\text{in the weak operator topology on}\quad L^2(\R,\sigma).\]
Together with \eqref{eq:Van}, this gives $\int_\R V_{-t}\,dP(t)=0$ on $L^2(\R,\sigma)$. Therefore for all $h_1,h_2\in L^2(\R,\sigma)$, we have
\begin{align*}
0&=\langle \int_\R V_{-t}\,dP(t)(h_1),h_2\rangle=\int_{\R}\int_\R e^{-its}h_1(s)\overline{h_2}(s)\,d\sigma(s)\,dP(t)
=\int_{\R}\widehat{P}(s)h_1(s)\overline{h_2}(s)\,d\sigma(s),
\end{align*}
where $\widehat{P}$ is the Fourier transform of the measure $P$. Therefore (since $h_1$ and $h_2$ are arbitrary), $\widehat{P}(s)=0$ for $\sigma$ a.e.\ $s\in\R$. As $\sigma$ is continuous, we have $\widehat{P}=0$ on an uncountable subset of $\R$.
On the other hand, as $P$ has exponentially decaying tails (see \eqref{eq:tails}), its Fourier transform $\widehat{P}$ is an analytic function on $\R$.
It follows that $\widehat{P}\equiv 0$, contrary to non-triviality of the measure $P$. This completes the proof.
\end{proof}



\section{Singularity from trimmed Birkhoff sums}\label{sec:singularity}
In this section we reduce the verification of the criterion for {spectral} singularity (Theorem~\ref{thm:singcrit}) for the class of special flows we are interested in, to a result about Birkhoff sums of the \emph{derivative} $f'$ of the roof function $f$.  Both the tightness and exponential tails required for the {spectral} singularity criterion will be proven through these bounds.  We first  introduce the notion of \emph{trimmed Birkhoff sum}, state the main result on trimmed Birkhoff sums (Proposition~\ref{prop:cohexistence}) which will be proved in \S~\ref{sec:coexistence}, and then deduce from it the proof of singularity (namely of Theorem~\ref{thm:dm_sf}).
\subsection{ Birkhoff sums of the derivative trimming}\label{sec:trimmedBS}
Since we assume that $f$ has logarithmic singularities,  the derivative $f'$ of $f$ will have \emph{power-like singularities} of type $1/x$.

\subsubsection{Form of the singularities}
To describe these singularities, analogously to what we did for the logarithmic singularities of $f$,
let us use the following auxiliary function $u$: 
\begin{equation}\label{eq:positive+}
u(x) :=- \frac{1}{(x)^{pos}},\qquad
\text{where}\ (x)^{pos}: =
 \begin{cases}
 x & \text{if}\ 0<x\leq 1, \\   +\infty & \text{if}\ x\leq 0.
\end{cases}
\end{equation}
(here $(x)^{pos}$ denotes, in a sense, the positive part of $x$). Recalling the definition of $\mathrm{Log}$ in \eqref{eq:positive} one can check that the definition is given so that $u(x)=\mathrm{Log}'(x)$.
  Let us then consider $u_i^+ (x): = u(x-\beta_i)$, for  $i=0,\dots , d-1$ and $u_i^-(x):= -u(\beta_i-x)$, $i=1,\dots , d$.  Then one sees that if $f$ has pure logarithmic singularities of the form \eqref{eq:formlogsing}, then the derivative $f'$ has the form:
\be \label{eq:form1xsing}
f'(x)=     -  \sum_{i=0}^{d-1}  \frac{C_i^+}{(x-\beta_i)^{pos}} +
    \sum_{i=1}^{d} \frac{C_i^-}{(\beta_i-x)^{pos}}   
= \sum_{i=0}^{d-1}    C_i^+  u_i^+(x) + \sum_{i=1}^{d}  
 C_{i}^- u^-_{i}(x)  .
\ee

We remark that this function is not in $L^1$ 
(since singularities of type  $1/x$ are \emph{not} integrable). To bound the derivative Birkhoff sums $S_n f'$ it is necessary to consider separately the contributions of the largest terms, which are controlled by the closest visits of an orbit to the singularities.  We therefore introduce the notation for closest visits to singularities and define \emph{trimmed} Birkhoff sums.


\subsubsection{Closest visits}\label{sec:closests}%
Given an IET $T:I\to I$, a point $x\in I$ and a natural number $r$,   let us consider the orbit segment   of length $r$,
\[
\mathcal{O}_T(x,r) :=\{ x, T(x), \dots , T^{r-1}(x)\}.
\]
For {every} discontinuity $\beta_i$ with $0\leq i<d$, denote by $x_i^{+}=x_i^+ (x,r)$ 
 the closest visits of the orbit segment to $\beta_i$  \emph{from  the right}, and, correspondingly, by
$m_i^{+}=m_i^+ (x,r)$ 
 the minimum distance of the orbit points to  $\beta_i$  \emph{from  the right}, namely:
\bes
m_i^{+}= m_i^{+}(x,r) : =   \min \{ (T^j x - {\beta}_i )^{pos} , \\ 0\leq j < r \},\quad  {x_i^{+} = x_i^{+}(x,r)  :={\beta}_i+ m_i^{+}}, \ \
 i=0, \dots, d-1\footnote{Notice that we exclude here $i=d$ since $\beta_d=1 $ cannot be approached from the right.},
\ees
 where $(x)^{pos}$ the 'positive' part
of $x$ defined in \eqref{eq:positive+}. %
Similarly, denote  by {$x_i^{-}:=x_i^{-}(x,r)$,  for $i=1, \dots, d$,  the closest visit to ${\beta}_i$  among the points of $\mathcal{O}_T(x,r)$ \emph{from the left}, and correspondingly by} $m_i^{-}:=m_i^{-}(x,r)$,  for $i=1, \dots, d$,  the minimum distance from the ${\beta}_i$  of the points of $\mathcal{O}_T(x,r)$ from the left\footnote{We exclude now $i=0$ since $\beta_0=0$ cannot be approached from the left.}:
\bes
m_i^{-} = m_i^{-}(x,r) : = \min \{  ({\beta}_i  - T^j x )^{pos}, \  0\leq j < r \}, \quad {x_i^{-} = x_i^{-}(x,r) : = \beta_i- m_i^{-}(x,r),} \ \  i=1, \dots, d.
\ees
Moreover,  denote by $m(x,r)$ the \emph{{minimum distance} to the singularity set} of $T$ given  by
\begin{align*}
m(x,r):=& \min \big\{ \{ m_i^{+}(x,r), \ 0\leq i< d   \} \cup  \{ m_j^{-}(x,r), \ \   1\leq j\leq d \}\big\}  \\
=& \min \left\{ \left|T^k(x) - \beta_i\right|, \ \  0\leq k<r, \ 0\leq i\leq d \right\} .
\end{align*}

\subsubsection{Definition of trimmed Birkhoff sums}\label{sec:trimmedBS def}%
We can now introduce the notion of  \emph{trimmed} Birkhoff sums. Since we will use it both for the derivative $f'$ of the form \eqref{eq:form1xsing} and for the function $f$, let us define it more generally for functions $v:\bigcup_{i=1}^d I_i \to \mathbb{R}$ of the form  {$v=   \sum_{i=0}^{d-1}  C_i^+   v^+_i  + \sum_{i=1}^{d}  C_{i}^- v_{i}^-$},
where the functions $v_i^+$ are monotonically decreasing, have a right-singularity at $\beta_i$ (i.e.~blow up as $x\to \beta_i^+$) and $ v_i^-$ are monotonically increasing and have a left singularity at $\beta_i$ (i.e.~blow up as $x\to \beta_i^-$).  (Then both $f$ and $f'$ have this form when taking $v_i^\pm := \mathrm{Log}_i^\pm$ for $f$ and, respectively, $v_i^\pm:=  u_i^\pm$ for $f'$).
A trimmed Birkhoff sum is then defined removing from  Birkhoff sums all contributions given by all these closest visits to singularities, from the right and from the left:
\begin{definition}\label{def:trimmedBS}
 Given an IET $T$ and a function $v$ as above, the \emph{trimmed Birkhoff sum} $\widetilde{S}_r v $ is defined for {every} $x\in I\setminus End(T)$ as
\[
\widetilde{S}_r v (x):= S_r v (x) -
\sum_{i=0}^{d-1} C^+_iv_i^+(x_i^+(x,r))  -  \sum_{i=1}^{d}  C^-_iv_i^- (x_{i}^- (x,r)).
\]
\end{definition}
\noindent Notice that, since the functions $v_i^\pm$ are monotonic, { $|v_i^\pm(x_i^\pm(x,r))|$}
are the largest values of $|v_i^\pm|$ evaluated along the orbit
segment $\mathcal{O}(x,r)$; this justifies the use of the terminology \emph{trimmed} Birkhoff sums (\emph{trimming} a Birkhoff sums usually refers to removing the largest element from the sum, see e.g.~\cite{DV,KS,Au-Sch}).

This definition applies in particular to $f\in   \pSymLog{T}$ with \emph{pure} logarithmic singularities.
Let us extend this definition to $f\in   \SymLog{T}$, by writing $f=f^{\operatorname{p}}+g$, where $f^{\operatorname{p}} \in \pSymLog{T}$ and $g$ is piecewise absolutely continuous and setting
\be \label{eq:logtrimming}
 \widetilde{S}_n f: =   \widetilde{S}_n f^{\operatorname{p}} + S_n g.
\ee

\subsection{Bounds of trimmed derivatives along rigidity towers}\label{sec:sc via Bs}
We now state the form of control of trimmed Birkhoff sums of the derivatives of the roof which we will prove and will allow us to verify the assumptions of the {spectral} singularity criterion.
The estimates that we need give a \emph{linear} control of the trimmed Birkhoff sums of the derivative over a suitable large set which is a rigidity set. Such rigidity set for the IET on the base will be given by large Rohlin towers (namely \emph{rigidity towers}, that we now define).


\subsubsection{Rigidity towers} The rigidity sets that we will build for the base IETs will all be given by (Rohlin) towers by intervals:
\begin{definition}[towers by intervals]   \label{def:tower}
Given an interval $J :=(a,b) \subset I:=[0,1]$ and $h\in \N$ we say that the union
$$\cC :=\bigcup_{i=0}^{h-1}T^i J$$
is a (Rohlin) \emph{tower by intervals} of \emph{base} $J$ and of \emph{height} $h$  for the  IET $T:I\to I$ if and only if the images $T^i J, 0\leq i < h$ are pairwise disjoint intervals which do not contain any discontinuity for $T$.
\end{definition}
\noindent Notice that since $T$ is a piecewise isometry and the floors of a tower are disjoint, the (Lebesgue) \emph{measure} $|\cC|$ of the Rohlin tower is
\begin{equation}\label{hwtower}
0\leq |\cC| = h |J| = h |b-a|\leq 1.
\end{equation}
We recall Rohlin towers by intervals for IETs can be produced naturally by \emph{inducing}: if $J\subset [0,1]$ is a subinterval and $T$ a minimal IET,
the first return map of $T$ to $I$ is an interval exchange of at most $d+2$ exchanged intervals,     
and the  return time on each of them is constant. If $h_j$ is the return time on the continuity interval $J_i\subset J$, then $\mathcal{T}^J_i:=\cup_{k=0}^{h_j-1}T^k (J_i)$ is a Rohlin tower by intervals, that we will call \emph{tower over the inducing interval} $J_i$.

\begin{definition}[rigidity towers]\label{def:tow}
Given $\epsilon>0$, an $\epsilon$-\emph{rigidity tower} for the  IET $T:I\to I$ is  a  \emph{tower by intervals} of \emph{base} $J$ of \emph{height} $h$  for the  IET $T:I\to I$
if we have that
 $$d( T^h x,x)<{\epsilon}|J|\leq \frac{\epsilon}{h}, \qquad \text{for\ every}\ x\in J,$$
(the second inequality being automatic since $|J| h\leq 1$).

 Furthermore, we say that a sequence $(\cC_n)_{n\in\mathbb N}$ of towers by intervals is a sequence of \emph{rigidity towers} for the IET $T$ if there exists a sequence $(\epsilon_n)_{n\in\mathbb N}$ such that each $\cC_n$ is $\epsilon_n$-rigid and $\epsilon_n \to 0$ as $n\to\infty$.

 \end{definition}

\subsubsection{Linearly bounded trimmed derivatives}\label{sec:linearbounds}
The form of control on Birkhoff sums of derivatives that we need is the following, {which shows that, at special times and for sums inside a Rohlin tower, $\widetilde S_{r} v(x)$ grows \emph{linearly} (as expected for an integrable function, although $v$ is \emph{not} integrable). We stress that here the \emph{symmetry} of the singularities plays a crucial role, since if the singularities are logarithmic but \emph{asymmetric}, the growth is known to be faster than linear\footnote{If $v = 1/x$, one can show that for a.e.~IET $T $ the Birkhoff sums $S_r v$ grown like $r\log r$, i.e.~that $S_r v /r\log r$ converges in probability to $1$ (as random variable in $x$). In \cite{Au-Sch}, when $T$ is a rotation, it shown that $S_r v (x)/r\log r$ converges for almost every $x\in [0,1]$.}, as shown in \cite{Ul:mix, Rav:mix}.
\begin{definition}[linearly bounded trimmed derivatives]\label{def:bounded}
Given an IET $T$, a function $f\in  \SymLog{T}$ 
 and a sequence  $(\mathcal{T}_n)_{n\in\mathbb N}$ of Rohlin towers by intervals with bases $J_n\subset[0,1]$ and heights $h_n$ {for $T$},  we say that $f$ has \emph{linearly bounded trimmed derivatives} along the Rohlin towers $(\mathcal{T}_n)_{n\in\mathbb N}$ {(for $T$)} if there exists  a constant $M>0$ such that
the  trimmed Birkhoff sums of the derivative satisfy
\[
\left|\widetilde{S}_r f'(x)\right| \leq M h_n 
 \ \  \text{for \ any} \ x\in J_n \text{\ and\ any}  \ 0\leq r \leq h_n.
 \]
In this case,  to specify the constant $M>0$, we say that the sequence  $(\mathcal{T}_n)_{n\in\mathbb N}$ has $M$-bounded trimmed derivatives.
\end{definition}
\noindent Thus, the linearly bounded trimmed derivatives assumption provides uniform  control for all trimmed Birkhoff sums (using the extended definition \eqref{eq:logtrimming} of trimming) of points from the base of any tower up to the height. 



\subsubsection{Trimmed derivative bounds along balanced Rohlin towers}
\label{sec:cancellations}
Linear bounds on trimmed  Birkhoff sums of the derivative $S_nf'$ of a function with symmetric logarithmic singularities were proved by  the last author in \cite{Ul:abs},  
  (since they  are at the core of the result on absence of mixing for locally Hamiltonian flows proved in \cite{Ul:abs}). With the notation for trimmed Birkhoff sums we just introduced in \S~\ref{sec:trimmedBS def}, Proposition 4.2 in \cite{Ul:abs} can indeed be restated as follows:
\begin{proposition}[Prop.~4.2 in \cite{Ul:abs}]\label{boundSf'growthprop}
For a.e.~IET $T$, 
there exists a sequence of nested intervals $( I^{\ell})_{\ell \in \mathbb{N}}$ such that  {every} $f\in \pSymLog{T}$  has linearly bounded trimmed derivative Birkhoff 
 along the  towers $( \mathcal{T}^{\ell}_j)_{\ell\in\N}$, $1\leq j\leq d$ over the inducting intervals $I^{\ell}_j$, $1\leq j\leq d$, i.e.~there exists a constant
  $M>0$ such that:
\begin{equation} \label{eq:propconclusion}
\left| \widetilde{S}_r f'(x)\right| \leq M h^{\ell}_j, \quad \text{for \ all} \ 1\leq j\leq d, \ x\in I^{\ell}_j, \ 0\leq r\leq h^{\ell}_j,
 \end{equation}
  where  $h^{\ell}_j$ is the first return time of $I^{\ell}_j$ to $I^{\ell}$, i.e.~the height of the tower $\mathcal{T}^{\ell}_j$.
\end{proposition}
\noindent
{
We will not use
Proposition~\ref{boundSf'growthprop}, directly,
but we will exploit it indirectly to prove singularity,} see \S~\ref{sec:coexistence},  %
as we explain in the next \S~\ref{sec:rigidity_trimmedbounds}.} 

\subsubsection{Trimmed derivative bounds along rigidity towers}\label{sec:rigidity_trimmedbounds}
The main technical result that we need to prove singularity  is the following.

\begin{proposition}[derivative bounds along rigid towers]\label{prop:cohexistence}
 {For a.e.~IET $T$, there exists  a sequence of rigidity towers $(\cC_\ell)_{\ell\in \mathbb{N}}$ with $|\cC_\ell|\to 1$ such that  every $f\in \pSymLog{T}$
has linearly bounded trimmed derivatives along the sequence  $(\cC_\ell)_{\ell\in \mathbb{N}}$.}
\end{proposition}
\noindent
Section~\ref{sec:coexistence} is devoted to the proof of this Proposition.
{ We stress that the towers provided by this Proposition~\ref{prop:cohexistence} are \emph{rigidity sets}. Linear bounds on trimmed Birkhoff sums of the derivative also
along the towers produced by Proposition~\ref{boundSf'growthprop}, but, by nature of their construction,  those towers are \emph{balanced}\footnote{The balance of the towers $\mathcal{T}^\ell_j$
  produced  in the proof of Proposition~\ref{boundSf'growthprop}
in \cite{Ul:abs} plays indeed an essentially role in the proof of the cancellations in the Birkhoff sums of the derivative (which are very delicate and depend on a sophisticated Diophantine-type condition for the IET, see \S~\ref{sec:app} or \cite{Ul:abs} for more details).} (i.e.~have  measure  uniformly bounded away from $0$ and $1$, see Remark~\ref{rk:cancellations_are_balanced}).
In order to apply the singularity criterion given by Proposition~\ref{thm:singcrit}, on the other hand,  we require that linear bounds hold along towers which are \emph{rigidity towers} (i.e.~have measure tending to $1$).}

\subsection{Proof of singularity from trimmed bounds along rigidity towers}
In this section, assuming the Proposition~\ref{prop:cohexistence} just stated,  we conclude the proof of singularity (Theorem~\ref{thm:dm_sf}). We want in particular to verify the assumptions of the {spectral} singularity criterion (Theorem~\ref{thm:singcrit}), namely \emph{tightness} and \emph{exponential tails} along a rigidity set.
We first recall how to obtain rigidity sets from rigid towers for IETs (see \S~\ref{sec:rigidtowers}). The heart of the verification, namely that linear bounds from trimmed derivatives along rigid towers allow to prove the exponential tails, is proved in \S~\ref{sec:exp_tails}, the final arguments are then summarized in \S~\ref{sec:final_sing}.
\smallskip

\subsubsection{Rigidity sets in rigid towers.}\label{sec:rigidtowers}
The existence of a rigidity tower in an IET immediately gives a rigidity set.
 We include the short proof of this well known mechanism for rigidity in IETs for convenience for the reader (highlightening the description of the dynamics inside the tower, which will be useful for later proofs).
\begin{lemma}[almost cylinders inside rigidity towers] \label{lem:Rohlinrigidity}
If $(\cC_n)_{n\in\mathbb N}$
 is a sequence of rigidity towers for an  IET $T$ on $[0,1]$ 
such that $|\cC_n|\to 1 $ as $n\to \infty$, the sequence
 $(h_n)_{n\in\mathbb N}$  of their heights is a rigidity sequence for $T$. Furthermore, if $J_n$ is the base of the tower $\cC_n$, then there exists a sequence of smaller rigidity towers $C_n\subset \cC_n $ of base $J'_n\subset J_n$, still satisfying $|C_n|\to 1$, such that $T^{h_n} J'_n\subset J_n$ and $T^{h_n}$ acts as an isometry on every floor of $C_n$.
\end{lemma}
\begin{definition}[almost cylinders] We will call the sets $(C_n)_{n\in\mathbb N}$ given by this Lemma the \emph{almost cylinders} given by the rigidity towers $(\cC_n)_{n\in\mathbb N}$.
\end{definition}
\begin{proof}[Proof of Lemma~\ref{lem:Rohlinrigidity}]
Since by assumption each floor $T^i J_n$, $0\leq i<h_n$  of $\cC_n$ 
is an interval of \emph{continuity} for $T$, $T^i$ acts isometrically on $J_n$ for every $0\leq i \leq h_{n}$. By assumption there exists a sequence $(\epsilon_n)_{n\in\mathbb N}$ with $\epsilon_n\to 0$ such that  $d(T^{h_n}z,z)\leq\delta_n$ for {every} $z\in J_n$, where $\delta_n:=\epsilon_n |b_n-a_n|$.  Then, we can find a smaller subinterval  $J_n'\subset J_n$ such that $T^{h_n}(J_n')\subset J_n$ and $|J_n'|=  |J_n|-\epsilon_n$, by taking $J_n'=(a_n+\delta_n, b_n)$ if $T^{h_n} a_n\leq  a_n$ or $(a_n, b_n- \delta_n)$ otherwise. 
Then setting $C_n $ to be the Rohlin tower of height $h_n$ over $J_n'$, we have that\footnote{To see this formally, given $x\in C_n$, write $x=T^k x_0$ for some $x_0\in J_n'$ and $0\leq k< h_n$. Then   $T^{h_n}(x)=T^k T^{h_n-k}(x) $} $T^{h_n} (x) \in \cC_n$ for {every} $x\in C_n$  and that $|C_n|=|J_n'|h_n = (1-\epsilon_n)|J_n| h_n\to 1$ since $|\cC_n|=|J_n| h_n\to 1$.
\end{proof}

\subsubsection{Exponential tails from trimmed derivative bounds}\label{sec:exp_tails}
We now show that this control on large Rohlin towers filling the space suffices to show tightness and exponential bounds.
\begin{proposition}[exponential tails via trimmed derivatives]\label{lem:expfromderivatives}
If $(\mathcal{T}_n)_{n\in\mathbb N}$ is a sequence of rigidity towers for $T$ with $$|\mathcal{T}_n| \to 1\  \text{as} \ n\to \infty$$ and $f\in  \pSymLog{T}$  has \emph{linearly bounded trimmed derivatives} along $(\mathcal{T}_n)_{n\in\mathbb N}$, then there exists  centering constants $(c_n)_{n\in\mathbb N}$
 such that the sequence $(S_{h_n} f(x)-c_n)_{n\in\N}$ has exponential tails. 
\end{proposition}
The rest of this subsubsection is devoted to the proof of this Proposition. In essence, the uniform bounds on the trimmed sums gives tightness, while  the contribution of closest visits are responsible for the exponential tails of the distribution.

\begin{proof}[Proof of Proposition~\ref{lem:expfromderivatives}]
To prove that $S_{h_n} f(x)-c_n$ have exponential tails we need to find sets  $(A_n)_{n\in\mathbb N}$ with $|A_n|\to 1$ for which the exponential tail bound \eqref{eq:expdecay} holds on $A_n$. As sets  $(A_n)_{n\in\mathbb N}$ we will take the almost cylinders $(C_n)_{n\in\mathbb N}$ given by the rigidity towers $ \cC_n $  constructed in the proof of  Lemma \ref{lem:Rohlinrigidity}, i.e. set     $A_n:=C_n\subset \cC_n $. 
We then want  to prove tightness and exponential tails of $S_{h_n} f$ on $(A_n)_{n\in\mathbb N}$ after a suitable centering.  For defining the  \emph{centering constants} $(c_n)$, let $z_n:=(b_n-a_n)/2$ be the midpoint of the base interval $J_n=(a_n,b_n)$ (of the tower $\cC_n$) and set
\[
c_n:= S_{h_n} f (z_n), \qquad n\in \mathbb{N}.
\]
Given any $x\in C_n$, to study the difference $S_{h_n} f(x)-c_n$  we therefore want to compare
 $S_{h_n} f(x)$ and $S_{h_n} f(z_n)$.  A standard way to do so in a tower is
to write $x = T^{k}x_n$, where $x_n \in J'_n$ is  a point in the  base of $C_n$ and $0\leq k < h_{n}$ the height of the floor of the tower $C_n$ which contains $x$. Then, using the cocycle properties of Birkhoff sums, we can add and subtract $\BS{f}{k}(x_n)$ and  split the Birkhoff sums difference\footnote{To do so, using the cocycle properties we see that  $ \BS{f}{h_n} (x)
 = \BS{f}{h_n-k}(x)+ \BS{f}{k}(T^{h_n-k}x)$ and  $ \BS{f}{h_n} (x_n)
 = \BS{f}{k}(x_n)+ \BS{f}{h_n-k}(T^{k}x_n)$. After subtracting, we get the sum of $\BS{f}{h_n} (x)- \BS{f}{h_n} (x_n)=\BS{f}{k}(T^{h_n}x_n)-\BS{f}{k}(x_n)$
(where we used that $x=T^k (x_n)$).}
\begin{equation} 
\label{eq:BSsplitting}
 \BS{f}{h_n} (x) -c_n
=\left( \BS{f}{h_n}(x_n)  - \BS{f}{h_n}(z_n)\right) - {\left(   \BS{f}{k}(T^{h_n} x_n) - \BS{f}{k} (z_n) \right)+\left(   \BS{f}{k}(x_n) - \BS{f}{k} (z_n) \right)}.
\end{equation}
Notice that both terms are now differences of Birkhoff sums along orbits that travel together inside a Rohlin tower, so we can apply to each difference the mean value theorem. We want, though, to first to set aside the contributions given by the closest singularities (which will produce the exponential tails) and then apply mean value only to the corresponding
 \emph{trimmed} sums (to which we can apply the trimmed derivative bounds). By the definition of trimmed sums (see Definition~\ref{def:trimmedBS}) and recalling  the notation introduced in \S~\ref{sec:closests}, for {all} $y\in I$ and $r\in \mathbb{N}$,
\begin{align*}
 \BS{f}{r}(y)& = \widetilde{S}_r f (y) + \sum_{i=0}^{d-1}C_i^+ \mathrm{Log}_i^+ (x_i^+(y,r))  + \sum_{i=1}^{d}C_i^- \mathrm{Log}_i^- (x_i^-(y,r))\\
& = \widetilde{S}_r f (y) {-} \sum_{i=0}^{d-1}C_i^+ \log m_i^+(y,r)   {-} \sum_{i=1}^{d}C_i^- \log m_i^-(y,r) .
\end{align*}
Thus we can estimate the  first difference of Birkhoff sums in  \eqref{eq:BSsplitting} by:
\begin{align}\label{1BSestimate}
\left| \BS{f}{h_n}(x_n)  -\right.  \left. \BS{f}{h_n}(z_n)\right| \ & \leq  \left| \widetilde{S}_{h_n} f (x_n) -   \widetilde{S}_{h_n} f (z_n) \right| \\ & +\left| \sum_{i=0}^{d-1}C_i^+ \log \frac{m_i^+(x_n,h_n)}{m_i^+(z_n,h_n)}  + \sum_{i=1}^{d}C_i^- \log \frac{m_i^-(x_n,h_n)}{m_i^-(z_n,h_n)} \right|. \nonumber
\end{align}
Let us first estimate the difference of trimmed Birkhoff sums in \eqref{1BSestimate} using the mean value theorem. Notice that  for any $x$ in $J_n$ we can differentiate and get $\frac{\mathrm{d}}{{\mathrm{d}x}}\widetilde{S}_{h_n} f (x)= \widetilde{S}_{h_n} f' (x)$ (since  $T^k$ acts as an isometry on $(a_n,b_n)$  for {all} $0\leq k< h_n$).     Thus, by mean value first and then the linearly bounded trimmed derivative assumption, for some $\xi$ between $x_n$ and $z_n$
\[
\big| \widetilde{S}_{h_n} f (x_n) -   \widetilde{S}_{h_n} f (z_n) \big| = \big|  \widetilde{S}_{h_n} f' (\xi)  \big|\cdot |x_n-z_n|\leq M h_n |b_n-a_n|\leq M.
\]
Let us now estimate  the sum in the RHS of \eqref{1BSestimate} by estimating the contribution coming from closest visits. We claim that for {every} $i$ for which they are defined and {every} $0<l\leq h_n$,
\begin{equation}\label{logbounds}
\left| \log \frac{m_i^\pm(x_n,l)}{m_i^\pm(z_n,l)} \right|\leq \max \left\{ \log 2, -\log \left( {h_n m(x_n, l)}\right) \right\},
\end{equation}
where  $m(x_n,l)$ is the distance of the closest visit {to the singularity set}, which by definition satisfies
$m(x_n,l)\leq m_i^\pm(x_n,l)$  (see \S~\ref{sec:closests}).
For this, remark that 
 $\cC_n$ does not contain any endpoint of $End(T)$ in its interior, and $T^j(z_n)$ are all midpoints of the corresponding floor $T^j(J_n)$ for all $0\leq j<h_n$ (since
 $z_n$ is the midpoint of $J_n$ and each of these $T^j$ acts as an isometry on $J_n$). Thus, this gives (since $T^jz_n$ are all at least $|b_n-a_n|/2$ far from $End(T)$, and $T^j x_n$ belongs to the same floor of the tower $\cC_n$ as $T^j z_n$) that
 $$
\frac{|b_n-a_n|}{2} \leq m_i^\pm(z_n,l) \leq  m_i^\pm(x_n, l)  + \frac{|b_n-a_n|}{2}, \quad m_i^\pm(x_n,l) \leq  m_i^\pm(z_n, l)  + \frac{|b_n-a_n|}{2}. 
$$
Using these estimates and recalling that $|b_n-a_n|\leq 1/h_n$ by \eqref{hwtower}, we can get an upper and lower bound for the ratios:
$$
\frac{1}{2}\leq \frac{m_i^\pm(z_n, l)}{m_i^\pm(x_n, l)}\leq 1+\frac{1}{2h_n  m(x_n, l)}\leq \max \left\{ 2, \frac{1}{h_n m(x_n, l)}\right\},
$$
(the last inequality follows from $a+b\leq 2\max\{a,b\}$). These upper and lower bounds imply  the bounds \eqref{logbounds} for the absolute value of their logarithms.
 Thus, using this and the difference of trimmed Birkhoff sums bounds to estimate  \eqref{1BSestimate},
 we get
\begin{align*}
\left| \BS{f}{h_n}(x_n)  - \BS{f}{h_n}(z_n)\right|  \leq  M + C \max\{\log 2 ,-\log h_n m(x_n ,h_n)|\},
\end{align*}
where $C:= \sum_{i=0}^{d-1} C_i^+ +  \sum_{i=1}^d C_i^- $.
The same reasoning can be carried out also for the second difference in \eqref{eq:BSsplitting} (to which mean value and then the linearly bounded trimmed derivative assumption also applies, since $T^{h_n} x_n$ and {$x_n$} also both belong to the  base $J_n =(a_n,b_n)$ and $0\leq k< h_n$).  Thus, adding up both analogous estimates, we get 
\be\label{estimatebulk}
|\BS{f}{h_n} (x) -c_n|\leq {3M+ 3 C} \max\left\{\log 2, -\log {h_n \, m(x_n,h_n)} ,  -\log h_n\, {m(T^{h_n}x_n,k)}\right\}.
\ee
To conclude, we now want to use this estimate to show that the tails are exponential. Indeed, \eqref{estimatebulk} shows that  for any $\alpha>M':={3 M + 3C\log 2}$,
\be\label{levelset}
\min\{ m(T^{h_n}x_n,k), m(x_n, h_n)\} \geq  \varepsilon_n(\alpha):=\frac{e^{{-(\alpha-M')}/{3C}}}{h_n}\quad \Rightarrow \quad |\BS{f}{h_n} (x) -c_n|\leq \alpha.
\ee
If we denote by $\mathcal{N}_\alpha$ the $\varepsilon_n(\alpha) $-neighbourhood of the endpoints, i.e.
$\mathcal{N}_\alpha := \{ x\in [0,1]:\ d(x, End(T))< \varepsilon_n(\alpha)\}$,
we then claim that
\be \label{containement}
 \{x\in C_n : |S_{h_n}f(x)-c_n|\geq \alpha \}  \subset \bigcup_{-h_n<  i<  h_n } T^i
\mathcal{N}_\alpha .
\ee
Indeed, if $x\in C_n$ is such that $|S_{h_n}f(x)-c_n|\geq \alpha$,  by \eqref{levelset} either   $m(x_n, h_n)$ or  $ m(T^{h_n}x_n,k)$ are less than  $\varepsilon_n(\alpha) $, i.e.~either the orbit segment $\mathcal{O}(x_n,h_n)$ or $\mathcal{O}(T^{h_n }x_n,k)$ intersect $\mathcal{N}_\alpha$.
Recalling that $x_n$ is the projection to the base of the tower $C_n$ of $x$, i.e.~$x=T^k x_n$, one can see that  the following inclusion of orbits holds: 
$$\mathcal{O}(x_n,h_n)\cup \mathcal{O}(T^{h_n }x_n,k)\subset \mathcal{O}(T^{-(h_n-1)}x_n,2 h_n-1)  = \{T^{-(h_n-1)}x, \dots, T^{-1}x, x,Tx  \dots  T^{h_n-1}x \},$$  so if either the orbit segment $\mathcal{O}(x_n,h_n)$ or $\mathcal{O}(T^{h_n }x_n,k)$ intersect $\mathcal{N}_\alpha$, there exists $-h_n<i<h_n$ such that $T^i (x) \in \mathcal{N}_\alpha$
and  \eqref{containement} follows. Then, recalling the definition of $\varepsilon_n(\alpha)$ in \eqref{levelset}, we get
\[
\left| \{x\in C_n : |S_{h_n}f(x)-c_n| \geq \alpha \} \right|  \leq 2h_n |\mathcal{N}_\alpha| = 4 dh_n\varepsilon_n (\alpha)={4d}{e^{\frac{M'-\alpha}{3C}}} =  K  e^{-\alpha/{3C}}
\]
where $K:=4d e^{M'/{3C}}$. This
 shows that the tails are exponential and hence finishes the proof.
\end{proof}

\subsubsection{Concluding the proof of singularity}\label{sec:final_sing}
We finish this section summarizing why with the proof of the proposition we also conclude the proof of the singularity result for special flows, i.e.~Theorem~\ref{thm:dm_sf}.
\begin{proof}[Proof of Theorem~\ref{thm:dm_sf}]
Let us verify the assumptions of the {spectral} singularity criterion (Theorem~\ref{thm:singcrit}),
in particular the exponential tails assumption.
By Proposition~\ref{prop:cohexistence}, there exists  a sequence of rigidity towers $(\cC_n)_{n\in \mathbb{N}}$ with $|\cC_n|\to 1$ with $(h_n)_{n\in\N}$ as a rigidity sequence such that {$f^{\operatorname{p}}$} has linearly bounded trimmed derivatives along the sequence  $(\cC_n)_{n\in \mathbb{N}}$.

By Lemma~\ref{lem:Rohlinrigidity} and Proposition~\ref{lem:expfromderivatives}, there exists a sequence of trimmed towers $(C_n)_{n\in \mathbb{N}}$
with $|C_n|\to 1$ and such that $T^{h_n}J'_n\subset J_n$, where $J_n$ is the base of $\cC_n$ and $J'_n$ is the base of $C_n$, and the sequence $(S_{h_n}f^{\operatorname{p}}(x)-c_n)_{n\in\N}$ has exponential tails with $c_n=S_{h_n}f^{\operatorname{p}}(z_n)$, where $z_n$ is the midpoint of $J_n$, i.e.\
there are positive constants $C,b$  such that
\[Leb\left(\left\{x\in C_n:|S_{h_n}f^{\operatorname{p}}(x)-c_n|\geq t\right\}\right)\leq Ce^{-bt}.\]
As $g=g_f:I\to\RR$ is of bounded variation and $T^{h_n}J'_n\subset J_n$, standard arguments first used for IETs by Katok in \cite{Ka:int}, show that
\[|S_{h_n}g(x)-v_n|\leq V\quad\text{for}\quad x\in C_n,\quad\text{where}\quad v_n=S_{h_n}g(z_n)\quad\text{and}\quad V=\textrm{Var}(g).\]
As $f=f^{\operatorname{p}}+g$, taking $A_n:=C_n$ and $a_n:=c_n+v_n$, it follows that
$$Leb \left( \left\{  x\in A_n  : |S_{h_n}f(x)-a_n | \geq t \right\} \right)\leq Ce^{-b(t-V)},$$
so the sequence $(S_{h_n}f(x)-a_n)_{n\in\N}$ has exponential tails.
 Since by assumption $(h_n)_{n\in\N}$ is a rigidity sequence and the previous equation gives the exponential tails assumption \eqref{eq:expdecay}, the {spectral} singularity criterion given by Theorem~\ref{thm:singcrit} implies that the special flow $(T_t^f)_{t \in \RR}$ has singular spectrum and is spectrally disjoint from all mixing flows.
\end{proof}

\section{Resonant mixing times and {spectral} disjointness}\label{sec:orthogonalityproof}
In this section we present the main arguments towards the proof of  the {spectral} disjointness result for special flows (i.e.~Theorem~\ref{thm:sfdisjointness}). 
In \S~\ref{sec:mainsteps_dj} we state the two results (Proposition~\ref{prop:mix'} and Proposition~\ref{prop:3}) from which (together with the {spectral disjointness} criterion given by Theorem~\ref{thm:spectorth}) we can deduce the proof of Theorem~\ref{thm:sfdisjointness}. We use them to prove pairwise {spectral} disjointness in \S~\ref{sec:final_disj}.

The rest of this section is then devoted to proving one of these two results (Proposition~\ref{prop:mix'}),  namely to showing that certain \emph{resonant rigid times} (see Definition~\ref{def:resonantrigidity_t})
are \emph{mixing times}. After having introduced several auxiliary results in the subsections   \S~\ref{sec:mixingviashearing} to \S~\ref{sec:Rothestimtes}, the  proof of this proposition is given in \S~\ref{sec:proof_resonantmixing}.
The other result (Proposition~\ref{prop:3}),
that gives the existence of  infinitely many such \emph{mixing {resonant} times} 
within a given sequence, will be  
proved in Section~\ref{sec:resontanttimes2}.

\subsection{Main Steps in the proof of {spectral} disjointness}\label{sec:mainsteps_dj} 
We present here the main results to prove {spectral} disjointness, which will be used to verify the assumptions of  the spectral {disjointness} criterion (Theorem~\ref{thm:spectorth}).
For almost every IET $T$, to prove singularity of the spectrum of $(T^f_t)$, we have already shown in the previous section the existence of a sequence of centering constants $(a_n)_{n\in\mathbb N}$ such that the distribution of Birkhoff sums $(S_{h_n} f - a_n)_{n\in\mathbb N}$  for $T$ is tight and has exponential tails.


The assumptions of the  {spectral disjointness} criterion hold if there  exists such a centering  sequence $(a_n)_{n\in\mathbb N}$ 
which is also a \emph{mixing} sequence for the second flow (in the sense of Definition~\ref{def:mixingseq}). The type of times along which we want to prove mixing (resonant rigid times) is defined in \S~\ref{sec:resonantrigid}. We state the two auxiliary results in \S~\ref{sec:mixingresontant_prop} and \S~\ref{sec:resontanttimes1}
 (Proposition~\ref{prop:mix'} and
 Proposition~\ref{prop:3} respectively), from which, in the next \S~\ref{sec:final_disj},
we  deduce Theorem~\ref{thm:sfdisjointness}.

\subsubsection{Resonant rigid times}\label{sec:resonantrigid}
The mixing times for the second flow that we will use are \emph{resonant} rigid times for the IET in the base, i.e.~are multiples $kq $ of a rigid time $q$ where $k$ is an integer
 and $q, 2q, \dots , kq$ are all rigid times. We formalize this idea through the following definition. Recall that the notion of an $\epsilon$-\emph{rigid tower by intervals} was defined in Definition~\ref{def:tow}.

\begin{definition}
Given an IET $U$ and $0<\epsilon<1$, we say that a natural number $q$ is an $\epsilon$-\emph{rigidity time} for $U$ if $q$ is the height of an $\epsilon$-rigid tower $\cC$ such that the measure of $\cT$ is greater than $1-\epsilon$.
\end{definition}

\begin{definition}[resonant rigidity time]\label{def:resonantrigidity_t}
Given an IET $U$, we say that $r>0$ is a $(\epsilon,k)$-\emph{resonant time}, or simply a \emph{resonant time}, if we can write $r=k q$ where:

\begin{itemize}
\item[(i)] $\epsilon>0$ and $k$ is an integer such that $1\leq k \leq \frac{1}{\epsilon}$;
\item[(ii)] $q$ is the height of an $\epsilon^2$-rigid tower by intervals $\cC$ for $U$;
\item[(iii)] if $B$ is the base of the tower $\cC$, the measure $|B| q > 1-\epsilon$.
\end{itemize}
\noindent Given sequences $\underline{\epsilon}:=(\epsilon_m)_{m\in\mathbb N}$ and $\underline{k}:=(k_m)_{m\in\mathbb N}$, where
 $\epsilon_m>0$ and $k_m\in \mathbb{N}$ for {every} $m\in\mathbb{N}$, we say that $\underline{r}:=(r_m)_{m\in\mathbb N}$ is a \emph{sequence} of $(\underline{\epsilon}, \underline{k})$-\emph{resonant times} if, for every $m\in \mathbb{N}$, $r_m$ is an
 $(\epsilon_m, k_m)$-resonant time.
\end{definition}
\noindent Notice that we request that $\mathcal{T}$ is $\epsilon^2$-rigid (rather than only $\epsilon$-rigid). Since $k\leq 1/\epsilon$, this guarantees that not only the height $q$ of $\mathcal{T}$ is an $\epsilon$-rigidity time, but also its multiples $q, 2q, \dots, kq$ are all $\epsilon$-rigidity times. 


\subsubsection{Mixing of {resonant} rigid times}\label{sec:mixingresontant_prop}
We will show that resonant rigidity times (under suitable conditions) give mixing times (and so do times which are close to them).
Recall that a sequence $(t_n)_{n\in\mathbb N}$ is a mixing sequence for the flow $(U^h_t)_t$ if for {all} measurable $A,B$
\[
\lim_{n\to \infty} Leb^h(A\cap U^h_{t_n}B)= Leb^h(A)Leb^h(B).
\]



Let us recall that we denote by $\mathcal{I}_d$ the space of $d$-IETs on $[0,1]$ and by $m$ the natural Lebesgue measure on $\mathcal{I}_d$ (see \S~\ref{sec:aeIET}).
\begin{proposition}\label{prop:mix'}
For {every} $d\geq 2$, there exists a full-measure set $\mathcal{F}_d\subset \mathcal{I}_d$ of the set $\mathcal{I}_d$ of $d$-IETs such that
if $T\in \mathcal{F}_d$ 
 and $(k_mq_m)_{m\in\mathbb N}$ is a sequence of $(\underline{\epsilon},\underline{k})$-resonant times for $T$ with
 $$\log\log q_m\geq k_m,\quad \text{for\ all}\ m\in\mathbb N,$$
   then,  for every roof function $f\in \SymLogC{T}$ of geometric origin {(with $\int f(x)\,dx=1$)},  the sequence $(k_mq_m)_{m\in\mathbb N}$ is a mixing sequence for the flow $(T_t^f)_{t\in\R}$.
Furthermore, 
if  $(t_m)_{m\geq 1}$ is a sequence {for which there exists $C>1$ such that
\[t_m\in [k_mq_m, Ck_mq_m], \qquad \text{for all}\quad m\geq 1, \]}
then also  $(t_m)_{m\geq 1}$ is a \emph{mixing sequence} for the flow $(T_t^f)_{t\in\R}$.
\end{proposition}
\noindent We will prove this Proposition later in this section, see \S~\ref{sec:proof_resonantmixing}.

\subsubsection{Prescribing resonant rigid times}\label{sec:resontanttimes1}
This type of mixing times, which are based on (resonant) rigidity  
  turns out to be \emph{frequent}. In particular, it is possible to control where they occur, so that they can be chosen to coincide (infinitely often) with {a given sequence} which will be later taken to be the sequence of centralizing constants for the Birkhoff sums of the first flow.
This is the content of the second result needed for {spectral} disjointness:

\begin{proposition}[{control of  rigid times}]\label{prop:3}
{}
For {every}  $d\geq 2$, there exists a constant
$C=C_d>1$ such that the following holds.
Fix $0<\epsilon <1$ and
\begin{itemize}
\item[-] {an increasing} sequence $\underline{s}:=(s_n)_{n\in\mathbb N}$ of real positive numbers converging to $\infty$.
 \end{itemize}
Consider the subset  $\Omega:= \Omega \left(\underline{s}, {\epsilon}\right)\subset \mathcal{I}_d$
consisting  of all $d$-IETs on $[0,1]$ for which
there exists an increasing sequence $(q_j)_{j\geq 1}$ such that:
\begin{itemize}
\item[($\Omega$1)] $q_j$ is a $\epsilon$-rigid time
for {all} $j\geq 1$,
\item[($\Omega$2)]  for {every} $j\in \mathbb{N}$ there exists $n_j\in \mathbb{N}$ such that
\[
s_{n_j}\leq q_j<Cs_{n_j} .
\]
\end{itemize}
Then the set $\Omega$ 
has full measure in $\mathcal{I}_d$, i.e. $m(\mathcal{I}_d\setminus \Omega)=0$.
\end{proposition}
\noindent The proof of this proposition, which is heavily based on Rauzy-Veech induction properties (in particular on Kerckhoff's Lemma) and close in spirit to some of the technical arguments used by Chaika in \cite{Cha} to show that almost every pair of IETs is disjoint,  is given in Section~\ref{sec:resontanttimes2}.


\subsection{Final arguments in the proof of {spectral} disjointness}\label{sec:final_disj}
Let us now show that  Propositions \ref{prop:mix'} and \ref{prop:3} together imply the main {spectral} disjointness result for special flows (namely Theorem~\ref{thm:sfdisjointness}).

{
\begin{proof}[Proof of Theorem~\ref{thm:sfdisjointness}]
Let $T$ be in the
full measure sets of IETs for which the conclusion of Proposition \ref{prop:cohexistence} holds. Then there exists a sequence of rigidity towers $(\cC_\ell)_{\ell\in \mathbb{N}}$ with $|\cC_\ell|\to 1$ such that for every  $f\in \SymLog{T}$ the function $f^{\operatorname{p}}$ has linearly bounded trimmed derivatives along   $(\cC_\ell)_{\ell\in \mathbb{N}}$. By Proposition~\ref{lem:expfromderivatives} and the proof of Theorem~\ref{thm:dm_sf}, there exists  a sequence $(a_n)_{n\in\mathbb N}$  such that the sequence $(S_{h_n}f -a_n)_{n\geq 1}$ is tight with exponential tails, where $(h_n)_{n\geq 1}$ is the sequence of heights of the towers $\cC_n$, i.e.\ there are two positive constants $C$, $b$ such that
\[\lambda(t)=Leb \left( \left\{  x\in \cC_n  : |S_{h_n}f(x)-a_n | \geq t \right\} \right)\leq Ce^{-bt}.\]
Let $J_n$ be the base interval of the tower $\cC_n$. Then
\[\int_{J_n}|S_{h_n}f(x)-a_n|dx\leq t\lambda(t)+\int_t^{+\infty}\lambda(s)\,ds+t|J_n|\leq (Ct+C/b)e^{-bt}+t|J_n|\]
for any $t\geq 1$ and $n\geq 1$. Taking first the limit as $n\to+\infty$, and then as $t\to+\infty$, we obtain that
\[\lim_{n\to \infty}\int_{J_n}|S_{h_n}f(x)-a_n|dx=0.\]
As $|\cC_n|=h_n|J_n|\to 1$ and $\int_If(x)\,dx=1$, it follows that
\[1-\lim_{n\to\infty}\frac{a_n}{h_n}=1-\lim_{n\to\infty}a_n|J_n|=\lim_{n\to\infty}(\int_{\cC_n}f(x)\,dx-a_n|J_n|)=0.\]

Fix any increasing sequence $\underline{k}=(k_j)_{n\in\mathbb N}$ of natural numbers and take the sequence $\underline{\epsilon}=(\epsilon_j)_{j\in\mathbb N}$ given by $\epsilon_j:=k_j^{-1}$, so that $\epsilon_j\to 0$ as $j$ grows. Consider now the sequences $\underline{s}^j=(s^j_n)_{n\in\mathbb N}$ defined by
\begin{equation*}\label{def:sequence}
s_n^j:= C_d^{-2} h_n \epsilon_j=C_d^{-2} h_n k_j^{-1} , \qquad \forall\ n\in \mathbb{N},
\end{equation*}
where $C_d>1$ is a constant given by Proposition~6.4.
Let ${S}$ be an IET in the intersection of the full measure set of IETs $\mathcal{F}_d$ (for which the conclusion of Proposition~\ref{prop:mix'} holds) and
 the conclusion of Proposition \ref{prop:3} holds for all pairs  $(\underline{s}^j,\epsilon^2_j)$, $j\in\N$, i.e.\ $S\in \mathcal{D}_T:=\mathcal{F}\cap\bigcap_{j\geq 1}\Omega(\underline{s}^j,\epsilon_j^2)$.
By Proposition \ref{prop:3}, for any such $S$, for any $j\in\N$ there exist infinitely many  $\epsilon^2_j$-rigidity times $q$ for $S$  such that
\[
C_d^{-2}h_n k_j^{-1}=s^j_{n}\leq q   \leq  C_d s^j_{n}=C_d^{-1}h_n k_j^{-1} \quad \text{for some} \quad n\in \mathbb{N}.
\]
It follows that there exist increasing sequences of natural numbers $(n_j)_{j\in\N}$ and $(q_j)_{j\in\N}$ such that $q_j$ is an  $\epsilon^2_j$-rigidity time, $\log\log q_j\geq k_j$ and
\[C_dk_jq_j\leq h_{n_j}\leq C_d^{2}k_jq_j.\]
Hence, $(k_jq_j)_{j\in\mathbb N}$ is a sequence of $(\underline{\epsilon},\underline{k})$-resonant times. Since $a_n/h_n\to 1$, we also have
\[k_jq_j\leq a_{n_j}\leq 2C_d^{2}k_jq_j\]
for all $j$ large enough.

 Then by Proposition \ref{prop:mix'},  the sequence $(a_{n_j})_j$ (that is trapped between successive resonant rigid times with $k_j\to \infty$)  is a mixing subsequence for $(S_t^g)$ for any roof function $g\in \SymLogC{S}$.
Furthermore, (by the choice of $(a_n)_{n\in\mathbb N}$ at the beginning of the proof) we also have that $(S_{h_{n_j}}f  - a_{n_j})_{j\in\N}$ is tight with exponential tails. Thus, the assumptions of disjointness criterion (Theorem~\ref{thm:spectorth}) hold and the criterion implies that the special flows $T^f_t$ and $S^g_t$ are spectrally disjoint.
\end{proof}
}

The rest of this section will be devoted to the proof of Proposition~\ref{prop:mix'}. We will first present some auxiliary results, then give the proof in \S~\ref{sec:proof_resonantmixing}.

\subsection{Mixing via shearing criterion}\label{sec:mixingviashearing}
This Proposition uses the standard mixing mechanism called \emph{mixing via shearing} which is at the heart of virtually every proof of mixing in parabolic flows, see for example \cite{Ko:mix, Fa:ana, Ul:mix, Rav:mix}. We will use the following classical criterion that axiomatises the properties of a (partial) partition (into curves on which the shearing phenomenon happens and is controlled) needed to prove mixing via this mechanism.


Let us consider a special flow $(T^f_t)_{t\in\R}$ on $I^f$ built over an ergodic IET $T:I\to I$ and under a roof function $f
\in \mathcal{L}og(T)$.  For any $x\in I$ and $t>0$ denote by $N(x,t)$ the number of crossings of the roof of $I^f$ (or, equivalently, the number of discrete iterations of the base transformation) undergone by the orbit segment starting from $(x,0)\in I^f$ and flowing up to time $t${, i.e.
\begin{equation}\label{def:Nxt}
N(x,t):=\#\{s\in(0,t]:T^f_s(x,0)\in I\times\{0\}\}.
\end{equation}}
For any subinterval $J\subset I$, let us also introduce the notation
\begin{equation}\label{def:Nxtbar}
\underline{N}(J,t):= \min_{x\in J} N(x,t), \quad \textrm{and} \quad
\overline{N}(J,t):= \max_{x\in J} N(x,t)
\end{equation}
for the {minimum and maximum}  of the function $x \mapsto N(x,t)$ on the interval $J$ respectively.
\begin{lemma}[mixing via shearing criterion] \label{lemma:mixingviashearing}
Let $(T_t^f)_{t\in\R}$ be a special flow over an IET $T$ and with the roof function $f\in \SymLogC{T}$ of geometric origin. Assume that for a sequence $t_n\to \infty$ we have that there exists a sequence $(\cP_n)_{n\geq 1}$ of finite collections of disjoint intervals (called partitions) satisfying:
\begin{gather}
\tag{M}Leb(\cP_n)\to 1\quad\text{as}\quad n\to\infty;\label{M}\vspace{1.3mm}\\
\tag{D}T^jJ\cap End(T)=\emptyset\quad\text{for every}\ J\in \cP_n\ \text{and}\  0\leq j\leq \overline{N}(J,t_n);\label{D}\vspace{3mm}\\
\tag{S1} \min_{J\in \cP_n}\min_{r\in [\underline{N}(J,t_n),\overline{N}(J,t_n)]}\min_{x\in J}|S_r(f')(x)||J| \to \infty\text{ as }n\to \infty;\label{S1}\vspace{6mm} \\
\tag{S2} \max_{J\in \cP_n}\frac{\max_{0\leq r\leq 2t_n}\max_{x\in J} |S_{r}(f'')(x)||J|}{\min_{r\in [ \underline{N}(J,t_n),\overline{N}(J,t_n)]}\min_{x\in J}|S_r(f')(x)|}\to 0
\text{ as }n\to \infty. \vspace{1.3mm}\label{S2}
\end{gather}
Then $(T_t^f)_{t\in\R}$ is mixing along $(t_n)_{n\geq 1}$.
\end{lemma}
Even though our formulation is slightly different, this proposition has already been used in the literature, in particular by Ko\v{c}ergin, Fayad and Ulcigrai. Indeed, in \cite{Fa:ana}, the control in \eqref{S1} and \eqref{S2} is formally required for all $t_n/2\leq r \leq 2t_n$, although in the proof  it is only used for all $\underline{N}(J,t_n) \leq r \leq \overline{N}(J,t_n)$. We refer also to \cite{Ul:mix}, where, although the criterion is not explicitly stated as such, the proof of mixing of $(T_t^f)_{t\in\R}$  is deduced from the existence of partitions which satisfy exactly these assumptions (M), (D), (S1) and {(S2)}.


\subsection{Deviations of the number of discrete iterations in time $t$.}
 Fluctuations of the function $x\mapsto N(x,t)$ defined {in \eqref{def:Nxt} of} the previous \S~\ref{sec:mixingviashearing} can be reduced to the study of  ergodic integrals of the special flow (or, equivalently, of  locally Hamiltonian flows).  Ergodic integrals of smooth functions over translation flows are well known to display \emph{power deviations} of ergodic averages, a phenomenon discovered by Zorich (see \cite{Zo:dev}) and proved by Forni \cite{Fo:dev}. For locally Hamiltonian flows, a different proof (which uses a \emph{correction} method inspired by \cite{MMY:Coh}) which yields this result for smooth observables was proved in  \cite{Fr-Ul24}.

In  the proof of mixing at resonant times we will need
the following quantitative estimate on the values of $N(x,t)$ {given by \eqref{def:Nxt}}, which shows that, on sets of arbitrarily large measure of initial points, they display  power deviations from the mean:
\begin{lemma}\label{thm:crossingsbounds}
For {every} irreducible permutation $\pi \in \mathfrak{S}_d^0$
 there exists $0\leq \lambda_2<\lambda_1$
such that, for almost every IET $T$
 with permutation $\pi$ and {every roof function} $f\in \mathcal{L}og(T)$  with mean value equal $1$ and {every} $\frac{\lambda_2}{\lambda_1}<\gamma<1$, we have
\[\lim_{t\to+\infty}Leb\{x\in I: N(x,t)\in [t-t^\gamma, t+t^\gamma]\}=1.\]
\end{lemma}

\begin{remark}\label{rem:Lyapunov}
Here $\lambda_1$ and $ \lambda_2$  depend only on the \emph{Rauzy class}  $\mathcal{R}=\mathcal{R} (\pi)$ of the permutation $\pi$ (see \S~\ref{sec:parameterspaceIETs})  and are the two top Lyapunov exponents of the Kontsevich-Zorich cocycle on the connected component of the stratum of Abelian differentials corresponding to  $\mathcal{R} $.
\end{remark}
The estimate can be deduced by the recent work  \cite{Fr-Ul24} by the first and last author on deviation of ergodic averages of locally Hamiltonian flows by a simple smoothening argument, as we now explain.
\medskip


For any measurable bounded map $\xi:I^f\to\R$ let us consider $\varphi_\xi:I\to\R$ given by
\[\varphi_\xi(x)=\int_0^{f(x)}\xi(x,r)dr\quad\text{for}\quad x\in I.\]
The following  result 
 that can be deduced from the results in \cite{Fr-Ul24}.
\begin{proposition}\label{thm:asxi}
For any irreducible $\pi\in \mathfrak{S}^0_d$ and $0\leq \lambda_2<\lambda_1$ as in Lemma~\ref{thm:crossingsbounds} and Remark~\ref{rem:Lyapunov},  for almost every IET $T$
 with permutation $\pi$
and {every} $\xi:I^f\to\R$ measurable bounded map with zero mean such that
$\varphi_\xi\in \Log{T}$,
for a.e.\ $(x,r)\in I^f$, we have
\begin{equation}\label{eq:la2/la1}
\limsup_{t\to+\infty}\frac{\log|\int_{0}^t\xi(T^f_s(x,r))ds|}{\log t}\leq \frac{\lambda_2}{\lambda_1}<1.
\end{equation}
\end{proposition}
\noindent This result shows that indeed, for $(x,r)$ in a large measure  set, the ergodic integral $\int_{0}^t\xi(T^f_s(x,r))ds$ has oscillations of size bounded  by  $O(t^{\alpha})$ for any power $0<\alpha<1$ such that $\lambda_2/\lambda_1<\alpha$.
We include the proof of this Proposition deducing it from the results proven in \cite{Fr-Ul24} in the Appendix~\ref{app:deviations}.

Let us show that the above Proposition~\ref{thm:asxi} implies the desired control of oscillations of $N(x,t)$ by a simple \emph{thickening} argument.
\begin{proof}[Proof of Lemma~\ref{thm:crossingsbounds}]
Consider an IET $T$ in the full measure set of IETs with permutation $\pi$ for which Proposition~\ref{thm:asxi} holds.
Fix $0<\epsilon<\inf_{x\in I}f(x)$ and take $\xi = \xi_\epsilon:I^f\to\R$  to be the (normalized) indicator function of a thickened cross section, given by
$$\xi_\epsilon(x,y):=\frac{1}{\epsilon}\, \mathbf{1}_{I\times[0,\epsilon]}(x,y), \qquad (x,y)\in I^f.$$
Then for {all} $x\in I$ and $t>0$, we have
\[\int_{0}^t\xi(T^f_s(x,0))ds-1\leq N(x,t)\leq \int_{0}^t\xi(T^f_s(x,0))ds+1.\]
By Proposition~\ref{thm:asxi}, {applied to $\xi-1$}, for a.e.\ $x\in I$, we have
\[\limsup_{t\to+\infty}\frac{\log|\int_{0}^t(\xi(T^f_s(x,0))-1)ds|}{\log t}\leq \frac{\lambda_2}{\lambda_1}<\gamma.\]
It follows that
for a.e.\ $x\in I$, we have
\[\limsup_{t\to+\infty}\frac{\log|N(x,t)-t|}{\log t}<\gamma.\]
This shows our claim.
\end{proof}

\subsection{Estimates of the Birkhoff sums of the first derivative}\label{sec:Rothestimtes}
Let us first estimate Birkhoff sums of $S_r(f')$ of the \emph{derivative} $f'$ of the roof function. The goal of this subsection is to prove the following estimates. Let us recall that with the notation introduced in \S~\ref{sec:closests}, given the orbit segment $\mathcal{O}(x,r)$ of $x$ under $T$ of length $r$, we denote by $m(x,r)$  the  minimal distance of $\mathcal{O}(x,r)$ from the discontinuities of $T$.
\begin{lemma}\label{lem:upb}
For almost every IET $T$ and {every} $f\in \SymLogC{T}$, for {every} $\epsilon>0$ there is a constant $M'>0$ such that for {all} $r\in \N$ and $x\in [0,1)$, we have
\[
|S_r(f')(x)|\leq M' r^{1+3\epsilon} +\frac{M'}{m(x,r)}.
\]
\end{lemma}
\noindent The set of IETs for which this Lemma holds contains  the full measure set of IETs which are called \emph{Roth type} in the work by Marmi, Moussa and Yoccoz \cite{MMY:Coh}. We will not give here the definition of Roth IETs (which we do not use directly), but we recall one of the consequences of being Roth type, which is proved in \cite{MMY:Coh}.


\smallskip
Given an IET $T$ and a sequence  $(I^{(k)})_{k\in\mathbb N}$ of subintervals of $I=[0,1)$, recall that the first return map of $T$ to $I^{(k)}$ is again an IET, of the same number of intervals if the sequence is chosen properly (in particular if it is the sequence given by Rauzy-Veech induction, see \S~\ref{sec:RV}). We denote by $I^{(k)}_j$ the continuity intervals of the induced map and  by $\lambda^{(k)}_j=|I^{(k)}_j|$ their lengths. Let us denote by
 $h^{(k)}_j$ is the first return time of a (and hence {every}) point in $I^{(k)}_j$ to $I^{(k)}$ under $T$.

\begin{lemma}[Roth type distortion control]\label{lemma:Roth_distortion} If $T$ is Roth type, for every $0<\epsilon<1$  there exists a constant $M:=M(\epsilon)>0$
and  there exists an inducing sequence of nested intervals $( I^{(k)})_{k \in \mathbb{N}}$, such that for any $k\in \mathbb{N}$ the lengths $ \lambda^{(k)}_j$ and the return times $ h^{(k)}_j$ of the continuity intervals $I^{(k)}_j$ of the induced IET on $I^{(k)}$
satisfy the following controlled distortion bound:
$$
\max_j h^{(k+1)}_j \leq M (\min_j h^{(k)}_j)^{1+\epsilon}\;\; \text{ and }\;\; \max_j \lambda^{(k)}_j \leq M (\min_j \lambda_j^{(k+1)})^{1-\epsilon}.
$$
\end{lemma}
We include the proof of this Lemma deducing it from the results proven in \cite{MMY:Coh} in Appendix~\ref{app:Roth}.
Let us recall that one of the central results proved by Marmi-Moussa-Yoccoz in \cite{MMY:Coh} is that Roth-type IETs have full measure.
\begin{proof}[Proof of Lemma~\ref{lem:upb}]
Assume that $T$ belongs to the set of IETs which are of Roth type, which has full measure by  \cite{MMY:Coh}. Let $(I^{(k)})_{k\in\mathbb N}$ be the sequence of inducing intervals given by  Lemma~\ref{lemma:Roth_distortion} and denote as above by $\lambda^{(k)}_j$ and $h^{(k)}_j$, for each $1\leq j\leq d$ respectively, the length and the (constant) return time of the continuity interval $I^{(k)}_j$.

\smallskip
Choose  $k$ to be the smallest natural number such that $r< \min_j h^{(k)}_j$.

\textbf{Case 1.} First suppose that the points in the  orbit $\mathcal{O}(x,r)=\{T^ix\}_{i=0}^{r-1}$ belong to $r$ subsequent floors of the tower
over $I^{(k)}_j$ with height $h^{(k)}_j$, whose floors are  $\{T^i I^{(k)}_j:\ 0\leq i<h^{(k)}_j\}$.
 Then (since $T$ is an isometry), the points  $\mathcal{O}(x,r)$
  are $\lambda^{(k)}_j$-separated, i.e., the minimum distance between any two of them is bounded below by $\lambda^{(k)}_j$.
   Using classical Ko\v{c}ergin estimates based on the sum of inverses of an arithmetic progression, we have that if the points $(x_i)_{i=0}^N\subset [0,1]$ are $\delta$-separated for some $\delta>0$, then
\[
\sum_{i=0}^N \frac{1}{x_i} \leq \frac{1}{\min_{0\leq i\leq N} x_i} + \sum_{j=1}^N \frac{1}{j\, \delta}  \leq \frac{1}{\min_{0\leq i\leq N} x_i} + \frac{1+\log N}{\delta} .
\]
This shows that, {for $C_{sum}$ as in \eqref{def:Csum}, we have}
\[|S_r(f')(x)|\leq C_{sum}\frac{1+\log r}{\min_j \lambda^{(k)}_j} +\frac{C_{sum}}{m(x,r)}+r \Vert g'\Vert_\infty.\]

Choose
$j_k$ such that $\lambda^{k-1}_{j_k}h^{k-1}_{j_k}\geq 1/d$ (such $j_k$ always exists since there are $d$ towers whose total measure is $1$). Then,  by using both inequalities of Lemma~\ref{lemma:Roth_distortion} (for lengths first and heights later),  we have  that
\[\min_j \lambda^{(k)}_j \geq \Big(\frac{\lambda^{k-1}_{j_k}}{M}\Big)^{\frac{1}{1-\epsilon}}\geq \frac{1}{{(dM  h^{k-1}_{j_k})}^{\frac{1}{1-\epsilon}}}\geq
\frac{1}{M^{\frac{2}{1-\epsilon}} d^{\frac{1}{1-\epsilon}} {(\min_j h^{k-2}_j})^{\frac{1+\epsilon}{1-\epsilon}}}.\]
Since $r\geq \min_jh_j^{k-1}\geq \min_j h^{k-2}_j$,  {(choosing $\epsilon<1/3$ so that  $\frac{1+\epsilon}{1-\epsilon}<1+3\epsilon$)} we have   that
$$|S_r(f')(x)|\leq C_{sum} (dM^2)^{\frac{1}{1-\epsilon}}(1+\log r)r^{\frac{1+\epsilon}{1-\epsilon}} +\frac{C_{sum}}{m(x,r)}
+r \Vert g'\Vert_\infty\leq M_0  r^{1+3\epsilon}+\frac{M_0}{m(x,r)}$$
for a suitable constant $M_0>0$ {and $C_{sum}$ as in \eqref{def:Csum}}.

\medskip
\noindent
\textbf{Case 2.} Suppose that the orbit $\{T^ix\}_{i=0}^{r-1}$ is not contained in any tower $\{T^iI^{(k)}_{j}:0\leq i<h^{(k)}_{j}\}$ for $1\leq j\leq d$. Since by choice of $k$, $r\leq \min_j h^{(k)}_j$, it follows that $\mathcal{O}(x,r)$  can be decomposed into at most two towers over intervals of $I^{(k)}$, 
 to each of which we can apply the arguments from Case~1. This gives
\[|S_r(f')(x)|\leq 2M_0 r^{1+3\epsilon}+\frac{2M_0}{m(x,r)},\]
which, setting $M':=2M_0$, completes the proof.
%
\end{proof}

\subsection{Mixing at resonant times proof}\label{sec:proof_resonantmixing}
In this section we prove Proposition~\ref{prop:mix'}, i.e.~we show that for a full measure set of IETs, times comparable with resonant rigid times, i.e.~times of the form $k_m q_m$, where $q_m$ is a rigidity time for the IET and $k_m$ sufficiently large, are mixing times for logarithmic special flows over the IET.

\subsubsection{The full measure set  of IETs}\label{sec:fullmeasure}
 Denote by $\mathcal{F}^0_d\subset \mathcal{I}_d$ the union over permutations $\pi\in \frak{S}^0_d$ of the full measure sets of IETs with permutation $\pi $ 
for which Lemma~\ref{thm:crossingsbounds} and  Proposition~\ref{thm:asxi} controlling fluctuations of {the function $N(x,t)$  given by \eqref{def:Nxt}} and of ergodic integrals hold.   By construction,  $\mathcal{F}^0_d$ has full measure in $ \mathcal{I}_d$.
The full measure set $\mathcal{F}_d$ of IETs in $\mathcal{I}_d$ for which we will prove Proposition~\ref{prop:mix'} is given by intersecting $\mathcal{F}^0_d$ with the full measure set  of $\mathcal{I}_d$  (containing Roth type IETs) for which Lemma~\ref{lem:upb} holds.

\subsubsection{Choices of parameters and notation}\label{sec:notation}
 If for each $\pi\in \mathfrak{S}^0_d$ we denote by $\lambda_1(\pi), \lambda_2 (\pi)$ the exponents in Lemma~\ref{thm:crossingsbounds} and Proposition~\ref{thm:asxi} (namely the Lyapunov exponents of the Rauzy class $\mathcal{R}$ containing $\pi$, see Remark~\ref{rem:Lyapunov}),  since   $\lambda_2(\pi)/ \lambda_1 (\pi)<1$ for any  $\pi \in \mathfrak{S}^0_d$, we can choose  $0<\gamma<1$ and $\epsilon>0$ such that for each $\pi \in \mathfrak{S}^0_d$
$$ \max_{\pi \in \mathfrak{S}^0_d}\lambda_2(\pi) /\lambda_1 (\pi)<\gamma<1, \qquad  0<\epsilon<\gamma/6, \qquad \gamma(1+3\epsilon)<1. $$
We will for simplicity assume throughout this section (without loss of generality) that $\int_0^1 f(x)dx=1$. We will also denote by $\delta_m$
\begin{equation}\delta_m:=\frac{1}{\log k_m}\to 0,\end{equation}
where {$(k_mq_m)_{m\in\mathbb N}$ is a sequence of $(\underline{\epsilon},\underline{k})$-resonant times for $T$.}
By definition, for every $m\geq 1$ there exists an $\epsilon^2_m$-rigid tower ($\epsilon_m\leq 1/k_m$) of intervals $\mathcal{T}_m$ such that $q_m$ is its height and $Leb(\mathcal{T}_m)\geq 1-\epsilon_m$.
Denote the base of $\mathcal{T}_m$ by $I_m=[a_m,b_m]$.

\subsubsection{Variation of $N(x,t_m)$ in the large tower.}
For the proof of the proposition we will first need a lemma which gives good control on $N(x,t_m)$ {(as defined by  \eqref{def:Nxt})} for most of the points in the tower $\mathcal{T}_m$:
\begin{lemma}[controlled deviations floors]\label{lem:cont}
Assume that the roof function $f$ belongs to $\SymLogC{T}$.
Suppose that $t\in [k_mq_m, Ck_mq_m]$.
Let $J$ be a floor of the tower
$$\cC'_m=\bigcup_{j=0}^{q_m-1} T^j\left[a_m+ \frac{\delta_m}{q_m}, b_m- \frac{\delta_m}{q_m}\right]$$ for which we have that $N(x_0,t)\in [t-t^{\gamma},t+t^{\gamma}]$ for some $x_0\in J$. Then
\begin{equation}\label{eq:goodcontrol}
N(x,t)\in [t-2t^{\gamma},t+2t^{\gamma}]\quad \text{for\ all} \ x\in J.
\end{equation}
\end{lemma}
\noindent The rest of this subsection is devoted to the proof of the Lemma.
\begin{proof}[Proof of Lemma~\ref{lem:cont}]
Let $c:=\inf_{x\in[0,1)}f(x)>0$. By definition,
\begin{align*}
c|N(x,t)-N(x_0,t)| \leq |S_{N(x,t)-N(x_0,t)}(f)(T^{N(x_0,t)}x)|= |S_{N(x,t)}(f)(x)-S_{N(x_0,t)}(f)(x)|.
\end{align*}
By the definition of $N$, for {all} $y\in[0,1)$,  we have
\[0\leq t-S_{N(y,t)}f(y)<f(T^{N(y,t)}y).\]
It follows that
\[|S_{N(x,t)}f(x)-S_{N(x_0,t)}f(x_0)|\leq f(T^{N(x,t)}(x))+f(T^{N(x_0,t)}(x_0)).\]
Therefore,
\begin{equation}\label{eq:N1}
c|N(x,t)-N(x_0,t)| \leq f(T^{N(x,t)}(x))+f(T^{N(x_0,t)}(x_0))+|S_{N(x_0,t)}(f)(x_0)-S_{N(x_0,t)}(f)(x)|.
\end{equation}
Note that by the rough bound $N(y,t)\leq c^{-1} t\leq c^{-1}Ck_mq_m$ for {all} $y\in J$. Moreover, as the tower $\mathcal{T}_m$ is $\epsilon^2_m$-rigid and $\epsilon_m\leq \frac{1}{k_m}$, for {all} $y\in J$ we have $|T^{q_m}y-y|<\frac{\epsilon^2_m}{q_m}\leq \frac{1}{k^2_mq_m}$. Repeating it $\frac{C}{c}k_m$-times,
we have that for {all} $0\leq N\leq \frac{C}{c}k_m{q_m}$, $T^Ny$ is in the tower $\mathcal{T}_m$ and at least $\frac{\delta_m}{q_m}-\frac{C}{ck_mq_m}$ distant from the ends of the levels of $\mathcal{T}_m$. Since $\delta_m=\frac{1}{\log k_m}$, we have  $\frac{\delta_m}{q_m}-\frac{C}{ck_mq_m}\geq \frac{\delta_m}{2q_m}$.  It follows that, {for $C_{sum}$ as in \eqref{def:Csum}},
\begin{align}\label{eq:N2}
\begin{aligned}
f(T^{N}(y))&\leq  C_{sum} \log(\frac{2q_m}{\delta_m})+\Vert g\Vert_{sup}\leq   C_{sum} \log(2q_m\log k_m)+\Vert g\Vert_{sup}\\
&< C_{sum}\log(k_mq_m)\leq  C_{sum}\log t
\end{aligned}
\end{align}
uniformly for $y\in J$. Moreover, by Lemma~\ref{lem:upb} (which we can apply since we are assuming that $T$ belongs to the set $ \mathcal{F}_d$ for which the conclusion of the Lemma applies, see \S~\ref{sec:fullmeasure}),
\begin{align*}
|S_{N(x_0,t)}(f)(x_0)-S_{N(x_0,t)}(f)(x)|&\leq |S_{N(x_0,t)}(f')(\theta)| q_m^{-1}
\leq \frac{M'}{q_m}\left((c^{-1}Ck_mq_m)^{1+3\epsilon}+\frac{2q_m}{\delta_m}\right) \\
&\leq C'((k_mq_m)^{3\epsilon}\log\log q_m+\log k_m)<(k_mq_m)^{\gamma/2}\leq t^{\gamma/2} ,
\end{align*}
where  we use $\epsilon$-Roth type, $3\epsilon<\gamma/2$, the fact that $\theta\in J$ and $k_m\leq \log\log q_m$.
Finally, in view of \eqref{eq:N1} and \eqref{eq:N2}, this gives
\begin{align*}
|N(x,t)-N(x_0,t)| &\leq c^{-1}(f(T^{N(x,t)}(x))+f(T^{N(x_0,t)}(x_0))+|S_{N(x_0,t)}(f)(x_0)-S_{N(x_0,t)}(f)(x)|)\\
&\leq 2C_{sum}c^{-1}\log t+c^{-1}t^{\gamma/2}<t^\gamma.
\end{align*}
 This finishes the proof.
\end{proof}

\subsubsection{Estimates of Birkhoff sums of the second derivative.}
We will also need one more lemma which gives good control for  the Birkhoff sums of the second (and hence the first) derivative.

\begin{lemma}\label{lem:der}
Assume that the roof function $f\in \SymLogC{T}$ is of geometric origin.
Let $\ell\in [\frac{1}{2}k_mq_m,2Ck_mq_m]$. Then for {every} $x\in \cT'_m$,
\begin{equation}\label{eq:sec}
\underline{c}k_mq_m^2\leq S_\ell(f'')(x)\leq \overline{c}k_m\delta_m^{-2}q_m^2,
\end{equation}
where $\underline{c}:={C_{min}/16}>0$ and $\overline{c}:=20CC_{sum}$ with
\[
{C_{min}:=\min\{C^\pm_i:C^\pm_i>0\}.}\]
\end{lemma}
\begin{proof}
First note that if the points $(x_i)_{i=0}^N\subset [0,1]$ are $\delta$-separated for some $\delta>0$, 
 then
\[
\frac{1}{\min_{0\leq i\leq N} x^2_i}\leq \sum_{i=0}^N \frac{1}{x^2_i} \leq \frac{1}{\min_{0\leq i\leq N} x^2_i} + \sum_{j=1}^N \frac{1}{j^2\, \delta^2}  \leq \frac{1}{\min_{0\leq i\leq N} x^2_i} + \frac{2}{\delta^2} .
\]
As $f$ is a roof function, belongs to $\SymLogC{T}$, and is of geometric origin,
it follows that for {every} $x\in [0,1)$ and {every} natural $r$ if the orbit $\{T^ix\}_{i=0}^{r-1}$ is $\delta$-separated then
\begin{equation}\label{eq:Srddf}
{\frac{C_{min}}{m_e(x,r)^2}}-r\Vert g''\Vert_{sup}\leq S_r(f'')(x)\leq \frac{C_{sum}}{m(x,r)^2}+\frac{2C_{sum}}{\delta^2}+r\Vert g''\Vert_{sup},
\end{equation}
where $m(x,r)$ is the closest visit {to all singularities} (see \S~\ref{sec:closests}) {and $m_e(x,r)$ is the closest visit to essential singularities defined as follows.
Due to the geometric origin of the roof function $f$,
 there are exactly two zero coefficients, namely $C^+_{i_+}$ and $C^-_{i_-}$ for some $0\leq i_+<d$ and $1\leq i_-\leq d$. The closest visit to the essential singularities is then defined by
\begin{align*}
m_e(x,r):&= \min \big\{ \{ m_i^{+}(x,r): \ 0\leq i< d ,\ i\neq i_+ \} \cup  \{ m_i^{-}(x,r): \    1\leq i\leq d,\ i\neq i_- \}\big\}\\
&= \min \big\{ \{ m_i^{+}(x,r): \ 0\leq i< d ,\ C^+_i>0 \} \cup  \{ m_i^{-}(x,r): \    1\leq i\leq d,\ C^-_i>0 \}\big\}.
\end{align*}
Let $i_0:=\pi^{-1}(1)-1$ and $i_d:=\pi^{-1}(d)$, that is, $T\beta{i_0}=\beta_0$ and $\beta{i_d}$ is the right endpoint of the interval translated by $T$ to the last exchanged interval, whose right endpoint is $\beta_d$. Then $i_+\in\{0,i_0\}$ and $i_-\in\{d,i_d\}$.

Suppose that $x$ is such that its orbit $\mathcal{O}_T(x,q_m)\subset \cT_m$. Since $T\beta_{i_0}=\beta_0$, we obtain
\[
|m^+_{i_0}(x,q_m)-m^+_0(x,q_m)|\leq |T^{q_m}x-x|\leq \frac{\epsilon^2}{q_m}\leq\frac{1}{q_m}.\]
A similar argument shows also that $|m^-_{i_d}(x,q_m)-m^-_d(x,q_m)|\leq 1/q_m$. Since $C^+_0$ or $C^+_{i_0}$ is positive, and $C^-_d$ or $C^-_{i_d}$ is positive, it follows that
\[m_e(x,q_m)\leq m(x,q_m)+\frac{1}{q_m}\leq \frac{2}{q_m}.\]}

\smallskip

\noindent {\it Lower bound.}
As $f''\geq 0$,
the lower bound follows from the fact that
\[S_\ell(f'')(x)\geq S_{[\frac{k_m}{2}]q_m}(f'')(x)= \sum_{j=0}^{[\frac{k_m}{2}]-1}S_{q_m}(f'')(T^{jq_m}x)\]
and {$\mathcal{O}_T(T^{jq_m}x,q_m)\subset\mathcal{T}_m$} for {all} $0\leq j<k_m/2$. Then {$m_e(T^{jq_m}x, q_m)\leq \frac{2}{q_m}$} and, by \eqref{eq:Srddf}, we have
\[
{S_{q_m}(f'')(T^{jq_m}x)\geq \frac{C_{min}}{4} q_m^2-\Vert g''\Vert_{sup}q_m\geq 2\underline{c}q^2_m,}
\]
which gives the lower bound.


\smallskip
\noindent {\it Upper bound.}
As $f''\geq 0$, for the upper bound we similarly write
\[S_\ell(f'')(x)\leq S_{(2Ck_m+1)q_m}(f'')(x_0)= \sum_{j=0}^{2Ck_m}S_{q_m}(f'')(T^{jq_m}x_0),\]
where $x_0$ is the first backward visit (to the base of $\cT'_m$) point of the orbit of $x$.
We then use the fact that points $\{T^i(T^{jq_m}x_0)\}_{0\leq i<q_m}$ are $\frac{1}{2q_m}$ separated and by the assumption $x\in \cT'_m$ and the rigidity condition it follows that $m(T^{jq_m}x_0,q_m)\geq\frac{\delta_m}{2q_m}$ for {all}  $0\leq j\leq 2Ck_m$. In view of \eqref{eq:Srddf}, this gives
\begin{align*}
S_\ell(f'')(x)&\leq\sum_{j=0}^{2Ck_m}S_{q_m}(f'')(T^{jq_m}x_0)\\&
\leq 3Ck_m\left(4C_{sum}q^2_m \delta^{-2}_m+8C_{sum}q_m^2+\|g''\|_{sup}q_m\right)\leq \overline{c}k_m\delta_m^{-2}q_m^2.
\end{align*}
 This finishes the proof.
\end{proof}

\subsection{Construction of mixing partitions}\label{sec:proof_resonantmixing1}
We can now conclude the proof that resonant rigid times are mixing (i.e.~prove Proposition \ref{prop:mix'}), using the estimates proved and the \emph{mixing via shearing} criterion (i.e.~of Lemma~\ref{lemma:mixingviashearing}).
\begin{proof}[Proof of Proposition \ref{prop:mix'}]
 Let us assume that $T\in \mathcal{I}_d$ belongs to the full measure set $\mathcal{F}_d$ defined in \S~\ref{sec:fullmeasure}.
Let us define a partition of $[0,1]$ for which we can verify the assumptions of the \emph{mixing via shearing}  (Lemma~\ref{lemma:mixingviashearing}) criterion.
 For this, given $t_m\in [k_m q_m, Ck_mq_m]$ let $\cP'_m$ be the partition  into floors of the tower $\mathcal{T}'_m$, namely into intervals  $\{T^j(I'_m)\}_{0\leq j<q_m}$, where $I'_m$ is the base of $\cC'_m$. We will only consider the floors  $F_j^m:= T^j(I'_m)$ which have a non-empty intersection with the set given by Lemma~\ref{thm:crossingsbounds}, namely such that
\begin{equation}\label{eq:good}
F_j^m \cap \{x\in [0,1)\;:\; N(x,t_m)\in [t_m-t_m^{\gamma}, t_m+t_m^\gamma]\}\neq \emptyset.\end{equation}
We will call these floors \emph{good} and we denote by  $\mathcal{T}''_m$  the partial partitions obtained keeping only these good intervals. We claim that good intervals constitute the majority of the intervals, or in other words,   $Leb(\cT''_m)\to 1$ as $m\to \infty$.
 Indeed, since by construction the union of all (good) intervals in $\mathcal{T}''_m$ covers the set $\{x\in [0,1)\;:\; N(x,t_m)\in [t_m-t_m^{\gamma}, t_m+t_m^\gamma]\}$,
this follows
from the estimate given  by  Lemma~\ref{thm:crossingsbounds}.  
   Note that then by Lemma~\ref{lem:cont}, if a floor $F^m_j$ is good, then for all $x\in F^m_j$ we have a similar control of $N(x,t_m)$, see  \eqref{eq:goodcontrol}.
Moreover, by the construction of $I'_m$, it follows that  for {every} interval $J\in\cP''_m$, no discontinuity lies inside $T^j J$ for $j\leq 2Ck_mq_m$. This together with the bound
\[\overline{N}(J,t_m):= \max_{x\in  J} N(x,t_m)\leq 2t_m\leq 2Ck_mq_m,\]
gives the condition \eqref{D}.
Our partition $\cP_m$ will be a refinement of the partition $\cP''_m$ and note that the property \eqref{D} is valid also for any such refinement.
By Lemma~\ref{lem:cont}, the good floors not only intersect, but are fully contained in the above sets. We now consider three separate cases, { that we define using the quantities $\underline{N}(x,t)$ and $\overline{N}(x,t)$ given as in  \eqref{def:Nxtbar}.}

\medskip
\noindent \textbf{Case 0.} Suppose that $J$ is a subinterval of an atom of $\cP'_m$ such that {its length is in} $ [\frac{\delta_m^4}{q_m},\frac{2\delta_m^4}{q_m}]$ and
\[\min_{r\in [ \underline{N}(J,t_m),\overline{N}(J,t_m)]}\min_{x\in J}|S_r(f')(x)|\geq \overline{c}\delta_m k_mq_m.\]
Then, we have
\[\min_{r\in [\underline{N}(J,t_m),\overline{N}(J,t_m)]}\min_{x\in J}|S_r(f')(x)| \cdot |J|\geq \overline{c}\delta_m^5 k_m\to \infty,\]
as $\delta_m=(\log k_m)^{-1}$. This gives \eqref{S1}.  Moreover, by Lemma \ref{lem:der},
\[
 \frac{\max_{r\leq 2t_m}\max_{x\in J} |S_{r}(f'')(x)||J|}{\min_{r\in [ \underline{N}(J,t_m),\overline{N}(J,t_m)]}\min_{x\in J}|S_r(f')(x)|}\leq \frac{\overline{c}k_m\delta_m^{-2}q_m^2\cdot \delta_m^4 q_m^{-1}}{\overline{c}\delta_m k_mq_m}=\delta_m\to 0,
\]
and so \eqref{S2} also holds.

\smallskip
\noindent In the remaining cases,  we will  construct  ${\cP_m}$ as a refinement of ${\cP''_m}$.
Let $J$ be an atom of the partition ${\cP''_m}$. We will consider two cases depending on the size of the first derivative on $J$:

\medskip
\noindent
\textbf{Case I:} Suppose now that
$$\min_{r\in [\underline{N}(J,t_m),\overline{N}(J,t_m)]}\min_{x\in J}|S_r(f')(x)|\geq \overline{c}\delta_m k_mq_m.$$ In this case we further divide $J$ (in an arbitrary way) into intervals $\{J_i\}$ of {length in} $ [\frac{\delta_m^4}{q_m},\frac{2\delta_m^4}{q_m}]$.
In this case the atoms of the partition $\cP_m$ coming from $J$ will be the intervals $\{J_i\}$.
In view of Case~0, all these new atoms satisfy \eqref{S1} and \eqref{S2}.


\medskip
\noindent
\textbf{Case II:}  Suppose finally that
$$\min_{r\in [\underline{N}(J,t_m),\overline{N}(J,t_m)]}\min_{x\in J}|S_r(f')(x)|\leq \overline{c}\delta_m k_mq_m.$$
\noindent Let $r_J\in  [\underline{N}(J,t_m),\overline{N}(J,t_m)]$ and $x_J\in J$ be such that
\[\min_{r\in [\underline{N}(J,t_m),\overline{N}(J,t_m)]}\min_{x\in J}|S_r(f')(x)|=|S_{r_J}(f')(x_J)|\leq \overline{c}\delta_m k_mq_m.\]
Take any $r\in  [\underline{N}(J,t_m),\overline{N}(J,t_m)]$. By the definition of $\cP''_m$ and $t_m\leq Ck_mq_m$, we have $|r-r_J|<4t_m^\gamma\leq 4C(k_mq_m)^\gamma$. Then, by  the fact that $T$ is $\epsilon$-Roth  and using Lemma~\ref{lem:upb}, we get
\[|S_{r-r_J}{f'}(T^{r_J}x_J)|\leq M'|r-r_J|^{1+3\epsilon}+M'\frac{2q_m}{\delta_m}\leq 16C^2M'(k_mq_m)^{\gamma(1+3\epsilon)}+2M'q_m\log k_m.\]
As $\gamma(1+3\epsilon)<1$, this gives
\[|S_{r-r_J}{f'}(T^{r_J}x_J)|\leq \overline{c} \frac{k_m}{\log k_m}q_m=\overline{c}\delta_m k_mq_m,\]
so
\[|S_{r}(f')(x_J)|\leq |S_{r_J}(f')(x_J)|+|S_{r-r_J}{f'}(T^{r_J}x_J)|\leq 2\overline{c}\delta_m k_mq_m.
\]
We remove from $J$ the $\frac{3\overline{c}\delta_m}{\underline{c}q_m}$ neighborhood  of $x_J$. We denote the remaining intervals in $J$ by $J_1$ and $J_2$. The following holds for both $\bar{J}= J_1, J_2$. Using Taylor expansion and lower bounds on second derivative in Lemma \ref{lem:der}, we get that for {every} $x\in \bar{J}$ and {every}
\[r\in [\underline{N}(\bar{J},t_m),\overline{N}(\bar{J},t_m)]\subset[t_m/2,2t_m]\subset[\frac{1}{2}k_mq_m,2Ck_mq_m],\]
we have
\begin{align*}
|S_r(f')(x)|&= |S_r{f'}(x_J)+S_r(f'')(\theta)(x-x_J)|\geq  |S_r(f'')(\theta)||x-x_J|-|S_r{f'}(x_J)|\\
&\geq \underline{c}k_mq_m^2\frac{3\overline{c}\delta_m}{\underline{c}q_m}  -2\overline{c}\delta_m k_mq_m = \overline{c}k_mq_m \delta_m.
\end{align*}
In this case we further divide $\bar J$ (in an arbitrary way) into intervals $\{\bar J_i\}$ of {length in} $ [\frac{\delta_m^4}{q_m},\frac{2\delta_m^4}{q_m}]$.
In this case the atoms of the partition $\cP_m$ coming from $\bar J$ will be the intervals $\{\bar J_i\}$.
In view of Case~0, all these new atoms satisfy \eqref{S1} and \eqref{S2}.


\medskip
\noindent
{\it Final estimate.} In both cases (Case I and Case II), for the atoms of the partition $\cP_m$ that we defined in each case, we get
\[Leb(\cP_m)\geq Leb(\cP''_m)- q_m\frac{6\overline{c}\delta_m}{\underline{c}q_m}=Leb(\cP''_m)- \frac{6\overline{c}}{\underline{c}}\delta_m\to 1\text{ as }m\to\infty.\]
Therefore, the condition \eqref{M}  of the \emph{mixing via shearing} criterion (Lemma~\ref{lemma:mixingviashearing}) holds. Thus we can apply the criterion, which concludes the proof that resonant times are mixing.
\end{proof}

\section{Properties of Rauzy-Veech induction}\label{sec:backgroundRV}
To prove the two
technical propositions that we left to prove (Proposition~\ref{prop:cohexistence} to conclude the proof of singularity of the spectrum and Proposition~\ref{prop:3} to conclude the proof of {spectral} disjointness),   we will use as a tool Rauzy-Veech induction and exploit some of the stochastic properties that the induction features.
In \S~\ref{sec:noninvertibleRV} and \S~\ref{sec:invertibleRV}, we  first recall  some basic terminology and  features of Rauzy-Veech induction and  its invertible extension. In \S~\ref{sec:statisticalRV}, we  then recall some features  which are at the heart of some of the statistical and ergodic properties  of the induction (in particular, the \emph{local product structure} given by the coordinates $\lambda$ and $\tau$ in \S~\ref{sec:localproductstructure} and  the Markovian structure of the partitions given by prescribing finitely many iterates types, see  \S~\ref{sec:Markov}).
Finally in \S~\ref{sec:distortion_background}, we recall an estimate known as \emph{Kerckhoff's lemma} and some of its consequences on distortion properties of Markov partitions.

\subsection{Rauzy-Veech induction for IETs}\label{sec:noninvertibleRV}
The Rauzy-Veech algorithm, by now a classical and powerful tool to study IETs,  
 was originally introduced and developed in the works by Rauzy  and Veech \cite{Rau, Ve:gau}. We recall here only some basic definitions 
needed in the rest of this paper.  For more details on the definition of the algorithm as well as a different notation for permutations (using an \emph{alphabet} $\mathcal{A}$ and a \emph{pair} of permutations to keep track of the \emph{labels} of intervals, which is important for some of the properties of the induction), we refer e.g.\ to  lecture notes by Yoccoz \cite{Yo} or Viana \cite{Vi}.

\subsubsection{Basic algorithm on IETs.} \label{sec:basic}
The Rauzy-Veech induction algorithm associates to 
a.e.~IET, 
 a sequence of IETs  which are induced {transformations} of $T$ onto a sequence of nested subintervals $I^{(n)}$ contained in $[0,1]$. These are chosen so that the induced {transformations} are again  IETs of the same number $d$ of exchanged intervals.
 If $I'\subset I:=[0,1]$ denotes the subinterval associated to one step of the algorithm, $I'$ is defined as follows: let $I^{t}_d:=I_d =[\beta_{d-1},1]$ (where $t$ stands for \emph{top} last) be the last (i.e.~right-most) exchanged subinterval and let $j_d:=\pi^{-1}(d)$ be the index such that $T(I_{j_d})$ has $1$ as endpoint, i.e.~$I_d':=I_{j_d}$ is the interval which becomes last after the exchange; equivalently, $I^b_d:=T(I_{j_d})$ is the last (i.e.~rightmost) exchanged interval for the inverse $T^{-1}$. Then, we set:
 $$
I':= \begin{cases}
[0,1]\backslash I^b_d, & \text{if}\ |I^t_d|>|I^b_d|\   \qquad
\text{(case\ top\ or\ type\ 0),}
\\ [0,1]\backslash I^t_d, &
\text{if}\ |I^b_d|>|I^t_d|   \qquad
\text{(case\ bottom\ or\ type\ 1).}
\end{cases}
 $$
 The \emph{Rauzy-Veech map}  $\V $ then associates to $T=(\pi,\lambda)$ such that $\lambda_d\neq \lambda_{j_d}$ (so that $|I^t_d|\neq |I^b_d|$)
  the IET $\V (T)=\V (\pi,\lambda)$  obtained by renormalizing $T'$ by $|I'|$ so that the renormalized IET is again defined on a unit interval. If $|I^t_d|=|I^b_d|$, the induction is not defined.

\subsubsection{Parameter spaces of the Rauzy-Veech map.}\label{sec:parameterspaceIETs}
Interval exchanges $T=(\pi, \lambda)$ on the unit interval $I=[0,1]$ can be parametrized by the permutation $\pi$ and a length vector $\lambda$ which belongs to the simplex
\be \Delta_{d} := \left\{(\lambda_1,\dots  \lambda_d)\ \Big| \quad  \lambda_i\geq 0, {\sum_{i=1}^{d} \lambda_i}=1\right\}.\ee
\noindent  Recall that  a \emph{Rauzy class} is a non-empty minimal (with respect to inclusion) subset $\mathcal {R}\subset\mathfrak{S}_d^0$
of irreducible permutations  (see \S~\ref{subsec:IETs} for the definition of irreducible)
such that $\Delta_d\times\mathcal{R}$ is invariant under the Rauzy-Veech map $\mathcal{V}$.  For any irreducible $\pi_0$, we denote by $\mathcal{R}(\pi_0)$ the Rauzy class containing $\pi_0$. Thus,  $\mathcal{R}(\pi_0) $
is the subset  of all permutations $\pi'$ 
which appear as permutations of an IET $T'=(\pi',{\lambda}')$ in the orbit under $\V$ of some IET $(\pi_0,{\lambda})$  with initial permutation $\pi_0$. The associated Rauzy-\emph{graph}, also denoted by  $\mathcal{R}(\pi_0)$, has these permutations as vertices, with an arrow $\gamma$ starting at $\pi$ and ending at $\pi'$ and  labeled top, or $0$ (resp.~bottom or $1$) iff there is a $\lambda$ of type $0$ (resp.~type $1$) such that  $\mathcal{V}(\pi,{\lambda})=(\pi',{\lambda}')$.

The domain of definition of the Rauzy-Veech map $\V$ is a full Lebesgue measure subset of the space
\begin{equation}\label{eq:defX}
X:= X(\mathcal R) =\mathcal R\times  \Delta_{d}  = \{ (\pi, \lambda)\ | \ \pi \in \mathcal R, \lambda \in \Delta_d \},
\end{equation}
corresponding to those IETs for which the algorithm is defined for all $n\in \N$, more precisely, the full measure set of IETs which satisfy the Keane condition.
{Recall that an IET $T$ satisfies the {\em Keane condition} if (recalling that $End(T)$ denotes  the set of end points of the IET $T$, see \S~\ref{subsec:IETs}),
  $T^m
\beta_{i}\neq {\beta}_j$ for all $m\geq 1$ and for all $i,j\in \{0,\dots, d-1\}$ except $j=0$ with $m=1$.
Keane \cite{Keane} showed that an  IET with an irreducible permutation that satisfies the Keane condition is \emph{minimal}.}
Veech proved in \cite{Ve:gau} that $\V$ admits an invariant measure $\mu_{\V}$ on $X(\mathcal{R})$ which 
 is equivalent to the Lebesgue measure, but infinite.
Zorich showed in \cite{Zo:fin} that one can accelerate\footnote{The acceleration of a map is obtained a.e.-defining  an integer valued function $z(T)$ which gives the return time to an appropriate section. The accelerated map is then given by $\mathcal{Z} (T) := \V^{z(T)} (T )$.} the map  $\V $ in order to obtain a map $\mathcal{Z}$, which we call \emph{Zorich map}, that admits a  \emph{finite} invariant measure $\mu_{\mathcal{Z}}$. Let us also recall that both $\V$ and its acceleration $\Z$ are \emph{ergodic} 
 with respect to $\mu_{\V}$ and $\mu_{\Z}$ respectively \cite{Ve:gau, Zo:fin}.

\subsubsection{Matrix products indexed by paths}\label{sec:products}
Given an IET $T=(\pi,\lambda)$ in the domain of $\V$, write $\mathcal{V}(\pi,\lambda)=(\pi',\lambda')$, let $\epsilon\in\{0,1\}$ be the \emph{type} of $(\pi,\lambda)$ (see \S~\ref{sec:basic}), then denote by $\gamma$ the corresponding arrow in the Rauzy graph $\mathcal R$ starting from $\pi$ and ending at $\pi'$. Then
$$\lambda'=\frac{(B^*_\gamma)^{-1}\lambda}{|(B^*_\gamma)^{-1}\lambda|}, \quad \text{where}\  B_\gamma\in SL(d,\mathbb Z)$$
is a non-negative matrix with entries $0$ or $1$ and with $1$ on the diagonal and only one $1$ off the diagonal and $B^*$ denotes the transpose of the matrix $B$. If $\gamma = \gamma_1\gamma_2\ldots\gamma_n$ is a path  in the Rauzy graph $\mathcal R$ which starts at $\pi_0$ and ends at $\pi_n$ obtained concatenating the arrows $\gamma_i$,  %
 associate to  $\gamma$
 the product matrix  $B_\gamma=B_{\gamma_n} B_{\gamma_{n-1}}\cdots B_{\gamma_1}$. Then

\begin{equation}\label{eq:Bgamma}
\mathcal V^n\left(\pi_0,\frac{B^*_\gamma\lambda}{\left| B^*_\gamma\lambda\right|}\right)=(\pi_n,\lambda) 
,\quad \text{where}\ B^*_\gamma=(B_{\gamma_n} B_{\gamma_{n-1}} \cdots B_{\gamma_1} )^*=B^*_{\gamma_1} \cdots  B^*_{\gamma_{n-1}} B^*_{\gamma_n}.\end{equation}
\noindent 
If $\gamma = \gamma_1\gamma_2\ldots\gamma_n$ is a path  in the Rauzy graph $\mathcal R$ encoding the first $n$ steps of the Rauzy-Veech induction for an IET $(\pi_0,\lambda)$, then
\begin{equation}\label{eq:Bgamma1}
\mathcal V^n(\pi_0,\lambda)=\left(\pi_n,\frac{(B^*_\gamma)^{-1}\lambda}{\left| (B^*_\gamma)^{-1}\lambda\right|}\right).
\end{equation}

\subsubsection{Length-cylinder simplices}\label{sec:lambdacylinder}
Given a non-negative matrix $B$ in $SL(d,\mathbb Z)$, let us denote by 
\begin{equation}\label{def:simplex}
\Delta_B:= B^* \Delta_d= \left\{ \frac{B^* \lambda}{|B^* \lambda|}, \quad \lambda \in \Delta_d\right\}\subset \Delta_d
\end{equation}
the simplex image of $\Delta_d$ by the projective action of $B$. For short, if $B=B_\gamma$ for some path $\gamma$ in the Rauzy graph starting at $\pi$, we write
\begin{equation}\label{eq:Delta_gamma}
 \Delta_\gamma:=\{\pi\}\times\Delta_{B_\gamma} =
 \{\pi\}\times B^*_\gamma(\Delta_{d}).
 \end{equation}
Then one can show that if $(\pi,\lambda)\in \Delta_\gamma$, then the types of the first $n$ steps of the Rauzy-Veech induction are given by  $\gamma_1,\ldots,\gamma_n$ and
$B_\gamma^*$
 is given by \eqref{eq:Bgamma}.

\subsection{Invertible Rauzy-Veech induction}\label{sec:invertibleRV}
The \emph{natural extensions} $\widehat{\V}$ and $\widehat{\Z}$, respectively of the Rauzy-Veech map $\V$ and  Zorich map $\Z$, also known as \emph{invertible Rauzy-Veech} and \emph{{invertible} Zorich} induction, are invertible  maps $\widehat{\V}$ and $\widehat{\Z}$ defined on a domain $\widehat{X}$ 
which are extensions of $\V$ and $\Z$ respectively, i.e.~such that there exists a projection $p \colon \widehat{X} \rightarrow X$ for which $p \widehat{\V} = \V p$ and $p \widehat{\Z} = \Z p$.
\subsubsection{Parameter spaces of the invertible Rauzy-Veech and Zorich map}\label{sec:RV}
To define the domain of the \emph{natural extensions} of $\widehat{\V}$ and $\widehat{\Z}$,  for each $\pi \in \mathcal R$ let $\Theta_{\pi}\subset \mathbb{R}^{d}$ be the polyhedral cones  given by the inequalities 
\bes
\Theta_{\pi} := \left\{ {\tau}=(\tau_1,\dots ,\tau_d)  \in   \mathbb{R}^{d}\ \Big| \quad \sum_{i=1}^{k} \tau_i >0 , \, \sum_{i=1}^{k} \tau_{\pi^{-1}i} <0, \, k=1,\dots, d-1 \right\},
\ees
and let $\widehat{X}$ be the following space of triples  (that admits a geometric interpretation in terms of the space of zippered rectangles,  see for example  \cite{Yo, Vi}, but the latter will play no role in the present paper):

\bes
\widehat{X} := \widehat{X} (\mathcal R)=\bigcup_{\pi\in \mathcal{R} }\{ \pi\} \times \Delta_d\times \Theta_\pi  = \{ \, (\pi, {\lambda}, {\tau} ) \ | \quad \pi \in \mathcal{R}, \, {\lambda}\in \Delta_d, \, {\tau} \in \Theta_{\pi}   \}.
\ees
The action of the invertible extension $\widehat\V$  on triples $(\pi,\lambda, \tau) \in \widehat{X}$ is obtained by extending the projective action of $\V$ as follows. If we write $\widehat{\mathcal{V}}(\pi,\lambda, \tau)=(\pi',\lambda', \tau')$, we have that $(\pi',\lambda')=\V (\pi,\lambda)$ (so that $\widehat{\V}$ is an extension of $\V$) and
$$
\tau'=\frac{(B^*_\gamma)^{-1}\tau}{|(B^*_\gamma)^{-1}\lambda|}, \quad \text{where}\  B^*_\gamma\in SL(d,\mathbb Z)
$$
is the transpose of the matrix $B_\gamma$ associated to the arrow in the Rauzy graph $\mathcal R$ starting from $\pi$ and ending at $\pi'$ (see \S~\ref{sec:products}).

\subsubsection{Invariant measures}\label{sec:invariantextensions}
One can show  that  $\widehat{\Z}$ preserves {an infinite} invariant measure $\widehat{\mu}_\Z$ on $\widehat{X}$, {cf.\ \cite{Zo:fin,Ve:gau,Vi}}.  In fact, $\widehat{\mu}_\Z$ restricted to $\{\pi\}\times\Delta_d\times \Theta_\pi$ is the product of ${\mu}_\Z$ and the Lebesgue measure restricted to the open cone $\Theta_\pi$. Furthermore, on $\widehat{X}$ one can define the real--valued function $Area(\cdot)$ associates to $(\pi,{\lambda},  {\tau}) $ given by
$Area(\pi,{\lambda}, {\tau} ) := \sum_{k=1}^{d} \lambda_k \, ( \sum_{i=1}^{k-1} \tau_i   - \sum_{i=1}^{\pi(k)-1} \tau_{\pi^{-1}i} ) $
 (called this way since it gives geometrically the area of the corresponding zippered rectangle).
  Given any measurable subset $E \subset \widehat{X}$,
 we will denote by $E^{(1)}$ the intersection of $E$ with the level set  $Area^{-1}(1)$ of the area function, namely
$$E^{(1)}=  \{ \, (\pi, {\lambda}, {\tau} ) \in E\ | \ \ Area (\pi,\lambda,\tau)= 1   \}.
 $$
In particular, this notation applies to $\widehat{X}$, so 
\bes
\widehat{X}^{(1)} :=\widehat{X}\cap Area^{-1}(1)= \{ \, (\pi, {\lambda}, {\tau} ) \, | \quad \pi \in \mathcal{R}, \, {\lambda}\in \Delta_d, \, {\tau} \in \Theta_{\pi} ,\,  Area(\pi,{\lambda}, {\tau}  )= 1   \},
\ees
The projection $p$ is then simply the forgetful map
$p:\widehat{X} \rightarrow X$ given by $ p  ( \pi,{\lambda}, {\tau} )  = (\pi,{\lambda})$.

Since the renormalization $\widehat{\Z}$ preserves $Area$, we can restrict it to $\widehat{X}^{(1)} $ and get a well defined map
 $\widehat{\Z}: \widehat{X}^{(1)}\to \widehat{X}^{(1)}$.
Moreover,  the projectivization of the measure $\widehat{\mu}_\Z$, denoted by $\widehat{\mu}^{(1)}_\Z$, is a finite $\widehat{\Z}$-invariant measure on $ \widehat{X}^{(1)}$ and is also ergodic (see e.g.~\cite{Vi} for details).
In fact, the projectivization $\widehat{\mu}^{(1)}_\Z$ is defined as follows
\begin{equation}\label{def:mu}
\widehat{\mu}^{(1)}_\Z(B):=\frac{\widehat{\mu}_\Z(\bigcup_{t\in{[0,1]}}tB)}{\widehat{\mu}_\Z(\bigcup_{t\in{[0,1]}}t\widehat{X}^{(1)})}\quad\text{ for {all} measurable}\quad B\subset\widehat{X}^{(1)},
\end{equation}
where $t(\pi,\lambda,\tau):=(\pi,\lambda,t\tau)$.
The push-forward $p_*{\widehat{\mu}^{(1)}_\Z}$  by the projection $p$ (i.e.\ the measure such that $p_*{\widehat{\mu}}^{(1)}_{\Z}(A)= {\widehat{\mu}}^{(1)}_Z (p^{-1}A)$ for {every} measurable set $A\subset X$) equals $\mu_{\Z}$ (see again~\cite{Vi}).

\subsubsection{Accelerations of the Rauzy-Veech map.}\label{sec:acc}
One can consider \emph{accelerations} of the (invertible)  Zorich  map $\widehat{\mathcal{Z}}$ by
 considering Poincar{\'e} first return map as follows.  Fix a subset $E\subset \widehat{X}^{(1)}$
of positive $\widehat{\mu}_\Z^{(1)}$-measure. By the ergodicity of $\widehat{\Z}$, for 
  the orbit of $\widehat{\mu}_\Z^{(1)}$-almost every  $(\pi, {\lambda}, \tau )$ under $\widehat{\Z}$ visits $E$ infinitely often.
  The corresponding acceleration of $\widehat{\mathcal{Z}}$ is denoted by $\widehat{\mathcal{Z}}_E$ and it is a map $\widehat{\Z}_E : \widehat{X}^{(1)} \to E$ defined a.e. on $\widehat{X}^{(1)}$ as follows:
 $\widehat{\Z}_E(\pi, {\lambda}, \tau )=\widehat{\Z}^{n_1}(\pi, {\lambda}, \tau )$ if $n_1\geq 1$ is the first times when the forward orbit  of $(\pi, {\lambda}, \tau )$ hits $E$, i.e.\ $\widehat{\Z}^{n_1}(\pi, {\lambda}, \tau )\in E$.
 If 
   $\{n_\ell\}_{\ell\in\N}$ denotes the  sequence of subsequent visits to $E$ for a typical $(\pi, {\lambda}, \tau )$   (so that  $\widehat\Z^{n_\ell}(\pi, {\lambda}, \tau )\in E$ for all $\ell$ and no other iterate  $\widehat\Z^n(\pi, {\lambda}, \tau)$ belongs to $E$), then
$\widehat{\Z}_E^{\ell}(\pi, {\lambda}, \tau):=\widehat\Z^{n_\ell}(\pi, {\lambda}, \tau)$, $\ell\in\N$.

We say that a sequence $\{n_l\}_{l\in \mathbb{N}}$ is a sequence of \emph{(lengths) balanced times} for $T$  if there exists  $\nu>1$
 such that  the lengths satisfy
  \[\frac{1}{\nu} \leq\frac{\lambda_i^{(n_l)}}{\lambda_j^{(n_l)}} \leq \nu\quad \text{ for\ all\ } 1\leq i,j\leq d, \ \text{ and\  all\ } l \in \mathbb{N}.\]
  One can see that if the subset  $E $ is of the form $E= \{\pi\}\times K\times Y$, where $K\subset \Delta_d$ is precompact and $Y\subset \Theta_{\pi}$, then the acceleration $\widehat\Z_E$ is balanced. Recall that a subset $K\subset \Delta_d$ is \emph{precompact} if its closure $\overline{K} \subset \operatorname{Int}\Delta_d$ (then $\overline{K}$ is indeed compact with respect to the Hilbert metric on $\Delta_d$, which we will not use in this paper).

\subsection{Symbolic cylinders and Markovian structure}\label{sec:statisticalRV}
We describe now some features of the Rauzy-Veech and Zorich maps on parameter space, which are manifestation of the intrinsic hyperbolicity of the algorithm.
\subsubsection{Local product structure}\label{sec:localproductstructure}
A key observation that makes the coordinates $\pi, \lambda,\tau$ on the domain $\widehat{X} $ of the natural extension $\widehat{\Z}$ of Rauzy-Veech induction particularly useful is the following form of (\emph{local}) \emph{product structure}.
Notice first that by definition of the elementary step of Rauzy-Veech induction, only the permutation $\pi$ and the vector $\lambda$ determine whether the basic step of the (invertible) Rauzy-Veech induction is of type \emph{top} or \emph{bottom}  (and hence the formula which gives the image of $(\pi, \lambda,\tau)$).
Similarly, the \emph{backward} basic step of the \emph{inverse} $\widehat{\mathcal{V}}^{-1}$ of the  (invertible) Rauzy-Veech induction $\widehat{\mathcal{V}}$ is determined by $\pi$ and $\tau$ only:

\begin{rem}\label{rk:backwardstep} If $(\pi', \lambda', \tau')= \widehat{\mathcal{V}}^{-1}(\pi, \lambda, \tau)$, then one can know whether
$(\pi', \lambda', \tau')$ and hence the Rauzy-Veech operation performed by $\widehat{\mathcal{V}}$ on $(\pi', \lambda', \tau')$ was of type \emph{top} (resp.~\emph{bottom})  from the vector $\tau$ only. More precisely, if
\begin{equation}\label{eq:taurel}
\tau_{1}+ \dots + \tau_{d}<0 \qquad \text{(respectively \ $ >0$)},
\end{equation}
the previous move was of type \emph{top} (resp.~\emph{bottom}), we refer e.g.~to \cite{Vi}.  Thus, $\pi'$ can be determined from $\pi$ and $\tau$ only (independently on $\lambda$) and
 $\lambda'$ and $\tau'$ are proportional to {$B^*_\gamma \lambda$ and $B^*_\gamma \tau$} respectively, where $B_\gamma$ is an elementary matrix which is fully determined by $\pi$ and $\tau$.
\end{rem}

\subsubsection{Markovian structure}\label{sec:Markov}
One can furthermore show that $\widehat{\mathcal{Z}}$ has a \emph{Markovian structure} in the following sense.  For any $\pi \in \mathcal{R}$, if we denote by $\Delta_\pi^{t}$ (resp.~$\Delta_\pi^{b}$) the set of $\lambda\in \Delta_d$ such that $(\pi, \lambda)$ is of type top (resp.~bottom), and by $\Theta_\pi^t$ (resp.~ $\Theta_\pi^b$) the set of $\tau \in \Theta_\pi$ such that \eqref{eq:taurel} holds, then, if $\pi_0$ (resp.~$\pi_1$) is obtained from $\pi$ performing a top (resp.~bottom) operation, then
\begin{align}\label{backwardMarkov1}
& \widehat{\mathcal{Z}} \left( \{\pi\}\times \Delta_\pi^t \times \Theta_{\pi}\right) =
 \{\pi_0\}\times \Delta_d \times \Theta^t_{\pi_0 } &\ \  \Leftrightarrow \ \ \widehat{\mathcal{Z}}^{-1}
 \left(  \{\pi_0\}\times \Delta_d \times \Theta^t_{\pi_0}\right) =
\{\pi\} \times \Delta^t_\pi \times \Theta_{\pi},
\\ & \widehat{\mathcal{Z}} \left( \{\pi\}\times \Delta_{\pi}^b \times \Theta_{\pi}\right) =
 \{\pi_1\}\times \Delta_d \times \Theta^b_{\pi_1} & \ \  \Leftrightarrow \ \
  \widehat{\mathcal{Z}}^{-1} \left( \{\pi_1\}\times \Delta_d \times \Theta^b_{\pi_1}\right) = \{\pi\}\times \Delta_\pi^b \times \Theta_{\pi}.\label{eq:backwardMarkov2}
\end{align}

Let us recall that $p: \widehat{X} \to X$ denotes the projection $p(\pi,\lambda, \tau ) = (\pi, \lambda)$.  Let  $p'$ denote the projection on the $\tau$-coordinate, namely
\begin{gather}
p':  \widehat{X}(\mathcal R)  \to   \bigcup_{\pi \in \mathcal{R}} \{ \pi\} \times \Theta_\pi, 
\\ \label{eq:p2def}
 (\pi,\lambda, \tau )   \mapsto  p' (\pi,\lambda, \tau )  \doteqdot  (\pi, \tau) .  \nonumber
\end{gather}
From Remark~\ref{rk:backwardstep},
the following two observations about the orbits of $\widehat{\mathcal{Z}}$ (where  (F) stands for \emph{future}, (P) for \emph{past} iterates) follow and describe the local product structure.\footnote{We remark that these properties show that the invertible map $\widehat{\mathcal{Z}}:\widehat{X}\to \widehat{X}$  is  a \emph{fibred system} in the sense of Schweiger, see
\cite{Sch95}. }
Assume for simplicity that $ (\pi, \lambda, \tau)$ and  $({\pi}, \overline{\lambda}, \overline{\tau})$ are two triples in $\widehat{X}$ with the same combinatorial datum $\pi\in \mathcal{R}$, for which the orbit $\widehat{\mathcal{Z}}^n (T)$ is defined for all $n\in \mathbb{Z}$.

\begin{itemize}
\item[(F)] If  $\lambda= \overline{\lambda}$, then for {every} $n>0$, the \emph{future} orbits of $ (\pi, \lambda, \tau)$ and $ (\pi, {\lambda}, \overline{\tau})$ both lift the orbit  $\mathcal{Z}^n (\pi,\lambda)$ of the IET $(\pi,\lambda)$, i.e.~they both belong to the same $p$-fiber:
$$
p \, (\widehat{\mathcal{Z}}^n \left(\pi, \lambda, \tau\right))= p\, \widehat{\mathcal{Z}}^n \left(\overline{\pi}, \overline{\lambda}, \overline{\tau}\right)=  (\pi^{(n)},{\lambda}^{(n)})    = \mathcal{Z}^n(\pi,\lambda).
$$
\item[(P)] If  $\tau= \overline{\tau}$, then for {every} $n< 0$, the \emph{past} orbits of $ (\pi,\lambda , \tau)$ and $ (\pi,  \overline{\lambda}, {\tau})$ both belong to the same $p'$-fiber:
$$
p' (\widehat{\mathcal{Z}}^n \left(\pi, \lambda, \tau\right))  = p' \widehat{\mathcal{Z}}^n \left(\overline{\pi}, \overline{\lambda} , \overline{\tau}\right)= (\pi^{(n)},{\tau}^{(n)}) .
$$
\end{itemize}
Let us mention that  it is possible (although we do not want to do it here)  to define a map  $\mathcal{Z}^\ast$ on the space $ \cup_{\pi\in \mathcal{R}} (\{\pi\} \times \Theta_\pi)$ such that
$$
\pi' \left( \widehat{\mathcal{Z}}^{-1} (\pi,\lambda, \tau) \right)= \mathcal{Z}^\ast (\pi, \tau), \qquad \text{which gives} \qquad p' \left(\widehat{\mathcal{Z}}^{-n} (\pi, \lambda, \tau)\right)  =( \mathcal{Z}^\ast)^n  (\pi, \tau),
$$
so that (P) can have a similar form to $(F)$. The map  $\mathcal{Z}^\ast$ is then \emph{dual} to Zorich induction\footnote{The curious reader may be interested in the paper \cite{In-Na} by Inoue and Nakada,  where it is shown that the dual of Rauzy-Veech induction can be realized by an algorithm known in the literature as \emph{da Rocha induction}.} in the sense of Schweiger (see
\cite{Sch95}).
\subsubsection{Two-sided cylinders and Markov shifts}\label{sec:Markov2}
Let $\gamma$ be a path on $\mathcal{R}$ starting from $\pi$ and ending at $\pi'$. 
The simplex $\Delta_\gamma$ defined in \S~\ref{sec:lambdacylinder} by \eqref{eq:Delta_gamma} consists of all length parameters $\lambda$ of all triples $(\pi,\lambda, \tau)$  which share the first (future) $n$ steps of the (invertible) Rauzy-Veech induction. Similarly, by the local product structure,
there exists a set $\Theta_\gamma\subset \Theta_{\pi'}$ which consists of all parameters $\tau'$ of  the triples $(\pi', \lambda', \tau')$ which share the \emph{past} $n$ steps of $\widehat{\V}$ (i.e.~the first $n$ iterates of $\widehat{\V}^{-1}$). This set is a simplicial cone defined by
\begin{equation}\label{def:Theta}
 \Theta_\gamma:=\Theta_{B_\gamma} =
 (B^*_\gamma)^{-1}(\Theta_{\pi})=
\left\{ (B_\gamma^*)^{-1} \tau\ : \ \tau \in \Theta_\pi\right\}\subset \Theta_{\pi'} .\end{equation}
 The Markovian structure extends to these sets: if $\gamma=\gamma_1\cdots \gamma_n$ is obtained concatenating $n$ arrows,
 iterating $n$ times the Markovian structure of one Rauzy-Veech induction move given by  \eqref{backwardMarkov1} and \eqref{eq:backwardMarkov2}, we get by induction that:
\begin{equation}\label{bisided_Markov}
 \widehat{\mathcal{V}}^n \left( \{\pi\}\times \Delta_{B_\gamma} \times \Theta_{\pi}\right) =
 \{ \pi' \}\times \Delta_d \times \Theta_{B_\gamma} .
\end{equation}
The product sets $\{\pi\} \times \Delta_\alpha \times \Theta_\beta$ where $\alpha$ and $\beta$ are finite paths 
 respectively starting and ending at $\pi$ play the role of \emph{bi-sided cylinders}
 for the symbolic  coding of the induction, since they consist of all triples in $\widehat{X}$ which share the (finitely many) backward and forward steps determined by $\alpha$ and $\beta$ respectively.

\subsection{Kerckhoff's lemma and distortion bounds}\label{sec:distortion_background}
Finally, we want to recall some of the \emph{distortion} properties of the iterates of the induction, which relate distortion of the matrices which describe the algorithm (given by ratios of norms of columns) to distortion of volumes of cylinder sets for the induction.

\subsubsection{Notation for Rauzy graph paths}
Throughout this \S~\ref{sec:distortion_background},  we refer mostly to the notation and the results used in \cite{AGY}, that we now recall. 

Let $\mathcal{R}$ be a Rauzy class.  Denote by $\Pi(\mathcal R)$ the set of {paths} in $\mathcal{R}$ and for {every} $\gamma\in \Pi(\mathcal R)$ let $|\gamma|$ be the length of $\gamma$.
For any $\pi\in\mathcal R$, let us denote by $\Delta_\pi=\{\pi\}\times\Delta_d$ the copy of the simplex indexed by $\pi$.
Thus we can write the phase space of (non-invertible) Rauzy induction (defined by \eqref{eq:defX} in \S~\ref{sec:noninvertibleRV}) as $X(\mathcal R)=\bigcup_{\pi\in\mathcal R}\Delta_\pi$.

We use the notation introduced in \S~\ref{sec:products} to  index the products of matrices $B_\gamma$ which give iterates of $\mathcal{V}$ through the \emph{path}  $\gamma$ in the Rauzy-graph which describes the iterate.
If $\gamma$ is a path on $\mathcal{R}$ obtained by concatenation $\gamma=\alpha\beta$ of two paths, set $B_\gamma:= B_\beta B_\alpha $.
  For any path $\gamma\in\Pi(\mathcal R)$ starting at $\pi$ and ending at $\pi'$ we write for short $B^*_\gamma:\Delta_{\pi'}\to\Delta_{\pi}$ for the map which is given on the lengths coordinates
   by the projectivization of the linear map $B^*_\gamma$, as in \eqref{eq:Bgamma}.


\subsubsection{Columns balance}
For any non-negative matrix $M\in SL(d,\mathbb Z)$ and $1\leq i\leq d$ denote (using the notation introduced by Chaika in \cite{Cha}) by $C_i(M)$ its $i$-column and by $|C_i(M)|$ the sum of its entries. Denote by $C_{\max}(M)$ and $C_{\min}(M)$ the columns for which the sum of entries are maximal or minimal respectively. We say that $M$ is $C$-\emph{balanced} for some $C>1$ if
$$\frac{|C_{\max}(M)|}{|C_{\min}(M)|}\leq C.$$
Remark that, if $B^*_\gamma$  is the transpose of one of the matrices $B_\gamma$ associated to an arrow $\gamma$ by a step of Rauzy-Veech induction (see \S~\ref{sec:basic}), since
$B_\gamma$ has $1$ on the diagonal and only one $1$ off the diagonal, we have that
\begin{equation}\label{eq:MB}
|C_{\max}(M\cdot B^*_\gamma)|\leq 2 |C_{\max}(M)|.
\end{equation}
It follows that
\begin{equation}\label{eq:growthB}
|C_{\max}(B^*_\gamma)|\leq 2^{|\gamma|}\text{ for {every} }\gamma\in \Pi(\mathcal R).
\end{equation}

\subsubsection{Unique ergodicity contraction}
For any $(\pi,\lambda)\in X(\mathcal R)$ satisfying Keane's condition let $\gamma$ be the infinite path in $\mathcal R$ that indicated successively the types of all steps of the Rauzy-Veech induction and for any natural $n$ let $\gamma_{[1,n]}$ be its prefix of length $n$, i.e.~$\gamma_{[1,n]}=\gamma_1\gamma_2\cdots \gamma_n$ where $\gamma_1,\cdots ,\gamma_n$ are the first $n$ arrows of $\gamma$.
Consider the simplexes  $ \Delta_{\gamma_{[1,n]}}$, defined as  in \S~\ref{sec:lambdacylinder}, see \eqref{eq:Delta_gamma}.
Then $ \Delta_{\gamma_{[1,n]}}$, for $ n\geq 1$ form a decreasing sequence of sets (simplices) containing $(\pi,\lambda)$.

As shown by Veech in \cite{Ve:gau}, for a.e.\ $(\pi,\lambda)$, we have
\begin{equation}\label{eq:decay}
\bigcap_{n\geq 1}\Delta_{\gamma_{[1,n]}}=\{(\pi,\lambda)\}\quad\text{and}\quad \lim_{n\to\infty} |C_{\max}(B^*_{\gamma_{[1,n]}})|=+\infty.
\end{equation}
The above result  is actually equivalent to unique ergodicity of almost every $T=(\pi,\lambda)$ (since the intersection in \eqref{eq:decay} can be shown to be in one-to-one correspondence with the cone of invariant measures for $T$).

\subsubsection{Volumes of cylinders}
Let us denote by $m$ the Lebesgue measure on $X(\mathcal R)=\bigcup_{\pi\in\mathcal R}\Delta_\pi$, which we recall, for {every} measurable $A\subset \Delta_d$ and {every} $\pi\in\mathcal R$,  is given by
$$m(\{\pi\}\times A)=d! \, Leb(\{tx:x\in A,\ t\in[0,1]\}).$$
One can relate the Lebesgue measure of this simplex in terms of columns using the following formula (proved as Proposition~5.4 in \cite{Ve78}):
\[m(\Delta_\gamma)=\frac{1}{|C_1(B^*_\gamma)|\cdots|C_d(B^*_\gamma)|},\]
so $m(\Delta_\pi)=1$ for {every} $\pi\in\mathcal R$.
Moreover, by Corollary~1.2 in \cite{Ker} (see also \cite{Ve78} for details), if $B^*_\gamma$ is $C$-balanced then
\begin{equation}\label{eq:balmeas}
C^{-d}\frac{m(A_1)}{m(A_2)}\leq \frac{m(B^*_\gamma A_1)}{m(B^*_\gamma A_2)}\leq C^d\frac{m(A_1)}{m(A_2)}\quad\text{ for all measurable }\quad A_1,A_2\subset \Delta_{\pi'}.
\end{equation}

\subsubsection{Kerckhoff's lemma}
We recall now a key estimate which will be used later in \S~\ref{sec:resontanttimes2}. The statement was first proved by Kerckhoff in \cite{Ker} and is known as \emph{Kerckhoff's lemma}. The proof of this estimate is also presented by Avila, Gou\"ezel, and Yoccoz in \cite{AGY}. We here state it in the version using the notation in \cite{AGY}.
\begin{proposition}[Kerckhoff's Lemma, Corollary~1.7 in \cite{Ker} and Theorem~A.2 in \cite{AGY}]\label{prop:AGY}
For {every} Rauzy graph $\mathcal R$ there exists a constant $C>1$ such that for {every} $\gamma\in \Pi(\mathcal R)$ there exists a finite\footnote{The claim that the set of paths is \emph{finite} is not part of the statement of  Theorem~A.2 in \cite{AGY}, but since the inequality in \eqref{eq:inDC} is sharp, we can restrict to a smaller finite subset of $\Gamma(\gamma)$, without violating all its properties described in Proposition~\ref{prop:AGY}. Therefore, 
 the set $\Gamma(\gamma)$ can always be chosen to be finite.} subset $\Gamma(\gamma)\subset \Pi(\mathcal R)$ such that for {every} $\tilde\gamma\in \Gamma(\gamma)$ the following properties hold:
\begin{itemize}
\item[(K1)] $\gamma$ is a prefix of $\tilde\gamma$;\vspace{1mm}
\item[(K2)]  $B^*_{\tilde\gamma}$ is $C$-balanced;\vspace{1mm}
\item[(K3)] $|C_{\max}(B^*_{\tilde\gamma})|\leq C|C_{\max}(B^*_\gamma)|$.
\end{itemize}
Moreover, $\Gamma(\gamma)$ is prefix-free, that is, no path in $\Gamma(\gamma)$ is a prefix of any other path, and
\begin{equation}\label{eq:inDC}
\sum_{\tilde\gamma\in \Gamma(\gamma)}m(\Delta_{\tilde\gamma}|\Delta_\gamma)>C^{-1}.
\end{equation}
\end{proposition}
\noindent Kerckhoff's lemma is the key ingredient in the proofs of exponential mixing of Rauzy-Veech induction and of the Teichmueller geodesic flow (see \cite{AGY} and \cite{AB}). It is also used by Chaika  in the proof in \cite{Cha} that almost every pair of IETs is disjoint. We will use it in Section~\ref{sec:resontanttimes2}.

\section{Rigidity times with trimmed derivative control}\label{sec:coexistence}

 In this section we prove Proposition~\ref{prop:cohexistence}, showing the existence of \emph{rigid towers} (as defined by Definition~\ref{def:tower}) with \emph{trimmed derivative linear bounds} (in the sense of Definition~\ref{def:bounded} in \S~\ref{sec:linearbounds}).  In \S~\ref{sec:derivativesviaRV}, we first provide a more precise version of the result on trimmed Birkhoff sums  proved by the third author in   \cite{Ul:abs}, which provides the desired trimmed Birkhoff sums linear bounds, but along sequences of towers which are \emph{balanced}.

To produce a rigid time which still has the trimmed linear derivative bounds, we will use that linear bounds in \emph{large} towers persist for a finite number of steps of the induction (see \S~\ref{sec:persistence}) and, crucially, that the Markovian structure of Rauzy-Veech induction (see \S~\ref{sec:Markov})  allows to find rigid times finitely many steps after a balanced one (as shown in \S~\ref{sec:Markovloss}).
  In the final subsection \S~\ref{sec:coexistenceproof} we explain how to combine these results to conclude the proof of Proposition~\ref{prop:cohexistence}.


\subsection{Trimmed derivative bounds  using Rauzy-Veech induction}\label{sec:derivativesviaRV}
Let us first state a
technical
version of a result proved by the last author in  \cite{Ul:abs}, which allows to construct  towers
 where the linearly bounded trimmed derivative bounds hold (see Definition~\ref{def:bounded}) using suitable return times of Rauzy-Veech induction.

\begin{proposition}[acceleration for bounded  trimmed derivative sums]\label{prop:accelerationset}
For {every} irreducible permutation $\pi$ in any Rauzy class $\mathcal{R}$, and {every}
positive matrix $B=B_\gamma$ associated to a neat\footnote{A \emph{neat} path $\gamma$, according to the definition given  in \cite{AGY}, is a path $\gamma$ starting and ending in $\pi$ such that $B_\gamma$ is positive, $(B^*_\gamma)^{-1}(\overline{\Theta}_\pi\setminus \{0\})\subset \Theta_\pi$ and  if $\gamma=\gamma_0\gamma_e=\gamma_s\gamma_0$, then either $\gamma_0$ is trivial, or $\gamma_0=\gamma$.  Assuming that a path is neat is a technical condition which simplifies some proofs by ensuring that each branch of the first return map on $\Delta_\gamma$ starts with the matrix $B_\gamma$.} path $\gamma$ in $\mathcal{R}$, 
there exists a  subset $Y 
\subset \Theta_\pi$ { with  closure $\overline{Y} \subsetneq \Theta_\pi$,}   invariant under rescaling, i.e.~such that $tY= Y $ for all real $t>0$,
 and with $Leb(Y)>0$, such that the set
$$
E
:= (\{ \pi\}\times \Delta_B \times Y ) ^{{(1)}}= \{ (\pi,\lambda, \tau)\in \widehat{X}^{{(1)}}\, | \,\, \lambda\in \Delta_B, \, \tau \in Y \}
$$
(where $\Delta_B$ is defined as in \eqref{def:simplex})
has measure $\mu_{\widehat{\Z}}^{(1)}(E)>0$ and has the following property:
 if
 $(\pi,\lambda,\tau)$ 
 is recurrent to $E$  under $\widehat{\Z}$ and $(n_\ell)_{\ell\in \N}$ is the sequence of  visits, i.e.~$$\widehat{\mathcal{Z}}^{(n_\ell)}(\pi,\lambda,\tau)\in E \quad \textrm{for\ every}\ \ell\in \N,$$ then {every} $f\in \pSymLog{T}$
 has linearly bounded trimmed derivative Birkhoff sums for the IET $T=(\pi,\lambda)$  along the  towers $( \mathcal{T}^{(n_\ell)}_j)_{\ell\in\N}$, $1\leq j\leq d$ over the inducting intervals $I^{(n_\ell)}_j$, $1\leq j\leq d$.
 \end{proposition}
\noindent For the definition of linearly bounded trimmed Birkhoff sums along the towers, we refer the reader to Proposition~\ref{boundSf'growthprop}, see \eqref{eq:propconclusion}.

\begin{remark}\label{rk:cancellations_are_balanced}
Notice that since the length vectors $\lambda^{(n_\ell)}$ belong to a precompact $\Delta_B$ by definition of $E$, the induction times $(n_\ell)_{\ell\in\N}$ are \emph{balanced}, i.e. the ratios $\lambda^{(n_\ell)}_i/\lambda^{(n_\ell)}_j$ are all uniformly bounded.
\end{remark}
\noindent This Proposition is implicitly proved in \cite{Ul:abs}: in Proposition~4.2 of \cite{Ul:abs} the existence of a set $E\subset \widehat{X}^{(1)}$ with the above property is proved. The particular form of $E$ in Proposition~\ref{prop:accelerationset}, though, namely that $E$ has the product structure $E=(\{ \pi\} \times \Delta_B \times Y )^{(1)}$ for some $Y\subset \Theta_\pi$,  plays a crucial role in the following arguments. The fact
 that $E$ \emph{can} be chosen to be of this form is implicit in the proof of Proposition~4.2 of \cite{Ul:abs}. We provide some details and more comments in Appendix \ref{sec:app}.
We will use this proposition together with the Markovian loss of memory described in the next subsection to prove the existence of rigid times in Proposition~\ref{prop:cohexistence}.
\subsection{Markovian loss of memory}\label{sec:Markovloss}
We now show that the Markovian structure described in \S~\ref{sec:statisticalRV} allows to find a rigid Rauzy-Veech time finitely many steps after one of the times given by Proposition~\ref{prop:accelerationset}.
\begin{lemma}\label{lemma:loss}
Let $B:=B_\gamma$ be the matrix associated to a path $\gamma$ in $\mathcal{R}$ of length $|\gamma|=n$, starting at $\pi$ and ending at $\pi'$. Given any (open and) non-empty subset $\mathcal{N} \subset \Delta_d$, there exists a non-empty subset $\mathcal{U}$ of IETs in $\{\pi\}\times \Delta_B$, of the form $\mathcal{U} = \{\pi\} \times U$ for $U\subset \Delta_B$ (open and) non-empty  such that
$$
T=(\pi,\lambda )\in \mathcal{U}
 \quad \Rightarrow \quad \mathcal{V}^n(T) \in \{\pi'\}\times \mathcal N.
$$
\end{lemma}
\begin{proof}
The statement follows from the Markovian structure described in  \S~\ref{sec:Markov2}. Indeed, if we denote by $B_\gamma  = B_1 \dots B_n$ the Rauzy-Veech matrix corresponding to the $n$ top/bottom moves prescribed by the path $\gamma$ and denote by
$\Delta_{B_\gamma}$ and $\Theta_{B_\gamma}$ the simplices defined in \S~\ref{sec:lambdacylinder} and \S~\ref{sec:Markov2} respectively,
by the Markovian relation \eqref{bisided_Markov} for $B:=B_\gamma$, 
$$
 \widehat{\mathcal{V}}^n\left(  \{\pi\}\times \Delta_B \times \Theta_{\pi}\right) =
 \{\pi'\}\times \Delta_d \times \Theta_{B} ,
$$
i.e.~$ \widehat{\mathcal{V}}^n$ maps injectively and surjectively $\{\pi\}\times \Delta_B \times \Theta_{\pi}$ to $ \{\pi'\}\times \Delta_d \times \Theta_{B}$.  Therefore,  given any non-empty (open) $\mathcal N\subset \Delta_{d}$, the
set $ \{\pi'\}\times \mathcal{N} \times \Theta_{B}$ has a non-empty (open) preimage.
 Explicitly, one can check that the preimage has the form
$$ \{\pi\}\times \mathcal{U}\times \Theta_\pi, \quad \text{where}\ \
\mathcal{U}=\left\{\frac{B^*\lambda}{|B^*\lambda|}:\lambda\in\mathcal N\right\}.$$
   This concludes the proof.
\end{proof}


\subsubsection{Persistence of linear bounds}\label{sec:persistence}
We will need also the following Lemma, that shows that linear bounds on trimmed derivatives persist in large towers obtained by cutting and stacking a finite number of Rohlin towers.

Given a Keane IET, if $\{ I^{(n)}, \ {n\in\N}\}$ 
 is the sequence of inducing intervals given by (an acceleration of) Rauzy-Veech induction, we denote by $( \mathcal{T}^{(n)}_j)_{n\in\N}$, $1\leq j\leq d$ the Rohlin towers  over the continuity intervals $I^{(n)}_j$ given  by
$$
\cC^{(n)}_j:= \left\{  I^{(n)}_j,T(I^{(n)}_j), \dots, T^{h^{(n)}_j-1}I^{(n)}_j\right\}, \qquad 1\leq j\leq d,
$$
where $h^{(n)}_j$ is the first return time of  $I^{(n)}_j$ to $I^{(n)}$.
Recall Definition~\ref{def:bounded} for the definition of $M$-bounded trimmed derivatives.

\begin{lemma}\label{lemma:finitecombination}
Given a Keane IET, let $\mathcal{T}^{(n)}_j$, for $n\in \mathbb{N}$ and $1\leq j\leq d$, be Rohlin towers given by Rauzy-Veech induction. Assume that there exists a subsequence $(n_\ell)_{\ell\in\N}$
and a constant $M>0$  such that
  $f$ has $M$-bounded trimmed derivatives along each of  the sequences of towers $(\mathcal{T}^{(n_\ell)}_j)_{\ell\in\N}$  for  $1\leq j\leq d$.

\noindent Then if  $(\mathcal{T}^\ell)_{\ell\in\N}$ is a sequence of towers of the form $\mathcal{T}^\ell:= \mathcal{T}^{(m_\ell)}_{j_\ell}$, { where $j_\ell\in\{1,\ldots,d\}$  for $\ell\in\N$ and  $(m_\ell)_{\ell\in\N}$ is a sequence 
  such that:}
\begin{itemize}
\item[(A1)] {\rm (finite linear combination)} there exists natural $K$ such that, for every $\ell \in \mathbb{N}$, we have that
$m_\ell\geq n_\ell$ and  $  h^{(m_\ell)}_{j_\ell} $ can be decomposed into at most $K$ towers of level $n_\ell$, i.e.\
$$
  h^{(m_\ell)}_{j_\ell} = \sum_{k=0}^{k_\ell}   h^{(n_\ell)}_{i_k}, \qquad \text{for\ some} \ \ k_\ell < K \quad  \text{and}\quad  i_0,i_1,\dots , i_{k_\ell }\in \{1,\dots, d\};
$$
\item[(A2)] {\rm (large tower)}
there exists $\delta>0$  such that the tower $\mathcal{T}^{(m_\ell)}_{j_\ell}$ has measure
$$\left|\mathcal{T}^{(m_\ell)}_{j_\ell}\right|= \lambda^{(m_\ell)}_{j_\ell} h^{(m_\ell)}_{j_\ell}\geq \delta >0 \quad\text{for all}\quad\ell\in\N,$$
\end{itemize}
then
  $f$ has $M'$-bounded trimmed {derivatives} along $(\mathcal{T}^\ell)_{\ell\in\N}$ where  $M'$ depends only on $K$ and $\delta$ coming from
   \emph{(A1)} and \emph{(A2)} respectively.
 \end{lemma}
\noindent The first assumption (A1) implies that the tower $\mathcal{T}_{\ell} {= \mathcal{T}^{(m_\ell)}_{j_\ell}}$ is obtained by cutting and stacking at most $K$ Rohlin towers  $\mathcal{T}^{(n_\ell)}_{j}$ of the previous level $n_\ell$. Notice that this condition is automatically satisfied if
 the product of the cocycle matrices
  between time $n_\ell$ and $m_\ell$ has uniformly bounded norm independently on $\ell$. 
 {Moreover, we remark that if the times $(m_\ell)_\ell$ of a sequence that satisfies (A1)  happens to be all balanced, then
 condition (A2)  (and hence the $M'$-trimmed derivative bound given by the conclusion) holds automatically for any sequence of towers of the form $(\mathcal{T}^{(m_\ell)}_{j_\ell})_\ell$, i.e.~for any choice of sequence $(j_\ell)_\ell$ with
$j_\ell\in \{1, \ldots, d\}$.}

\smallskip
The rest of this subsection is spent on the proof of  Lemma~\ref{lemma:finitecombination}.
\begin{proof}[Proof of Lemma~\ref{lemma:finitecombination}]
For short, write
$$
J_\ell:= I^{(m_\ell)}_{j_\ell}, \qquad  h_\ell:= h^{(m_\ell)}_{j_\ell}, \quad \text{and} \quad\mathcal{T}_\ell := \mathcal{T}^{(m_\ell)}_{j_\ell}
$$
for respectively the base, the height and the tower $\mathcal{T}_\ell$.
Fix any  $x$ in the base $J_\ell$.
Since the towers of level $m_\ell$ are obtained by cutting and stacking towers of level $n_\ell$,
and any  $0\leq r \leq h_\ell$,
we can decompose the Birkhoff sum over the orbit $\mathcal{O}_T(x,r)$ (which has one point for each floor of the tower $\mathcal{T}_\ell$) into at most $K$   Birkhoff sums, each of length at most  $h^{(n_{\ell})}_{i}$ for some $1\leq i\leq d$, as follows.

Set  $x_0:=x$ and denote by $x_{1}, x_2, \dots , x_{n}$ ($n\leq k_\ell$) the successive visits of $\mathcal{O}_T(x,r)$  to  $I^{(n_{\ell})}$.
Let $i_0,\ldots,i_n $ be the sequence of indexes in $\{1,\ldots, d\}$ such that
\[
x_k\in I^{(n_\ell)}_{i_k}, \quad \text{for}\  0\leq k\leq n.
\]
To simplify the notation, let us also write $h^\ell_k:= h^{(n_{\ell})}_{i_k}$ for $0\leq k\leq n-1$.
Then, we have the decomposition:
\begin{equation}
\label{eq:BSdecomp}
S_{r}f'(x) = \sum_{k=0}^{n-1} S_{h^\ell_k} f'(x_k) + S_{r_n} f'(x_{n}), \quad \text{where} \, 0\leq r_n\leq h^{(n_{\ell})}_{i_{n_\ell}}.
\end{equation}
To keep the notation compact, set also $h^\ell_n:= r_n$. Then,  considering the lengths of the sums (which are heights of towers), we also have that
\begin{equation}\label{eq:towercombination}
r= h^{(n_\ell)}_{i_0} +h^{(n_\ell)}_{i_1} +  \dots + h^{(n_\ell)}_{i_{n-1}}+ r_n=: h^\ell_0+h^\ell_1 +  \dots + h^{\ell}_{n-1}+h^\ell_n.
\end{equation}
By assumption (A1), since $r\leq h_\ell$, we have that
 the number $n$ of terms in the decomposition is bounded by $n< K$.

For each of the Birkhoff sums that appear in the decomposition \eqref{eq:BSdecomp},  recalling the definition of trimmed Birkhoff sums and related notation (see \S~\ref{sec:trimmedBS def}) and using the trimmed Birkhoff sums derivative control for the towers $\mathcal{T}^{(n_{\ell})}_{i_k}$, we have that
\[
\left| S_{h^\ell_k} f'(x_k)  + \sum_{j=0}^{d-1} \frac{C^+_j}{m_j^+(x_k, h^\ell_k)} - \sum_{j=1}^{d} \frac{C^-_j}{m_j^-(x_k, h^\ell_{k})} \right| = \left| \widetilde{S}_{h^\ell_k} f'(x_k) \right|\leq M h^{(n_{\ell})}_{i_k},
\]
for $0\leq k\leq n$.
{ If now,  for every $j$ involved, we \emph{exclude} the contributions of the closest points of the orbit $\mathcal{O}_T(x,r)$ to a singularity $\beta_j$ from the right and left respectively, namely the terms
$$\frac{C_j^+}{{ m_j^+ (x_k, h^\ell_{{k}})}}\  \text{for}\  0\leq k\leq n
 \ \text{and} \   \frac{C_j^-}{{ m_j^- (x_k, h^\ell_{k})}}
\  \text{for}\  0\leq k \leq n,
  $$
 \emph{except}  the \emph{largest} term of  each of the groups of $n+1$ contributions, that  is achieved for some  $ 0\leq k^\pm(j)\leq n $ corresponding to the respective closest visit, so that
$$\frac{C^\pm_j}{m_j^\pm(x_{k^\pm_j},h^\ell_{{k^\pm_j}})}:= \max \left\{ \frac{C_j^\pm}{{ m_j^\pm (x_k, h^{\ell}_{k})}}\ : \ {0\leq k \leq n} \right\}. $$
 We claim that the contribution of all the other closest visits can be estimated by $2C_j^\pm h_\ell$ for {all}  $\ell$  sufficiently large.
Indeed, all the points of $\mathcal{O}_T(x,r)$ belong to distinct floors of the tower $\mathcal{T}_{\ell}$ and hence are $|J_\ell|$ spaced; thus, already the second closest (and hence any other point in the orbit) to each singularity are at least $|J_\ell|$ far from the respective singularity. Moreover,   $|J_\ell|\geq \delta/h_\ell$  (since the tower $\mathcal{T}_\ell$ by assumption (A2) has measure $|J_\ell| h_\ell\geq\delta$). In summary, we have
\[m_j^\pm(x_{k^\pm_j}^k,h_{i_{k^\pm_j}}^{(n_\ell)})=m_j^\pm(x,h_\ell)\quad\text{and}\quad m_j^\pm(x_{k}^k,h_{i_{k}}^{(n_\ell)})\geq \frac{\delta}{h_\ell}\ \text{ if }\ k\neq k^\pm_j.\] }
\noindent Using this, and recalling that   by definition trimmed of Birkhoff sums
\[ \widetilde{S}_{h_\ell}f'(x)   = {S}_{h_\ell}f'(x) +  \sum_{j=0}^{d-1} \frac{C^+_j}{ m_j^+(x, r)} - \sum_{j=1}^{d}\frac{C^-_j}{ m_j^-(x,r)} ,\] 
using this remark and the trimmed Birkhoff sums
we can rewrite the bounds in the decomposition~\eqref{eq:BSdecomp}, denoting  $C_{sum}:=\sum_{j=0}^{d-1} C^+_j+\sum_{j=1}^{d} C^-_j$,  to get:
\begin{align*}
\left| \widetilde{S}_{h_\ell}f'(x)  \right|  &\leq \sum_{k=0}^{n} |\widetilde{S}_{h^\ell_k} f'(x_k)|+ \sum_{j=0}^{d-1} C^+_j\sum_{\substack{0 \leq k\leq n \\ k \neq k^+_j}}\frac{1}{m_j^+(x_k, h^\ell_k)} + \sum_{j=1}^{d} C^-_j\sum_{\substack{0\leq k\leq n\\ k \neq k^-_j}}\frac{1}{m_j^-(x_k, h^\ell_k)}\\
&\leq \sum_{k=0}^{n}Mh^{(n_{\ell})}_{i_k}+n\frac{C_{sum}h_\ell}{\delta}\leq Mh_\ell+\frac{KC_{sum}}{\delta}h_\ell
\leq M' h_\ell, 
\end{align*}
where in the last inequality we used \eqref{eq:towercombination} and set $M':= M + K C_{sum}/\delta$
(since  $n\leq K$ and $\sum_{k=0}^{n}h^{(n_{\ell})}_{i_k}\leq \sum_{k=0}^{k_\ell}h^{(n_{\ell})}_{i_k}= h_\ell$).
This shows that the claimed trimmed Birkhoff sums derivative control holds for $f$ along the sequence of towers $(\mathcal{T}_\ell)_{\ell\in\N}$, thus concluding the proof.
\end{proof}

\subsection{Existence of rigid  times with trimmed derivatives bounds}\label{sec:coexistenceproof}
In this subsection we prove  Proposition~\ref{prop:cohexistence}, thus showing the existence of times where the trimmed Birkhoff sums of the derivative bounds  coexist with rigidity. The idea is to produce rigid times by considering visits to the product set $E$ given by Proposition~\ref{prop:accelerationset} and to exploit the memory loss (given by Lemma~\ref{lemma:loss}) to concatenate these visits with visit to a set which gives rigidity (simply given by asking that the first tower has large area) which happen  after finitely many iterations of Rauzy-Veech induction. Finally, the persistence of the trimmed derivative bounds (given by Lemma~\ref{lemma:finitecombination}) is used to show that the bounds still hold at these rigidity times.
\begin{proof}[Proof of Proposition~\ref{prop:cohexistence}]
 We  split the proof in steps for clarity.
\smallskip

\noindent{\it Step 1 (Construction of the rigidity neighbourhoods)}:
Fix a permutation $\pi$ of $d$ intervals and $\epsilon>0$.
 Consider the following open set $\mathcal{N}_1:= \mathcal{N}_1(\epsilon)\subset \Delta_d$ of lengths data
$$
\mathcal{N}_1=\mathcal{N}_1 (\epsilon):= \left\{  \lambda = (\lambda_1,\dots, \lambda_d) \in \Delta_d  \ : \ \lambda_1 >1-\frac{\epsilon}{2} \right\}.
$$
and notice that if $T=(\pi,\lambda)$ has length datum $\lambda\in \mathcal{N}_1(\epsilon)$, then the \emph{displacement}  $\delta_1$ of the first (left-most) interval $I_1$ (given by $\delta_1:={T} x - x$ for any $x\in I_1$)
$T $ 
 satisfies $\delta_1< \epsilon/2$ (since it is  \emph{at most} the sum of the lengths $\lambda_2+\dots + \lambda_d=1-\lambda_1<\epsilon/2$).

Furthermore, we claim that if $\epsilon<1/2$ and an iterate $\mathcal{V}^n(T)\in \mathcal{N}_1(\epsilon) $, then $T$ {has an $\epsilon$-rigid tower} (in the sense of Definition~\ref{def:tow}). Consider indeed the Rohlin tower $\mathcal{T}_n=\mathcal{T}^{(n)}_1$ which has as a base $J_n$ the first continuity interval  $I^{(n)}_1$ of  $\mathcal{V}^n(T)$ and as height $h_n$ the first return time $ h^{(n)}_1$ of $J_n:= I^{(n)}_1$ to $I^{(n)}$ under $T$. Then,  since $\mathcal{V}^n(T)\in \mathcal{N}_1(\epsilon) $,  by definition of $\mathcal{N}_1(\epsilon)$ and the above remark on the displacement, we get (since $\epsilon/2< \epsilon (1-\epsilon/2)$ since we assumed that $\epsilon<1/2$)
\[
d(T^{h_n}x,x)< \frac{\epsilon}{2}|I^{(n)}|\leq \epsilon \left(1-\frac{\epsilon}{2}\right)|I^{(n)}|< \epsilon |J_n|, \qquad  \text{for\ every}\ x\in J_n,
\]
which shows that {$T$ has an $\epsilon$-rigid tower in the sense of Definition~\ref{def:tow}}.

\smallskip
\noindent{\it Step 2 (Definition of the return sets)}: Let $\Delta_B\subset \Delta_d$
and $Y\subset \Theta_\pi$ be a simplex and a positive measure set $Y$ invariant under rescaling as in
 Proposition~\ref{prop:accelerationset} (where $B=B_\gamma$ for some neat path $\gamma$ of length $N$ starting and ending at $\pi$)
 and consider the corresponding product set $
E :={( \{ \pi \} \times \Delta_B \times Y )}^{(1)}$.
\noindent 
 Choose  a sequence ${({\epsilon_\ell})}_{\ell\in \mathbb{N} }$
of decreasing positive numbers  such that $\epsilon_1<1/2$ and
$\epsilon_\ell\to 0$ for $\ell\to+\infty$  and  consider the open sets
$\mathcal{N}_1(\epsilon_\ell)$
defined as in Step 1. By the Markovian loss of memory (see  Lemma~\ref{lemma:loss}), for each $\ell\in \mathbb{N}$ one can find a non-empty open set $U_\ell \subset  \Delta_B$ such that if $T\in
 \{\pi\} \times U_\ell$, then
$\mathcal{V}^{N} (T)\in \{\pi\}\times \mathcal{N}_1(\epsilon_\ell)$ and hence the tower $\cC_1^{(N)}$ is $\epsilon_\ell$-rigid.
Similarly, if $\V^n(T)\in \{\pi\} \times U_\ell$ for some integer $n\geq 0$, then $\mathcal{V}^{n+N} (T)\in \{\pi\}\times \mathcal{N}_1(\epsilon_\ell)$ and hence the tower $\cC_1^{(n+N)}$ is $\epsilon_\ell$-rigid.
 Consider now the sets
\be \label{def:Ek}
E_\ell:= \left ( \{ \pi\} \times U_\ell \times Y \right)^{(1)}\subset E,
\qquad \ell \in \mathbb{N} .
\ee

\smallskip
\noindent{\it Step 3 (Positive measure of the return sets)}:
 We claim that for each $\ell$,  $\widehat{\mu}_{\mathcal{Z}}^{(1)}(E_\ell)>0$.  Indeed,   $U_\ell$ is an open non-empty set in $\Delta_d$ and by assumption $Leb(Y)>0$ (see Proposition~\ref{prop:accelerationset}), so we know that $\{\pi\}\times U_\ell \times Y$ has positive $\widehat{\mu}_{\Z}=\mu_{\Z}\times Leb$-measure. Moreover, since
$Y$ is invariant under rescaling, its intersection with $\bigcup_{t\in[0,1]}Area^{-1}(t)$ has also positive $\widehat{\mu}_{\Z}$-measure. By the definition \eqref{def:mu} of $\widehat{\mu}_{\Z}^{(1)}$, this gives $\widehat{\mu}_{\mathcal{Z}}^{(1)}(E_\ell)>0$.


\smallskip
\noindent {\it Step 4 (Construction of return times for a typical IET):}
Since the invariant measure $\widehat{\mu}_{\Z}^{(1)}$ for
 the Zorich acceleration $\widehat{\mathcal{Z}}$ is finite and ergodic, and by Step 3 the measure with respect to it of any $E_\ell$ is positive, by ergodicity, for every $\ell\in \mathbb{N}$, there is a set  $\widehat{X}_\ell^{(1)}$ with  full measure $\widehat{\mu}_{\mathcal{Z}}^{(1)} (\widehat{X}_\ell^{(1)}) =1$,
such that  for {every} $(\pi,\lambda,\tau)\in \widehat{X}_\ell^{(1)}$
there are infinitely many  $k\in\mathbb{N}$ such that
$\widehat{\mathcal{Z}}^{k} (T) \in E_\ell$.

If we now consider   the intersection $\widehat{X}^{(1)}_\infty:= \bigcap_{\ell\geq 1} \widehat{X}_\ell^{(1)}\subset \widehat{X}^{(1)}$, which has still full measure, for {every} $(\pi,\lambda,\tau)\in \widehat{X}^{(1)}_\infty$ we can construct an increasing sequence $(k_\ell)_{\ell\in\N}$ such that $k_{\ell+1}>k_\ell$  and
$\widehat{\mathcal{Z}}^{k_\ell} (\pi,\lambda,\tau) \in E_\ell$ (once defined inductively $k_1< \cdots <k_\ell $, just choose as $k_{\ell+1}$ one of the infinitely many $k>k_\ell$ such that $\widehat{\mathcal{Z}}^{k} (\pi,\lambda,\tau) \in E_{\ell+1} $).
Since $\widehat{\mathcal{Z}}$ is an acceleration of  $\widehat{\mathcal{V}}$,  each iterate of  $\widehat{\mathcal{Z}}$ is also an iterate of $\widehat{\mathcal{V}}$. Let us denote by $(n_\ell)_{\ell\geq 1}$ the subsequence of iterates of $\widehat{\mathcal{V}}$ which correspond to the subsequence   $(k_\ell)_{\ell\in\N}$  of iterates of $\widehat{\mathcal{Z}}$, i.e.\ such that  $\widehat{\mathcal{V}}^{n_\ell}(\pi,\lambda,\tau) = \widehat{\mathcal{Z}}^{k_\ell}(\pi,\lambda,\tau) $ for every $\ell$. Thus, this sequence is such that
\begin{equation}\label{eq:returns}
\widehat{\mathcal{V}}^{n_\ell} (\pi,\lambda,\tau) \in E_\ell := \left ( \{ \pi\} \times
U_\ell \times Y  \right)^{(1)} \subset E \quad \text{for all} \ \ell \in \mathbb{N}.
\end{equation}
Since the Zorich measure $\mu_\mathcal{Z}$ on the IET space $X$ is given by the pushforward  $\mu_\mathcal{Z} = p_* \widehat{\mu }^{(1)}_\mathcal{Z}$ and is equivalent to the Lebesgue measure on $X$, it follows that {Lebesgue} a.e.~IET $T=(\pi, \lambda)$ has a \emph{lift} $(\pi, \lambda, \tau)$ (i.e.~a triple $(\pi, \lambda, \tau)$ in the fiber $p^{-1}(\pi,\lambda)$ of the projection map $p:\widehat{X}^{(1)}\to X$) which belongs to  $\widehat{X}^{(1)}_\infty$.  Hence, {Lebesgue} a.e.~IET $T$ has a lift for which there exists an increasing sequence $(n_\ell)_{\ell\in\N}$ that satisfies \eqref{eq:returns}.

\smallskip
\noindent {\it Step 5 (Shifted return times are rigidity times)}:  Let $(n_\ell)_{\ell\geq 1}$ be the sequence of visits of the $\mathcal{V}$-orbit of $(\pi, \lambda, \tau)$ to $(E_\ell)_{\ell\in\N}$ built in Step 4.
We now show that the \emph{shifted} sequence
$(m_\ell)_{\ell\in\N}$ of the form $m_\ell:= n_{\ell}+N$, $\ell\in\mathbb{N}$, i.e.~obtained \emph{shifting} the sequence $(n_\ell)_{\ell\in\N}$ of return times to $E_\ell$ constructed as in Step 4  by the length $N$ 
of the path $\gamma$ such that $B=B_\gamma$
gives the desired coexistence times for $T$, independently of the choice of the lift $(\pi,\lambda,\tau)$ of $(\pi,\lambda)$ used to define $(n_\ell)_{\ell\in\N}$.
Let us first show that $(m_\ell)_{\ell\in\N}$
give \emph{rigidity times} for the IET $T=(\pi,\lambda)$.

Notice first that, by the definition of $E_\ell$ and \eqref{eq:returns},
the projection $p \circ \widehat{\mathcal{V}}^{n_{\ell}} (\pi,\lambda,\tau)
= {\mathcal{V}}^{n_{\ell}}(T) $ belongs to  $\{ \pi\} \times U_\ell$. Thus, by the definition of $U_\ell$ (see Step 2), since $m_\ell=n_\ell+N$, we have that  $ {\mathcal{V}}^{m_{\ell}}(T)
 ={\mathcal{V}}^{N}(\mathcal{V}^{n_{\ell}}T)$ belongs to  $\{ \pi\} \times \mathcal{N}_1(\epsilon_\ell)$.
By Step 1, as ${\mathcal{V}}^{m_{\ell}}(T) $ belongs to  $\{ \pi\} \times \mathcal{N}_1(\epsilon_\ell)$, this gives that  the Rohlin towers $\mathcal{T}_\ell: = \mathcal{T}^{(m_\ell)}_1$ over the first inducing interval $I^{(m_\ell)}_1$ are $\epsilon_\ell$ rigid towers for $T$.
We also claim that their measure $|\mathcal{T}_\ell| \to 1$ as $\ell\to\infty$. This follows since the heights $h^{(m_\ell)}_j$ of the Rohlin towers over the inducing interval $I^{(m_{\ell})} $
are \emph{balanced}. Indeed, as  $h^{(m_\ell)} = h^{(n_\ell+N)}= B h^{(n_\ell)}$ and $B$ is a positive matrix, we have
\[\frac{h^{(m_\ell)}_j}{h^{(m_\ell)}_{j'}}\leq \nu(B):=\max_{1\leq i,k,l\leq d}\frac{B_{il}}{B_{kl}}.\]
Since the  measure of the base $I^{(m_\ell)}_1$ of $\mathcal{T}_\ell$ is $|I^{(m_\ell)}_1|> (1-\epsilon_\ell /2)|I^{(m_\ell)}|$ (by the definition of  $\mathcal{N}_1(\epsilon_\ell)$), by the choice of $\epsilon_\ell\to 0$, we have
\begin{align*}
|\mathcal{T}_\ell|&=h^{(m_\ell)}_1|I^{(m_\ell)}_1|=1-\sum_{j=2}^dh^{(m_\ell)}_j|I^{(m_\ell)}_j|\geq 1-\nu(B)h^{(m_\ell)}_1\sum_{j=2}^d|I^{(m_\ell)}_j|\\
&= 1-\nu(B)h^{(m_\ell)}_1(1-|I^{(m_\ell)}_1|)\geq 1-\frac{\epsilon_\ell}{2}\nu(B)h^{(m_\ell)}_1|I^{(m_\ell)}|\geq 1-\frac{\epsilon_\ell}{2}\nu(B) \to 1\quad\text{as}\quad \ell\to\infty.
\end{align*}


\noindent {\it Step 6 (Shifted return times satisfy trimmed Birkhoff sums linear bounds)}:
We will now show that for the IET $T=(\pi,\lambda)$, {every} $f\in\pSymLog{T}$ has linearly bounded trimmed derivatives also along the sequence of rigidity towers $\mathcal{T}_\ell$ defined in Step 5. 
Notice that, by construction, $E^\ell\subset E$ and therefore $(n_{\ell})_{\ell\in\N}$ is by construction a subsequence of the returns for $(\pi,\lambda,\tau)$ to $E={( \{ \pi \} \times \Delta_B \times Y )}^{(1)}$. Thus, by Proposition \ref{prop:accelerationset}, $f$ has linearly bounded trimmed derivatives along any of the (balanced) towers $\mathcal{T}^{(n_\ell)}_j$ for $1\leq j\leq d$. On the other hand,
by construction each tower $\mathcal{T}^{(m_\ell)}_i$ for fixed $1\leq i\leq d$ (so in particular the first tower  $\mathcal{T}_\ell:=\mathcal{T}^{(m_\ell)}_1$)
is obtained by cutting and stacking at most $\Vert B\Vert $ towers   $\mathcal{T}^{(n_\ell)}_j$, $1\leq j\leq d$ (as one can see since $h^{(m_\ell)} = h^{(n_\ell+N)}= B h^{(n_\ell)}$).
Moreover, as we showed in Step 5, the tower $\mathcal{T}_\ell:=\mathcal{T}^{(m_\ell)}_1$ has measure arbitrarily close to $1$ as $\ell\to\infty$.
Thus, both assumptions (A1) and (A2) of Lemma~\ref{lemma:finitecombination} (giving persistence of trimmed derivative bounds) hold. It follows from Lemma~\ref{lemma:finitecombination} that
$f\in\pSymLog{T}$ has linearly bounded trimmed derivatives also along the sequence of  towers $(\mathcal{T}_\ell)_{\ell\in\N}$, which are rigidity towers for $T$ with $|\mathcal{T}_\ell|\to 1$ by Step 5. This concludes the proof.  
\end{proof}

\section{Producing resonant rigid times in given sequences }\label{sec:resontanttimes2}
In this section, we prove Proposition~\ref{prop:3}, which allows to show that any given sequence $(s_n)_{n\in\mathbb N}$ contains infinitely many nearly resonant rigidity times. The proof  exploits crucially a classical distortion control result for Rauzy-Veech induction proved by Kerckhoff in \cite{Ker} and known as \emph{Kerckhoff's Lemma} and ideas used by Chaika in \cite{Cha}.
We will use throughout this section the notation and preliminaries introduced in \S~\ref{sec:distortion_background}.

\subsection{Preliminary constructions}
Let us first introduce the type of IETs which we will use to produce (balanced) rigid times.
We first need a result which allows to approximate sets by Rauzy-Veech simplexes.

\subsubsection{Approximation by simplexes}
The following technical result shows that we can produce good approximations (in measure) of {every} closed set by a union of simplexes which are cylinders for Rauzy-Veech induction. The proof will take the rest of this subsection.
\begin{lemma}\label{prop:extF}
For {every} closed subset $F\subset X(\mathcal R)$ {of positive measure} and {every} $\delta>0$ there exists $n=n_{F,\delta}\geq 1$ and a subset $\Gamma_F\subset \Pi(\mathcal R)$ of {paths} of length $n$ such that
\begin{equation}
F\subset \bigcup_{\gamma\in\Gamma_F}\Delta_\gamma \quad\text{ and }\quad m(\bigcup_{\gamma\in\Gamma_F}\Delta_\gamma)< (1+\delta)m(F).
\end{equation}
\end{lemma}

\begin{proof}
Let $U:=X(\mathcal R)\setminus F$. In view of \eqref{eq:decay}, for $m$-a.e.\ $(\pi,\lambda)\in U$ there exists $\gamma(\pi,\lambda)\in\Pi(\mathcal R)$ such that
\[(\pi,\lambda)\in \Delta_{\gamma(\pi,\lambda)}\subset U.\]
Choose $\gamma(\pi,\lambda)\in\Pi(\mathcal R)$ such that its length $n(\pi,\lambda)=|\gamma(\pi,\lambda)|$ is minimal. Then $n:U\to\N$ is defined a.e.\ and measurable. Hence, there exists $n\in\N$ such that
\begin{equation}\label{eq:Udelta}
\text{if}\quad U_\delta:=\{(\pi,\lambda)\in U:n(\pi,\lambda)\leq n\}\quad\text{then}\quad m(U_\delta)>m(U)-\delta m(F).
\end{equation}
Then for every $(\pi,\lambda)\in U_\delta$ there exists $\gamma_n(\pi,\lambda)\in\Pi(\mathcal R)$ of length $n$ such that $\gamma(\pi,\lambda)$ is a prefix of $\gamma_n(\pi,\lambda)$ and
\[(\pi,\lambda)\in \Delta_{\gamma_n(\pi,\lambda)}\subset\Delta_{\gamma(\pi,\lambda)}\subset U.\]
Let $\Gamma_U:=\{\gamma_n(\pi,\lambda):(\pi,\lambda)\in U_\delta\}$. By definition,
\[U_\delta\subset\bigcup_{\gamma\in \Gamma_U}\Delta_\gamma\subset U.\]
Let $\Gamma_F$ be the complement of $\Gamma_U$ in the set of {paths} in $\Pi(\mathcal R)$ of length $n$. Then
\[F\subset\bigcup_{\gamma\in \Gamma_F}\Delta_\gamma\subset X(\mathcal R)\setminus U_\delta.\]
In view of \eqref{eq:Udelta}, it follows that
\[m(\bigcup_{\gamma\in \Gamma_F}\Delta_\gamma)\leq m(X(\mathcal R)\setminus U_\delta)= m(F)+m(U)-m(U_\delta)< (1+\delta)m(F).\]
\end{proof}

\subsubsection{Selecting distortion controlled simplexes}
In this subsection we  show that we can select a (finite) subset of paths such that the corresponding simplexes give a good approximation in measure and have good distortion properties given by Kerckhoff's Lemma. Recall that, by Proposition~\ref{prop:AGY} (Kerckhoff's Lemma), for {every} $\gamma\in \Pi(\mathcal R)$ there exists a finite subset $\Gamma(\gamma)\subset \Pi(\mathcal R)$ of paths starting from $\gamma$  and having good distortion properties.
\begin{lemma}
For {every} $\gamma_0\in\Pi(\mathcal R)$, $0<\delta<1$, {and {every} real $r$ large enough, there exists a finite collection $\tilde\Gamma_r=\tilde\Gamma_{\gamma_0,\delta,r}$} of finite paths $\gamma$ on the Rauzy class $\mathcal{R}$ such that:
\begin{itemize}
\item[(i)] all paths in $\tilde\Gamma_r$ are prefix-free and have prefix $\gamma_0$;

\item[(ii)] for {every} $\gamma\in \tilde\Gamma_r$, $B^*_{\gamma}$ is $C$-balanced, and $\gamma$ has a prefix path $\gamma_s$, which is prefixed by $\gamma_0$, such that
\begin{equation*}
r\leq |C_{\max}(B^*_{\gamma_s})|<2r,\quad \text{and} \quad|C_{\max}(B^*_{\gamma})|\leq C|C_{\max}(B^*_{\gamma_s})|;
\end{equation*}

\item[(iii)] and we have,
$$
m(\bigcup_{\gamma\in\tilde\Gamma_r}\Delta_\gamma)>C^{-1}(1-\delta)m(\Delta_{\gamma_0}).$$
\end{itemize}
\end{lemma}
\begin{proof}
Fix any $\gamma_0\in\Pi(\mathcal R)$ and $0<\delta<1$.  Take any $r\geq |C_{\max}(B^*_{\gamma_0})|$. By \eqref{eq:decay}, for $m$-a.e.\ $(\pi,\lambda)\in \Delta_{\gamma_0}$ there exists $\gamma(\pi,\lambda)\in\Pi(\mathcal R)$ such that
\[(\pi,\lambda)\in \Delta_{\gamma(\pi,\lambda)},\quad |C_{\max}(B^*_{\gamma(\pi,\lambda)})|\geq r\quad\text{and}\quad |\gamma(\pi,\lambda)|\quad \text{is minimal.}\]
Then $\gamma_0$ is a prefix of $\gamma(\pi,\lambda)$. In view of \eqref{eq:MB}, we have
\begin{equation}\label{eq:r2r}
r\leq|C_{\max}(B^*_{\gamma(\pi,\lambda)})|< 2r.
\end{equation}
Moreover, by \eqref{eq:growthB}, we have
\begin{equation}\label{eq:log2r}
\log_2r\leq|\gamma(\pi,\lambda)|.
\end{equation}
Let us hence define:
\[\Gamma'_r:=\{\gamma(\pi,\lambda):\text{for all $(\pi,\lambda)\in \Delta_{\gamma_0}$ for which $\gamma(\pi,\lambda)$ is well defined}\}.\]
Then $\Gamma'_r$ is prefix-free,
\[\bigcup_{\gamma\in \Gamma'_r}\Delta_\gamma\subset\Delta_{\gamma_0}\quad \text{and}\quad
\sum_{\gamma\in \Gamma'_r}m(\Delta_\gamma)=m(\Delta_{\gamma_0}).\]
Finally choose a finite subset $\Gamma_r\subset \Gamma'_r$ such that
\begin{equation}\label{eq:Gr}
\sum_{\gamma\in \Gamma_{r}}m(\Delta_\gamma)\geq(1-\delta)m(\Delta_{\gamma_0}).
\end{equation}
Let
\[\tilde\Gamma_r=\tilde\Gamma_{\gamma_0,\delta,r}:=\bigcup_{\gamma\in\Gamma_{r}}\Gamma(\gamma).\]
Recall that, each $\Gamma(\gamma)\subset \Pi(\mathcal R)$ is a finite prefix-free set of paths starting with $\gamma$, given by Proposition~\ref{prop:AGY}.
As $\Gamma_r$ and $\Gamma(\gamma)$ are prefix-free and finite, the set $\tilde\Gamma_r$ is also prefix-free and finite. Moreover, for {every} $\gamma\in \tilde\Gamma_r$ we can find unique {paths} $\gamma_s$ and $\gamma_e$ such that $\gamma=\gamma_s\gamma_e$ with $\gamma_s\in\Gamma_r$ and {$\gamma=\gamma_s\gamma_e\in\Gamma(\gamma_s)$}. As $\gamma_s\in\Gamma_r$, by \eqref{eq:r2r}, we have $r\leq |C_{\max}(B^*_{\gamma_s})|<2r$. As {$\gamma=\gamma_s\gamma_e\in\Gamma(\gamma_s)$},
by (K2) and (K3) in Proposition~\ref{prop:AGY}, we have that $B^*_{\gamma}$ is $C$-balanced and $|C_{\max}(B^*_{\gamma})|\leq C|C_{\max}(B^*_{\gamma_s})|$.
In summary,  then $(i)$ and $(ii)$ hold. We are left to prove $(iii)$. For this,
note that, by \eqref{eq:inDC} and \eqref{eq:Gr}, we have
\begin{equation}\label{eq:measbigg}
m(\bigcup_{\gamma\in\tilde\Gamma_r}\Delta_\gamma)=\sum_{\gamma_s\in\Gamma_r}m(\Delta_{\gamma_s})
{\sum_{\gamma\in\Gamma(\gamma_s)}m(\Delta_{\gamma}|\Delta_{\gamma_s})}>C^{-1}(1-\delta)m(\Delta_{\gamma_0}).
\end{equation}
This proves $(iii)$ and thus concludes the proof.
\end{proof}

\subsubsection{Constructing balanced rigid times} \label{sec:Agammadef}
For {every} $1\leq k\leq d$ and $0<\epsilon<1$ let
\[F_k=F_k(\epsilon):=\mathcal R\times\{\lambda\in\Delta_d:\lambda_k>1-\epsilon/3\}.\]
Then $m(F_k(\epsilon))=m(F_1(\epsilon))\asymp \epsilon^{d-1}$. For {every} $\gamma\in\tilde\Gamma_r$ let
\[1\leq k(\gamma)\leq d\quad\text{such that}\quad C_{k(\gamma)}(B^*_\gamma)=C_{\max}(B^*_\gamma).\]
In view of the approximation Lemma~\ref{prop:extF}, {applied to the closure of $F_k(\epsilon)$,} there exists natural $n_{\epsilon,\delta}$ and for every $1\leq k\leq d$ there exists a subset
$\Gamma_{F_k(\epsilon)}$ of {paths} of length $n_{\epsilon,\delta}$ such that
{
\begin{equation}\label{eq:Fkep}
F_k(\epsilon)\subset \overline{F_k(\epsilon)}\subset \bigcup_{\gamma\in\Gamma_{F_k(\epsilon)}}\Delta_\gamma \quad\text{ and }\quad m(\bigcup_{\gamma\in\Gamma_{F_k(\epsilon)}}\Delta_\gamma)< (1+\delta)m(\overline{F_k(\epsilon)})=(1+\delta)m(F_k(\epsilon)).
\end{equation}
}
Let us consider the following set
\begin{equation}\label{def:Ar}
 A_r=A^{\gamma_0}_r(\epsilon):=\bigcup_{\gamma\in \tilde\Gamma_r}B^*_{\gamma}F_{k(\gamma)}(\epsilon).
\end{equation}
As for every $\gamma\in \tilde\Gamma_r$ the path $\gamma_0$ is a prefix of $\gamma$, we have $A^{\gamma_0}_r(\epsilon)\subset \Delta_{\gamma_0}$.  We now want to show that IETs in this set have rigidity times in a controlled range:
\begin{lemma}[simplex for balanced rigidity]\label{lem:mAr}
For {every} $0<\vep<1$,  the set  $A^{\gamma_0}_r(\epsilon)$ defined above has the following properties:
\begin{itemize}\item[(1)] if $T=(\pi,\lambda)\in A^{\gamma_0}_r(\epsilon)$ then there exists an $\epsilon$-rigidity time $q$ for $T$ with $r\leq q\leq 2Cr$;

\smallskip
\item[(2)] the measure $m(A^{\gamma_0}_r(\epsilon))$ satisfies:
$$m(A^{\gamma_0}_r(\epsilon))\geq C^{-d-1}{(\# \mathcal{R})^{-1}}(1-\delta)m(\Delta_{\gamma_0})m(F_1(\epsilon)).
$$
\end{itemize}
\end{lemma}
\begin{proof}
Let $\gamma\in\tilde\Gamma_r$.
Let first prove part $(1)$, namely that the IET has a rigidity time in $[r,2Cr]$.
If $(\pi,\lambda)\in B^*_{\gamma}F_{k(\gamma)}(\epsilon)$ for some $\gamma\in\tilde\Gamma_r$ of length $n$, then $(\pi^{(n)},\lambda^{(n)})=\mathcal V^n(\pi,\lambda)$ is such that $\lambda^{(n)}_{k(\gamma)}>1-\epsilon/3$. Let $q:=|C_{k(\gamma)}(B^*_\gamma)|=|C_{\max}(B^*_\gamma)|$. Then the IET $(\pi,\lambda)$ has a tower of intervals with height $q$ and whose measure is greater than $1-\epsilon/3$. Therefore,  $q$ is an $\epsilon$-rigidity time. Moreover,
\[r\leq |C_{\max}(B^*_{\gamma_s})|\leq q=|C_{\max}(B^*_\gamma)|\leq C|C_{\max}(B^*_{\gamma_s})|<2Cr.\]


\smallskip
Let us now prove $(2)$, namely the measure bound. As $B^*_{\gamma}$ is $C$-balanced, in view of \eqref{eq:balmeas},
\[m(B^*_{\gamma}F_{k(\gamma)}(\epsilon))\geq C^{-d}(\# \mathcal{R})^{-1}m(F_1(\epsilon))m(\Delta_{\gamma}).\]
It follows that
\begin{align*}
m(A^{\gamma_0}_r(\epsilon))&=\sum_{\gamma\in\tilde\Gamma_r}m(B^*_{\gamma}F_{k(\gamma)}(\epsilon))
\geq C^{-d}(\# \mathcal{R})^{-1}m(F_1(\epsilon))\sum_{\gamma\in\tilde\Gamma_r}m(\Delta_{\gamma})\\
&=C^{-d}(\# \mathcal{R})^{-1}m(F_1(\epsilon))m(\bigcup_{\gamma\in\tilde\Gamma_r}\Delta_{\gamma}).
\end{align*}
In view of \eqref{eq:measbigg}, this proves the bound in $(2)$.
\end{proof}

\subsubsection{Chung-Erd{\"os} inequality}
Finally, we recall
a classical  probabilistic
estimate, known as Chung-Erd{\"os} inequality, which will be used to estimate measures in the final proof.
\begin{lemma}[Chung-Erd\"os inequality, \cite{Ch-Er}] Let $\{A_i\}$ be events in a probability space. Then
\begin{equation}\label{eq:ChEr}
\mu( \bigcup_{i=1}^n A_i)\geq \frac{(\sum_{i=1}^n \mu(A_i))^2}{\sum_{i,j=1}^n \mu(A_i\cap A_j)}.
\end{equation}
\end{lemma}
\noindent The proof of this classical result can be found in~\cite{Ch-Er}.
\subsection{Proof of Proposition~\ref{prop:3}}
To conclude the proof of Proposition~\ref{prop:3} (which will be presented in the final \S~\ref{sec:proofProp3}), we first need to estimate the measure of the intersection of the sets $A_r^\gamma(\epsilon)$   constructed in the previous subsection.
\subsubsection{Quasi-independence of the sets $A_r^\gamma(\epsilon)$}
Let us show that the sets $A_r^\gamma(\epsilon)$, for different values of $r$ which are sufficiently far apart,  are \emph{quasi-independent}  in a sense made precise by the following Lemma~\ref{prop:Aindep}.
For {every} $\gamma_0\in\Pi(\mathcal R)$, $0<\delta<1$ and {$r> 0 $}, let us define
\begin{equation}\label{eq:defNs}
\overline{N}_r^{\gamma_0}:=\max\{|\gamma|:\gamma\in \tilde\Gamma_{\gamma_0,\delta,r}\}\quad\text{and}\quad \underline{N}^{\gamma_0}_{r}:=\min\{|\gamma|:\gamma\in\tilde\Gamma_{\gamma_0,\delta,r}\}\geq \log_2 r,
\end{equation}
where the last lower estimate on the growth of $\underline{N}^{\gamma_0}_{r}$ follows from \eqref{eq:log2r}.
\begin{lemma}[quasi-independence estimate]\label{prop:Aindep}
Suppose that $|C_{\max}(B^*_{\gamma_0})|\leq r<r'$ are such that $\underline{N}^{\gamma_0}_{r'}>\overline{N}^{\gamma_0}_r+n_{\epsilon,\delta}$, where $n_{\epsilon,\delta}$ is the common length of paths in the sets $\Gamma_{F_k(\epsilon)}$, $1\leq k\leq d$. Then
\begin{equation}\label{eq:Aindep}
m(A^{\gamma_0}_r(\epsilon)\cap A^{\gamma_0}_{r'}(\epsilon))\leq \widehat{C}\frac{m(A^{\gamma_0}_r(\epsilon))m(A^{\gamma_0}_{r'}(\epsilon))}{m(\Delta_{\gamma_0})},\quad\text{where}\quad \widehat C:=C^{4d+1}\frac{1+\delta}{1-\delta}.
\end{equation}
\end{lemma}

\begin{proof}
By the definition of $A_r$ (see \eqref{def:Ar}) and \eqref{eq:Fkep}, we have
\begin{align*}m(A_r\cap A_{r'})&\leq \sum_{\gamma\in\tilde\Gamma_{r}}\sum_{\gamma'\in\tilde\Gamma_{r'}}m(B^*_\gamma F_{k(\gamma)}\cap B^*_{\gamma'} F_{k(\gamma')})\\
&\leq \sum_{\gamma\in\tilde\Gamma_{r}}\sum_{\gamma''\in \Gamma_{F_{k(\gamma)}}}\sum_{\gamma'\in\tilde\Gamma_{r'}}m(B^*_\gamma \Delta_{\gamma''}\cap B^*_{\gamma'} F_{k(\gamma')})
.\end{align*}
As $B^*_\gamma \Delta_{\gamma''}=\Delta_{\gamma\gamma''}$ and $|\gamma\gamma''|\leq \overline{N}^{\gamma_0}_r+n_{\epsilon,\delta}<\underline{N}^{\gamma_0}_{r'}\leq |\gamma'|$,
if $m(B^*_\gamma \Delta_{\gamma''}\cap B^*_{\gamma'} F_{k(\gamma')})>0$ then $\gamma\gamma''$ is a prefix of $\gamma'$ and $B^*_\gamma \Delta_{\gamma''}\cap B^*_{\gamma'} F_{k(\gamma')}=B^*_{\gamma'} F_{k(\gamma')}$. Hence,
\begin{align*}
m(A_r\cap A_{r'})&
\leq \sum_{\gamma\in\tilde\Gamma_{r}}\sum_{\gamma''\in \Gamma_{F_{k(\gamma)}}}\sum_{\substack{\gamma'\in\tilde\Gamma_{r'}\\
\gamma\gamma''\text{ prefixes }\gamma'}}m(B^*_{\gamma'} F_{k(\gamma')}).
\end{align*}
As $B^*_{\gamma'}$ is $C$-balanced, in view of \eqref{eq:balmeas},
\[m(B^*_{\gamma'}F_{k(\gamma')}(\epsilon))\leq C^{d}(\# \mathcal{R})^{-1}m(F_1(\epsilon))m(\Delta_{\gamma'}).\]
It follows that
\begin{align*}
m(A_r\cap A_{r'})&
\leq C^{d}(\# \mathcal{R})^{-1}m(F_1(\epsilon))\sum_{\gamma\in\tilde\Gamma_{r}}\sum_{\gamma''\in \Gamma_{F_{k(\gamma)}}}\sum_{\substack{\gamma'\in\tilde\Gamma_{r'}\\
\gamma\gamma''\text{ prefixes }\gamma'}}m(\Delta_{\gamma'})\\ &
\leq C^{d}(\# \mathcal{R})^{-1}m(F_1(\epsilon))\sum_{\gamma\in\tilde\Gamma_{r}}\sum_{\gamma''\in \Gamma_{F_{k(\gamma)}}}m(\Delta_{\gamma\gamma''})\\
&
\leq C^{d}(\# \mathcal{R})^{-1}m(F_1(\epsilon))\sum_{\gamma\in\tilde\Gamma_{r}}m(B^*_\gamma\bigcup_{\gamma''\in \Gamma_{F_{k(\gamma)}}}\Delta_{\gamma''}).
\end{align*}
As $B^*_{\gamma}$ is $C$-balanced, in view of \eqref{eq:balmeas} and \eqref{eq:Fkep},
\begin{align*}
m(B^*_{\gamma}\bigcup_{\gamma''\in \Gamma_{F_{k(\gamma)}}}\Delta_{\gamma''})
&\leq C^{d}m(\bigcup_{\gamma''\in \Gamma_{F_{k(\gamma)}}}\Delta_{\gamma''})m(\Delta_{\gamma})\\
&< C^{d}(1+\delta)m(F_{k(\gamma)}(\epsilon))m(\Delta_{\gamma})\leq C^{2d}(1+\delta)m(B^*_{\gamma}F_{k(\gamma)}(\epsilon))
.
\end{align*}
It follows that
\begin{align*}
m(A_r\cap A_{r'})&\leq  (1+\delta)C^{3d}(\# \mathcal{R})^{-1}m(F_1(\epsilon))m(\bigcup_{\gamma\in\tilde\Gamma_{r}}B^*_{\gamma}F_{k(\gamma)}(\epsilon))\\
&=
(1+\delta)C^{3d}(\# \mathcal{R})^{-1}m(F_1(\epsilon))m(A_r).
\end{align*}
In view of the measure bound given by  part $(2)$ of Lemma \ref{lem:mAr}, this gives \eqref{eq:Aindep}.
\end{proof}

\subsubsection{Final arguments}\label{sec:proofProp3}
 We now have all the ingredients needed for the proof of  Proposition~\ref{prop:3}.
\begin{proof}[Proof of Proposition~\ref{prop:3}]
We will prove that the set $\Omega=\Omega(\underline{s},\epsilon)$ has full measure by \emph{localizing} it and estimating its local measure using  partitions made by the sets $A^{\gamma}_{r}(\epsilon)$  built in the previous subsection.

\noindent {\it Step 1: Localizing $\Omega$.}
For {every} path $\gamma\in\Pi(\mathcal R)$,
denote by $\Omega| \Delta_\gamma:= \Omega \cap \Delta_\gamma$  the restriction of $\Omega$ to $\Delta_\gamma$. We claim that
\begin{equation}\label{eq:eqnnu}
\Omega|  \Delta_\gamma \supset
\bigcap_{m\geq 1}\bigcup_{n\geq m}A^{\gamma}_{s_n}(\epsilon).
\end{equation}
To see this inclusion, recall that by Lemma~\ref{lem:mAr},  if the IET $T=:(\pi,\lambda)\in A^{\gamma}_{s_n}(\epsilon)$ then $T$ has an $\epsilon$-rigidity time in $[s_n, 2Cs_n]$. Thus, if $(\pi,\lambda)$ belongs to the RHS of \eqref{eq:eqnnu}, for every $j\in \mathbb{N}$, there exist infinitely many $n$ such that  $T$ has an $\epsilon$-rigidity time in $[s_n, 2Cs_n]$. In particular, we can build a sequence $(q_j)_{j\in \mathbb{N}}$ where, for each $j\in\mathbb{N}$, $q_j$ is an $\epsilon$-rigidity time, $q_{j}>q_{j-1}$ and $n_j$ is chosen so that $q_j \in [s_{n_j}, 2Cs_{n_j}]$. 
Recalling the definition of $\Omega (\underline{s}, {\epsilon})$, this precisely shows that ($\Omega 1$) and  ($\Omega 2$) hold and hence that $T$ is in $\Omega$. Moreover, from the definition of the set $ A^{\gamma}_{s_n}(\epsilon)$,  each $ A^{\gamma}_{s_n}(\epsilon)\subset \Delta_\gamma$ (see \S~\ref{sec:Agammadef}). Thus, $T$ belongs to $\Omega\cap \Delta_\gamma$, i.e.~to the  LHS of \eqref{eq:eqnnu}, as claimed.

\medskip \noindent
{\it Step 2. Estimating the local measure.}
We claim that 
\begin{equation}\label{eq:omega}
m(\Omega|\Delta_\gamma)\geq \widehat{C}^{-1}\quad\text{for {every}}\quad \gamma\in\Pi(\mathcal R).
\end{equation}
The rest of this Step will be taken by the proof, which uses quasi-independence and Chung-Erd{\"o}s inequality.
Passing to a subsequence in $(s_n)_{n\in\N}$, if necessary, we can assume that the quantities $\overline{N}^{\gamma}_{s_n}$ and $\underline{N}^{\gamma}_{s_n}$ defined as in \eqref{eq:defNs} satisfy
\[
\underline{N}^{\gamma}_{s_{n+1}}>\overline{N}^{\gamma}_{s_n}+n_{\epsilon,\delta}\quad \text{ for {all}}\quad  n\geq 1.
\]
Then, in view of the quasi-independence given by Lemma~\ref{prop:Aindep} under this assumption, we have that
\begin{align*}
\frac{\left(\sum_{i=1}^nm(A^{\gamma}_{s_i}(\epsilon))\right)^2}{\sum_{i=1}^n\sum_{j=1}^nm(A^{\gamma}_{s_i}(\epsilon)\cap A^{\gamma}_{s_j}(\epsilon))}&\geq
\frac{2\sum_{1\leq i<j\leq n}m(A^{\gamma}_{s_i}(\epsilon))m(A^{\gamma}_{s_j}(\epsilon))}{\sum_{i=1}^nm(A^{\gamma}_{s_i}(\epsilon))+\frac{2\widehat{C}}{m(\Delta_{\gamma})}\sum_{1\leq i<j\leq n}m(A^{\gamma}_{s_i}(\epsilon)) m(A^{\gamma}_{s_j}(\epsilon))}\\
&\geq \frac{m(\Delta_{\gamma})}{\widehat{C}}
\frac{1}{1+\frac{m(\Delta_{\gamma})}{\widehat{C}}\frac{\sum_{i=1}^nm(A^{\gamma}_{s_i}(\epsilon))}{2\sum_{1\leq i<j\leq n}m(A^{\gamma}_{s_i}(\epsilon)) m(A^{\gamma}_{s_j}(\epsilon))}}.
\end{align*}
Moreover, by the measure bound given by part $(2)$ of Lemma~\ref{lem:mAr}, we have
\begin{equation*}
\frac{\sum_{i=1}^nm(A^{\gamma}_{s_i}(\epsilon))}{2\sum_{1\leq i<j\leq n}m(A^{\gamma}_{s_i}(\epsilon)) m(A^{\gamma}_{s_j}(\epsilon))}\leq \frac{nm(\Delta_{\gamma})}{n(n-1) C^{-2d-2}(\# \mathcal{R})^{-2}(1-\delta)^2m(\Delta_{\gamma})^2m(F_1(\epsilon))^2}\to 0
\end{equation*}
as $n\to+\infty$. Hence,
\[\liminf_{n\to\infty}\frac{\left(\sum_{i=1}^nm(A^{\gamma}_{s_i}(\epsilon))\right)^2}{\sum_{i=1}^n\sum_{j=1}^nm(A^{\gamma}_{s_i}(\epsilon)\cap A^{\gamma}_{s_j}(\epsilon))}\geq \frac{m(\Delta_{\gamma})}{\widehat{C}}.\]
Using Chung-Erd\"os inequality (see \eqref{eq:ChEr}),
this gives
\[m\Big(\bigcap_{m\geq 1}\bigcup_{n\geq m}A^{\gamma}_{s_n}(\epsilon)\Big)\geq \frac{m(\Delta_{\gamma})}{\widehat{C}}.\]
In view of Step 1 this yields \eqref{eq:omega} as claimed.

\medskip
\noindent {\it Step 3: Final arguments.}
Let us consider a filtration $(\mathcal F_n)_{n\geq 1}$ of sub-$\sigma$-algebras of the $\sigma$-algebra $\mathcal F$ of Lebesgue-measurable subsets of $X(\mathcal R)$ so that $\mathcal F_n$ is generated by $\Delta_\gamma$ for all $\gamma\in\Pi(\mathcal R)$ of length $n$. {In view of \eqref{eq:decay}}, $\mathcal F$ is the smallest $\sigma$-algebra containing all $\mathcal F_n$, $n\geq 1$.
Therefore, by L\'evy's zero-one law, we have that the conditional expectations
\[\mathbb E(\mathbf 1_\Omega|\mathcal F_n)\to \mathbb E(\mathbf 1_\Omega|\mathcal F)=\mathbf 1_\Omega\quad\text{$m$-a.e.\ as $n\to\infty$.}\]
Since, by Step 2, we have that $\mathbb E(\mathbf 1_\Omega|\mathcal F_n)\geq \widehat{C}^{-1}$ for {all} $n\geq 1$, it then  follows from the above convergence that
$$\mathbf 1_\Omega =  \mathbb E(\mathbf 1_\Omega|\mathcal F)\geq \frac{1}{\widehat{C}}>0\qquad m\text{-almost-everywhere}. $$ Therefore (since $1_\Omega$ takes only $0$ or $1$ as values) we conclude that $m$-almost everywhere $\mathbf 1_\Omega\equiv 1$ and hence $m(\Omega)=m(X(\mathcal R))$, i.e.~$\Omega\subset X(\mathcal R)$ has full $m$-measure.
\end{proof}

\appendix


\section{Deviations of ergodic averages}\label{app:deviations}
In this Appendix we show how Proposition~\ref{thm:asxi}  can be deduced from the recent work of the first and last authors \cite{Fr-Ul24}.
Let us first summarize in \S~\ref{sec:results}
 some of the results from \cite{Fr-Ul24},  from which we can deduce in \S~\ref{sec:proofdeviations} the proof of Proposition~\ref{thm:asxi}. 

\subsection{Results on corrections of logarithmic cocycles.} \label{sec:results}
We work with the full measure set of IETs which satisfy the Diophantine-like condition defined in \cite[\S 3.2]{Fr-Ul24} and called UDC (for \emph{Uniform Diophantine Condition}).  By  \cite[Theorem~3.8]{Fr-Ul24}, almost every IET satisfies this UDC condition.

Associated with each IET $T$ satisfying the UDC is a descending sequence of intervals $(I^{(k)})_{k\geq 1}$ determining the corresponding renormalizations of $T$. For {every} $k\geq 1$ and {every} $\varphi\in \mathcal{L}og(T)$ let $S(k)\varphi:I^{(k)}\to\R$ be the induced cocycle defined by
\[S(k)\varphi(x)=S_{h^{(k)}_j}\varphi(x)\quad\text{for {all}}\quad x\in I_j^{(k)},\ 1\leq j\leq d,\]
where $h_j^{(k)}$ is the first return time of $I_j^{(k)}$ to $I^{(k)}$. For {every} $T\in X(\mathcal R)$ satisfying the UDC there exist piecewise constant (constant on exchanged intervals) maps $h_1,\ldots, h_g$ such that
\[\lim_{k\to\infty}\frac{\log\|S(k)h_i\|_{L^{\infty}(I^{(k)})}}{k}={\lambda_i}\quad\text{for}\quad 1\leq i\leq g,\]
where $0<\lambda_g<\ldots<\lambda_2<\lambda_1$ are the positive Lyapunov exponents related to $X(\mathcal R)$. Standard arguments show that the means of $h_2,\ldots, h_g$ are zero and the mean of $h_1$ is non-zero.

In view of Theorem~6.1, Corollary~7.11 and (7.12) (all) in \cite{Fr-Ul24}, we have the following:
\begin{proposition}\label{prop:FU1}
For {every} IET $T$ satisfying the UDC there exist bounded functionals $d_i:\Log{T}\to\R$, $1\leq i\leq {g}$ such that $d_j(h_i)=\delta_{ij}$ for {all} $1\leq i,j\leq g$ and  for {every} $\varphi\in \Log{T}$, we have
\[\limsup_{k\to\infty}\frac{\log \frac{1}{|I^{(k)}|}\|S(k)(\varphi-\sum_{i=1}^gd_i(\varphi)h_i)\|_{L^1(I^{(k)})}}{k}\leq 0.\]
If $d_i(\varphi)=0$ for {all} $1\leq i\leq g$, then $\int_I\varphi(x)dx=0$.
\end{proposition}

\smallskip

\begin{proposition}[Proposition~7.10 in \cite{Fr-Ul24}]\label{prop:FU2}
Suppose that the IET $T:I\to I$ satisfies the UDC and the roof function $f:I\to\mathbb R_{>0}$ is in $\Log{T}$. Let $\xi:I^f\to\R$ be a measurable bounded map such that
$\varphi_\xi\in \mathcal{L}og(T)$ and
\[\limsup_{k\to\infty}\frac{\log\frac{1}{|I^{(k)}|}\|S(k)\varphi_\xi\|_{L^{1}(I^{(k)})}}{k}\leq 0.\]
Then
for a.e.\ $(x,r)\in I^f$, we have
\begin{equation*}\label{eq:Birk0}
\limsup_{t\to+\infty}\frac{\log|\int_{0}^t\xi(T^f_s(x,r))ds|}{\log t}\leq 0.
\end{equation*}
\end{proposition}

For {every} $1\leq i\leq g$, choose a bounded measurable function $\xi_i:I^f\to\R$ such that, we have $\varphi_{\xi_i}=h_i$ and there exists $K>0$ for which $\xi_i(x,r)=0$ for $r\geq K$.

\begin{proposition}[Proposition~7.12 in \cite{Fr-Ul24}]\label{prop:FU3}
Suppose that the IET $T:I\to I$ satisfies the UDC and the roof function $f:I\to\mathbb R_{>0}$ is in $\Log{T}$.
Then
\begin{equation}\label{eq:FU3}
\limsup_{t\to+\infty}\frac{\log\|\int_{0}^t\xi_i\circ T^f_s\,ds\|_{L^\infty}}{\log t}= \frac{\lambda_i}{\lambda_1}\quad\text{for}\quad 1\leq i\leq g.
\end{equation}
\end{proposition}

\subsection{Deduction of Proposition~\ref{thm:asxi}}\label{sec:proofdeviations}
\noindent We now have all the elements for the proof.
\begin{proof}[Proof of Proposition~\ref{thm:asxi}]
By assumption, $\int_I\varphi_\xi(x)dx=\int_{I^f}\xi(x,r)dxdr=0$. By definition,
\[d_j(\varphi_\xi-\sum_{i=1}^gd_i(\varphi_\xi)h_i)=0\quad\text{for all}\quad 1\leq j\leq g.\]
In view of the second part of Proposition~\ref{prop:FU1}, the mean of $\varphi_\xi-\sum_{i=1}^gd_i(\varphi_\xi)h_i$ is zero. As the mean of $h_2,\ldots, h_g$ and $\varphi_\xi$ are zero, it follows that the mean of $d_1(\varphi_\xi)h_1$ is also zero, so $d_1(\varphi_\xi)=0$. In view of Proposition~\ref{prop:FU1}, this gives
\begin{gather*}\limsup_{k\to\infty}\frac{\log \frac{1}{|I^{(k)}|}\|S(k)(\varphi_{\xi-\sum_{i=2}^gd_i(\varphi)\xi_i})\|_{L^1(I^{(k)})}}{k}\\
=\limsup_{k\to\infty}\frac{\log \frac{1}{|I^{(k)}|}\|S(k)(\varphi_{\xi}-\sum_{i=2}^gd_i(\varphi)h_i)\|_{L^1(I^{(k)})}}{k}\leq 0.
\end{gather*}
By Proposition~\ref{prop:FU2},
for a.e.\ $(x,r)\in I^f$, we have
\begin{equation*}
\limsup_{t\to+\infty}\frac{\log|\int_{0}^t\xi(T^f_s(x,r))ds-d_i(\varphi)\sum_{i=2}^g\int_{0}^t\xi_i(T^f_s(x,r))ds|}{\log t}\leq 0.
\end{equation*}
In view of \eqref{eq:FU3}, this shows \eqref{eq:la2/la1}.
\end{proof}

\section{Roth-type IETs distortion}\label{app:Roth}
In this Appendix we briefly recall the definition of Roth-type IET as given by Marmi, Moussa and Yoccoz in \cite{MMY:Coh} and show that one of the consequences of being Roth-type introduced in \cite{MMY:Coh}  implies the result that we stated as Lemma~\ref{lemma:Roth_distortion} and took as working definition of Roth-type in \S~\ref{sec:orthogonalityproof}.

\subsection{Roth-type IETs}\label{sec:Rothdef}
We refer in this section to the notation and definitions in \cite{MMY:Coh}. Let $T$ be an IET with no-connection, i.e.~an IET which satisfies the Keane condition. In~\cite{MMY:Coh} the continuity intervals of $I$ are denoted by $I_\alpha, \alpha\in \mathcal{A}$, where $\mathcal{A} $ is an alphabet of cardinality $d$ whose elements are called \emph{names} and the combinatorial datum keeps memory of the \emph{name} of the interval\footnote{This different convention plays a crucial role in the study of positivity of the matrices of the Rauzy-Veech cocycle, but we chose to use the simpler convention of labelling the intervals $I_1,\dots, I_j, \dots, I_d $ in this paper since the names of intervals play no role in our work.}. We say that a name \emph{wins} (or \emph{loses}) if it is the name of the top (resp.~bottom) right-most interval and we perform a top (resp.~bottom) Rauzy-induction move.  When $T$ is Keane, Rauzy-Veech induction never stops and the infinite associated path $\gamma$ on the Rauzy-Veech class is \emph{infinite-complete}, i.e.~all \emph{names} win and lose infinitely often (as it is shown in  \S~$1.2.3$ in \cite{MMY:Coh}, or also \cite{Yo}). Therefore, 
one can define an acceleration of Veech and Zorich induction whose steps are maximal steps such that all names but  one win\footnote{More precisely, the acceleration considered is given by the increasing sequence of times $(n_{d-1}(k))_{k\in\mathbb N}$ defined inductively so that $n_{d-1}(k+1)$ is the largest $n>n_{d-1}(k)$ so that no more than $d-1$ names are taken by the path $\gamma^{(n)}$ for $n_{d-1}(k)\leq n \leq n_{d-1}(k+1)$.}.
Let us denote by $I^{(k)}$, $k\in \mathbb{N}$ the subsequence of Rauzy-Veech inducing intervals corresponding to this acceleration.
The authors denote by $\lambda^{(k)}= (\lambda^{(k)}_{\alpha})_\alpha$ the vector of lengths $\lambda^{(k)}_\alpha=|I^{(k)}_{\alpha}|$ of the continuity intervals $I^{(k)}_\alpha$ of the IET obtained as first return to $I^{(k)}$.
They denote by $Z(k)$ the matrices of the corresponding acceleration (in the sense of \S~\ref{sec:acc}) of the Rauzy-Veech length cocycle and by $Q(k,l):= Z(k+1)\cdots Z(l)$ their product for $k<l$, so that (as stated in \S~1.2.3 in \cite{MMY:Coh}) we have:
\begin{equation}\label{eq:lengthrel}
 \lambda^{(k)}=Z(k+1)\lambda^{(k+1)}, \qquad {\lambda^{(l)}= Q(k,l) \lambda^{(k)}}.
\end{equation}
Notice that, by properties of Rauzy-Veech induction, if  we define  $h^{(k)}= (h^{(k)}_{\alpha})_{\alpha\in\mathcal A}$ to be the vector whose entries are the return times $h^{(k)}_\alpha$ of $I^{(k)}_{\alpha}$ to $I^{(k)}$, then (denoting by $A^*$ the transpose of a matrix $A$)  we also have that 
\begin{equation}\label{eq:heightrel}
 h^{(k+1)}=Z(k+1)^* h^{(k)}, \qquad h^{(k)}= Q(k,l)^* h^{(l)}.
\end{equation}
The definition of Roth-type IET involves three conditions, called $(a), (b) $ and $(c)$ in \cite{MMY:Coh}.  For our purposes, we need  only condition $(a)$, which constraints the growth of the matrices $Z(k)$ in relation to their product $Q(k)$ (which grows exponentially), imposing in particular that $Z(k)$ grows subexponentially:
\begin{definition}[see \S~1.3.1 in \cite{MMY}]\label{def:Roth}
An IET satisfies condition $(a)$ of the Roth-type condition if for every $\epsilon>0$ there exists $C_\epsilon>0$ such that
\[
\Vert Z(k+1)\Vert \leq C_\epsilon \Vert Q(k)^*\Vert ^\epsilon, \qquad \forall k\in \mathbb{N}.
\]
\end{definition}

\subsubsection{Distortion control for Roth-type IETs}\label{sec:Rothconsequence}
In  \cite{MMY:Coh} it is shown that, from the condition $(a)$ in Definition~\ref{def:Roth},
one can deduce  the following control of \emph{distortion} of the entries of the  lengths
$\lambda^{(k)} = (\lambda^{(k)}_\alpha)_\alpha$ at the special times given by the induction (see \S~\ref{sec:Rothdef}).
\begin{lemma}[Proposition in \S~1.3.1 in \cite{MMY:Coh}]\label{lemma:MMYcontrol}
If $T$ is Roth type, then for {every} $\epsilon>0$ there exists $C_\epsilon$ such that for {all} $k\geq 0$,
\[
\max_\alpha \lambda^{(k)}_\alpha \leq C_\epsilon  \Vert Q(k)\Vert^\epsilon \min_\alpha \lambda^{(k)}_\alpha .
\]
\end{lemma}
\noindent From this we can also deduce an analogous bound for the distortion of the vectors
 $h^{(k)} = (h^{(k)}_\alpha)_{\alpha\in\mathcal A}$ of return times (defined in \S~\ref{sec:Rothdef}).
\begin{lemma}\label{cor:hcontrol}
If $T$ is Roth type, then for {every} $\epsilon>0$,  there exists
 $C_\epsilon>0$  such that for {all} $k\geq 0$,
\[
\max_\alpha h^{(k)}_\alpha \leq  C_\epsilon  \Vert Q(k)\Vert^\epsilon \min_\alpha h^{(k)}_\alpha.
\]
\end{lemma}

\begin{proof}
In view of Remark~1 in \cite[\S~1.3.1]{MMY:Coh}, there exists $C>0$ such that
\[\min_{\alpha,\beta}Q_{\alpha\beta}(k)\geq C^{-1}\|Q(k)\|^{1-\epsilon}.\]
As $h^{(k)}= Q^*(k) h^{(0)}$, we have
\[\frac{h^{(k)}_\alpha}{h^{(k)}_\beta}\leq \frac{\max_\gamma Q_{\gamma\alpha(k)}}{\min_\gamma Q_{\gamma\beta}(k)}\leq \frac{C\|Q(k)\|}{\|Q(k)\|^{1-\epsilon}}= C\|Q(k)\|^{\epsilon},\]
which gives our claim.
\end{proof}


\subsubsection{Distortion control for Roth-type IETs}\label{sec:Rothdistortion}
 Let us now show that Lemmas~\ref{lemma:MMYcontrol}~and~\ref{cor:hcontrol} together with Definition~\ref{def:Roth} of Roth type we can prove Lemma~\ref{lemma:Roth_distortion}. 

\begin{proof}[Proof of Lemma~\ref{lemma:Roth_distortion}]
Given a Keane IET which is of Roth-type, consider the sequence $(I^{(k)})_{k\in\mathbb N}$ 
of the inducing intervals for the acceleration used in \cite{MMY:Coh}. 
Using the notation of this paper, we now denote the continuity subintervals of the $k^{th}$ induced map by $I^{(k)}_{j}, 1\leq j\leq d$, and $h^{(k)}_j$ the corresponding return time to $I^{(k)}$ (notice that to pass from the intervals $I^{(k)}_\alpha$, $\alpha \in \mathcal{A}$ and $I^{(k)}_j, 1\leq j\leq d$ one has to \emph{relabel} the intervals, but the order does not matter for maximum and minimum lengths estimates).

\smallskip
\noindent {\it Heights distortion.} It follows from the height relation \eqref{eq:heightrel},
 the definition of norm  (given for a vector $h=(h_j)_j$ by $\Vert h \Vert = \sum_j |h_j|$ and for a non-negative matrix $A=(A_{ij})_{ij}$ by $\Vert A\Vert  = \sum_{ij}A_{ij} = \Vert A^*\Vert$, where $A^*$ denotes the transpose) and Definition~\ref{def:Roth} of condition $(a)$ of Roth-type
 that, for every $\epsilon>0$, for some $C_{\epsilon/2}>0$,
\begin{equation}\label{eq:heightsdistortion}
\max_j h^{(k+1)}_j \ \leq \Vert  h^{(k+1)}\Vert \leq \Vert Z(k+1)^*\Vert  \Vert  h^{(k)} \Vert
\leq C_{\epsilon/2} \Vert Q(k)\Vert^{\epsilon/2} d \max_j h^{(k)}_j.
\end{equation}
Since, in view of the return times transformation relation \eqref{eq:heightrel}, we have that $h^{(k)}=Q(k) h^{(0)}$ where $h^{(0)}$ is the (column) vector with all entries $1$ (which are the trivial return times of subintervals of $I^{(0)}:=[0,1]$ to $I^{(0)}$ itself), we can estimate $\Vert Q(k)\Vert\leq d \max_j h^{(k)}_j$. In view of  Lemma~\ref{cor:hcontrol}, this gives
\[\|Q(k)\|^{\epsilon}\max_j h^{(k)}_j\leq d(\max_j h^{(k)}_j)^{1+\epsilon}\leq dC_{\epsilon/4}^2\|Q(k)\|^{\epsilon(1+\epsilon)/4}(\min_j h^{(k)}_j)^{1+\epsilon}.\]
It follows that
\[\|Q(k)\|^{\epsilon/2}\max_j h^{(k)}_j\leq \|Q(k)\|^{\frac{3}{4}\epsilon-\frac{1}{4}\epsilon^2}\max_j h^{(k)}_j\leq  dC_{\epsilon/4}^2(\min_j h^{(k)}_j)^{1+\epsilon}.\]
Together with \eqref{eq:heightsdistortion}, this gives the first part of the conclusion of Lemma~\ref{lemma:Roth_distortion} (for any constant $M=M(\epsilon)$ such that $M\geq dC_{\epsilon/2}C^2_{\epsilon/4}$).

\smallskip
\noindent {\it Lengths distortion.}
 The proof of the second part (the lengths distortion bound), is analogous: starting from the relation $\lambda^{(k)}= Z(k+1)\lambda^{(k+1)}$
 and using Condition (a) in Definition~\ref{def:Roth}, the bound on ratios of lengths given by Lemma~\ref{lemma:MMYcontrol} (for $k+1$ and $\epsilon/2$) and the observations that
 $\min_j \lambda^{(k+1)}_j\Vert Q(k+1) \Vert \leq 1$ and $\Vert Q(k)\Vert\leq \Vert Q(k+1)\Vert$, we get:
\begin{align*}
\max_j \lambda^{(k)}_j  & \leq \Vert  \lambda^{(k)}\Vert \leq \Vert Z(k+1)\Vert  \Vert  \lambda^{(k+1)}\Vert
\leq C_{\epsilon/2} \Vert Q(k)\Vert^{\epsilon/2} \, \Big( d \max_j \lambda^{(k+1)}_j \Big) \\ & \leq d C_\epsilon \Vert Q(k+1)\Vert^{\epsilon/2} \max_j \lambda^{(k+1)}_j\leq
d  C^2_{\epsilon/2} \Vert Q(k+1)\Vert^{\epsilon} \,   \min_j \lambda^{(k+1)}_j \leq d C_{\epsilon/2}^2\Big(\min_j \lambda^{(k+1)}_j\Big)^{1-\epsilon},
\end{align*}
which  gives the desired estimate (enlarging $M=M(\epsilon)$ if necessary so that we also have that $M\geq dC_{\epsilon/2}^2$).
\end{proof}

\section{Adaptation of the Rauzy cocycle accelerations}\label{sec:app}
In this Appendix we explain how the statement of Proposition~\ref{prop:accelerationset} follows from the proof of Proposition 4.7 in \cite{Ul:abs}.
The proof of this Proposition  in \cite{Ul:abs} is based on repeated applications of two Lemmas (stated in \S~\ref{sec:generalcocycleLemmas}) which produce (via Lusin-type arguments) an inducing set for the (invertible) Rauzy-Veech induction on which one has quantitative control on the (backward) growth of the Rauzy-Veech \emph{heights} cocycle (see definitions in \S~\ref{sec:lengthcocycle} and \S~\ref{sec:heightcocycle} below).
 The key observation is that Rauzy-Veech induction  displays a form of \emph{local product structure} (i.e.~the length datum $\lambda$ alone determines  the future iterates of the renormalization and of the lengths cocycle which is locally constant in $\tau$, while    the height datum $\tau$ alone determines  the \emph{past} iterates of the renormalization and of the heights cocycle which is locally constant in $\lambda$, see \S~\ref{sec:localproductstructure}). For this reason, one can show that the needed Lemmas (which were stated and proved for general integrable cocycles in \cite{Ul:abs}), in the special case of the Rauzy-Veech height cocycles produce an inducing set which has the desired product structure, see \S~\ref{sec:newLemmasproofs} and \S~\ref{sec:adaptation} in this Appendix.
\subsection{Integrable cocycles controlled growth lemmas}\label{sec:generalcocycleLemmas}
We now state the two general  Lemmas about integrable cocycles as they were stated and proved in \cite{Ul:abs} and then used for the Rauzy-Veech height cocycle.
We first recall some terminology.
\subsubsection{Cocycles and integrability}\label{sec:cocyclesbasics}
Let $(X,\mu, F)$ be a discrete dynamical system, where $(X, \mu)$ is a probability space and $ F$ is a $\mu$-measure preserving map on $X$.
A \emph{cocycle}  $A : X  \rightarrow SL(d,\mathbb{Z})$ ($d\times d$ invertible matrices) over $(X,\mu, F)$ is a  measurable map which produces the
 deterministic matrix product $A_F^{n}(x) = A_{n-1} (x)  \cdots A_1(x)  A_0 (x) $,  where  for any $n\geq 0$ we set $ A_n (x): = A(F^n x)$. This product satisfies
the  \emph{cocycle identity} 
\be \label{cocycleid}
A_F^{m+n}(x) = A_F^{m}(F^n x)A_F^{n}(x), \quad \text{for\ all}\ m,n \geq 0\ \text{and} \ x \in X.
\ee
We will also deal with the \emph{inverse dual cocycle} $(A^{-1})^*$, 
where $M^*$ denotes the transpose of $M$.

 If $F$ is invertible, let us extend the definition $A_{n}: = A(F^{n}x)$ to {every} $n\in \mathbb{Z}$.
Consider the cocycle $B:=A\circ F^{-1}$  over $F^{-1}$ and 
notice that if, for $n>0$, we define
$A^{-n}_F(x):= A_{-n}(x)\cdots A_{-1}(x)$, then we have that $B^n_{F^{-1}}(x) =  A^{-n}_F(x)$, since
\begin{equation}\label{eq:backwardcocycle}
B^n_{F^{-1}}(x) = B (F^{-n+1 }x) \cdots B (F^{-1}x)B(x) = A (F^{-n }(x)) \cdots A(F^{-1}x) = A^{-n}_F(x).
\end{equation}

If $Y \subset X$ is a measurable subset with positive measure, the \emph{induced cocycle} $A_Y$ on $Y$ is a cocycle over $(Y, \mu_Y , F_Y)$, where $F_Y$ is the induced map of $F$ on $Y$ and $\mu_Y = \mu/\mu(Y)$  and $A_Y(y)$ is defined for all $y\in Y$ which return to $Y$ and is given by
$$A_Y(y) = A
(F^{r_Y(y) -1 } y)
\cdots A\left(F y \right) A\left(y\right),$$
where $r_Y(y)= \min \{ r\in\N \st F^r y \in Y\} $ is the first return time.
The induced cocycle is an acceleration of the original  cocycle, i.e.\ if $\{n_k\}_{k\in \mathbb{N}} $ is the infinite sequence of return times of some $y\in Y$ to $Y$ (i.e.\ $F^{n} y \in Y$ iff $n=n_k$ for some $k\in \mathbb{N}$ and $n_{k+1}> n_k$) then $(A_Y)_k(y) =A_{n_{k+1}-1 }(y)   \cdots   A_{n_{k}+1}(y) A_{n_k}(y) $.
We say that $x \in X$ is \emph{recurrent} to $Y$ if there exists an infinite increasing sequence $\{n_k\}_{k\in \mathbb{N}} $ such that $F^{n_k} x \in Y$. Let us extend the definition of the induced cocycle $A_Y$ to all $x \in X$ recurrent to $Y$.
If the sequence  $\{n_k\}_{k\geq 0} $ is increasing and  contains all $n\geq 0$ such that  $F^{n} x \in Y$, let us say that $x $ \emph{recurs to} $Y$ \emph{along}  $\{n_k\}_{k\geq 0} $. In this case,  let us set, for any $n\geq 0$:
\bes
A_Y(x) :=  A_Y( y ) A^{n_0}_F (x) , \quad \mathrm{where}\,\,  y:= F^{n_0} x \in Y\quad\text{and}  \quad (A_Y)_{n}(x) :=  (A_Y)_n (y )\quad\text{for}\quad n\geq 1. 
\ees
If $F$ is ergodic, $\mu$-a.e.~$x \in X$ is recurrent to $Y$ and hence $A_Y$ is defined on a full measure set of $ X$.

\subsubsection{Controlled growth general cocycles lemmas}\label{sec:cocyclesgrowth}
Consider  the norm  $\norm{A}  = \sum _{ij} |A_{ij}|$  on matrices. 
Remark that with this choice $\norm{ A } = \norm{ A ^* }$.
A cocycle over $(X , \mu,F)$ is called \emph{integrable} if $\int_X \ln \|A(x)  \| \ud \mu(x) < \infty$.
Let us recall that integrability is the assumption which allows to apply Oseledets Theorem. 
Recall that if $A$ is an integrable cocycle over $(X , \mu, F)$ assuming values in $SL(d,\mathbb{Z})$, then
the inverse  cocycle $A^{-1}$ is also integrable. If $F$ is invertible, also the  cocycle $B:= A\circ F^{-1}$ over  $(X , \mu, F^{-1})$ is integrable. Moreover,  any induced cocycle $A_Y$ of $A$ on a measurable subset $Y \subset X$ is integrable.

Let us now state the  two Lemmas which were proved and used in {\cite{Ul:abs}}.
For $m<n $, let us denote by\footnote{The reader should remark that here the order of the matrices in the product is the inverse order than the one used in (\ref{cocycleid}). This notation is convenient since we will apply it to matrices  $Z$ where $Z^{-1}$ is the Rauzy cocycle.}
\begin{equation} \label{def:Amn}
A^{(m,n)}(x)=  A_m (x) A_{m+1} (x) \dots A_{n-1}(x), \qquad \text{where} \ A_n(x):= A(F^n x)\ \forall n\in \mathbb{N}.
\end{equation}
\begin{lemma}[Lemma~2.1 in {\cite{Ul:abs}}]\label{lemma1}
Assume $A^{-1}$ is an integrable cocycle over an ergodic and invertible $(X,\mu, F)$. 
There exists a measurable $E_1 \subset X$  with positive measure\footnote{The same proof gives that for each $\epsilon>0$ there exists $E_1$ with $\mu(E_1) > 1-\epsilon$.}  and a constant $\overline{C}_1>0$ such that for all $x\in X$ recurrent to $E_1$ along the sequence $\{n_k \}_{k\geq 0}$ we have
\begin{equation*}\label{expgrowth0}
\frac{ \ln \| A^{(n, n_k ) } (x) \|}{n_k - n} \leq \overline{C}_1, \qquad  \forall \, 0 \leq n < n_k .
\end{equation*}
\end{lemma}

\begin{lemma}[Lemma~2.2   in {\cite{Ul:abs}}]\label{lemma2}
Under the same assumptions of Lemma \ref{lemma1}, for each $\epsilon>0$ there exists a  measurable $E_2 \subset X$ with 
positive measure
and a constant $\overline{C}_2>0$ such that if $x\in X$ is recurrent to $E_2$ along the sequence $\{n_k \}_{k\geq 0}$, we have
\be\label{subexpgrowth0}
 \| A_{n_k - n  } (x) \|  \leq \overline{C}_2 e^{\epsilon n }, \qquad \forall \,  0\leq n \leq n_k. 
 \ee
\end{lemma}
In the next section we specialize these Lemmas to the Zorich height and lengths cocycles and show that, with essentially the same proof, in view of the structure of these cocycles, they yield sets with the desired product structure of the inducing set.

\subsubsection{Growth control for Rauzy-Veech lengths and heights cocycles}\label{sec:RVcocyclesLemmas}
Two cocycles (dual to each other) associated by the Rauzy-Veech algorithm to the Veech map $\mathcal{V}$ and consequently to its Zorich acceleration $\mathcal{Z}$ (see \S~\ref{sec:backgroundRV}). Let us first recall the definition of the two  cocycles (dual to each other), namely  the length cocycle (in \S~\ref{sec:lengthcocycle}) and the height one (in \S~\ref{sec:heightcocycle}). We use here the notation adopted in \cite{Ul:abs}. We then explain the product structure of the invertible induction (see \S~\ref{sec:localproductstructure}), then state and prove the specific form of the two growth Lemmas in this setting in \S~\ref{sec:newLemmasproofs}. We refer to \S~\ref{sec:backgroundRV} for the notation relative to the invertible Rauzy-Veech induction parameter space.

\subsubsection{Rauzy-Veech lengths  cocycles.}\label{sec:lengthcocycle}
For each $T= (\pi,{\lambda})$ for which $\mathcal{V} (T) $ is defined, we define the matrix $B=B(T)\in SL(d, \mathbb{Z})$  such that ${\lambda} = B^* \cdot  {\lambda}'$, where ${\lambda}' $ satisfies $\mathcal{V} (T)  = (\pi',{\lambda}'/|{\lambda}'|)$. In particular, $|{\lambda}'|$ is the Lebesgue measure (length) of the inducing interval $I'$  on which $\mathcal{V} (T)$ is defined.

Since $\mathcal{Z}$ is an acceleration of $\mathcal{V}$,
$Z(T) = B^*_\gamma := B^*_{\gamma_1}\cdots B^*_{\gamma_n}$
 is a product of the elementary matrices $B_{\gamma_i}$  (see in \S~\ref{sec:products}) associated to a path $\gamma=\gamma_1\cdots \gamma_n$ on the Rauzy graph $\mathcal{R}$.
 The map  $Z^{-1}\colon X \rightarrow SL(d,\mathbb{Z})$ is a cocycle over $\mathcal{Z}$, that we will call the (Zorich)  \emph{lengths cocycle} since it describes how the lengths transform (also known as Zorich acceleration of the \emph{Rauzy-Veech cocycle}).

  If $T = (\pi,{\lambda})$ satisfies the Keane condition so that  its Zorich orbit $(\mathcal{Z}^n (T))_{n\in \mathbb{N}}$ is infinite,  then we set $\lambda^{(0)}:= \lambda$ and denote by
   $T^{(n)}:=  \mathcal{Z}^n (T) $ the IET  obtained at the $n^{th}$ step of the Zorich algorithm and write   $T^{(n)}=(\pi^{(n)},{\lambda}^{(n)}/|\lambda^{(n)}|)$,  where $| {\lambda}^{(n)}|= Leb(I^{(n)})$.
If we define $\RL{n} = \RL{n} (T) :=  Z(\mathcal Z^n (T))$  iterating the lengths relation, we get that
$$ {\lambda} ^{(n)} = ( \RLp{n})^{-1 }  {\lambda}^{(0)} , \qquad \text{where} \ \RLp{n}= \RLp{n}(T) := \RL{0} \cdot \ldots \cdot \RL{n-1}.$$

\subsubsection{Return times and  heights cocycle.}\label{sec:heightcocycle}
Associated to an infinite  Zorich orbit $(\mathcal{Z}^n (T))_{n\in \mathbb{N}}$ and the sequence    $((\pi^{(n)},{\lambda}^{(n)}))_{n\in\mathbb N}$ of  IETs obtained inducing $T$ on the intervals $(I^{(n)})_{n\in\mathbb N}$ the algorithm also produces a sequence  of vectors  $({h}^{(n)})_{n\in\mathbb N} $ of  \emph{return times}, or \emph{heights}\footnote{The vector $h^{(n)}$ of return times is often called \emph{height} vector since the action of the initial IET $T=T^{(0)}$ can be seen in terms of Rohlin towers over $T^{(n)}= \mathcal{Z}^n(T)$ and $h^{(n)}_{i}$ is in this representation the height of the tower with base $I^{(n)}_i$.}: here $h^{(n)}= (h^{(n)}_i)_{1\leq i\leq d}\in \mathbb{R}^d$ denotes the vector whose entry $h^{(n)}_i$ for $1
\leq i \leq d$ gives the return time of any $x\in I^{(n)}_i$ to $I^{(n)}$.
One can verify that
  $h^{(n)}_j= \sum_{i=1}^d \RLp{n}_{ij}$ is the norm of the $j$-th column of $\RLp{n}$. 
Therefore the \emph{height vectors} ${h}^{(n)}$ 
can be obtained by applying the \emph{dual cocycle} $Z^*$, that we will call (Zorich) \emph{heights cocycle},
i.e., if ${h}^{(0)}$ is the column vector with all entries equal to $1$,
\be\label{heightsrelation}
{h}^{(n)} = (Z^*)^{n} {h}^{(0)} = (Z^{(n)})^* {h}^{(0)}.
\ee
\noindent Following \eqref{def:Amn}, for more general products with $m<n$, we use the notation
\begin{equation}\label{eq:product}
Z^{(m,n)}(x) := \RL{m}(x) \cdot   \RL{m+1}(x)  \cdot\ldots\cdot \RL{n-1}(x) .
\end{equation}
We omit $x$ from the notation when it is clear from the context and simply write $Z^{(m,n)} = \RL{m} \cdot   \RL{m+1}  \cdot\ldots\cdot \RL{n-1}$.

\subsubsection{Integrability and extension}
 Both  Zorich cocycles, the Zorich lengths cocycle $Z^{-1}$ and the Zorich heights cocycle $Z^*$, can be extended  to  cocycles over $(\widehat{X},  {\mu}_{\widehat{\Z}},\widehat{\Z})$ by defining the extended cocycle (for which we will use the same notation $Z^{-1}$ and $Z^*$) to be  constant on the fibers of  $p$, namely setting
 $Z(\pi, \lambda, \tau) :=  Z(\pi,\lambda)$ for {all} $(\pi, \lambda, \tau) \in \widehat{X}$, so that the extended cocycles $Z^{-1}$ and $Z^*$ over $\widehat{\mathcal{Z}}$ are constant on the fibers of $p$.
Both cocycles $Z^{-1}$ and $Z^*$  over the (natural extension of) Zorich acceleration $\widehat{\Z}$ are \emph{integrable} (recall \S~\ref{sec:cocyclesgrowth} for the definition): the integrability of the lengths cocycle 
was proved in Zorich \cite{Zo:fin}.
Moreover, as recalled in the previous \S~\ref{sec:cocyclesgrowth}, this implies that any cocycle induced  from them is again integrable. 

{
\subsubsection{Fibered cocycles.}\label{sec:locallyconstantcocycles}
In view of the local product structure of the Zorich map $\widehat{\mathcal{Z}}:\widehat{X}\to \widehat{X}$, cocycles which are locally constant on the fibers of the two  projections  $p(\pi,\lambda,\tau) = (\pi,\lambda)$ and   $p'(\pi,\lambda,\tau) = (\pi,\tau)$ (introduced in \S~\ref{sec:localproductstructure}) play an important role (since they project to cocycles on $\mathcal{Z}$ and the map $\mathcal{Z}^\ast$ respectively).

\smallskip
\noindent
Consider a product subset $E\subset \widehat{X}$, of the form
$$E=(\{ \pi\} \times E^+\times E^-), \qquad \text{where}\ \ E^+\subset \Delta_d  , \ \ E^-\subset \Theta_\pi $$
{Let us say that the cocycle $A_E: \widehat{X}\to SL(d,\mathbb{Z})$ obtained by inducing a cocycle $A$ over $\widehat{\mathcal{Z}}:\widehat{X}\to \widehat{X}$ to the subset $E$
is \emph{$\tau$-constant if it is constant along the fibers  $\{\pi\}\times \{\lambda\} \times E^-$  of the projection $p:E \to X$}, i.e.:
\begin{itemize}
\item[($\tau$-c)] for {every}  $\lambda \in E^+$,
\[
A_E (\pi, \lambda,\tau )= A_E (\pi,{\lambda}, \overline{\tau})  \qquad  \text{for\ {all}}\ \tau,\, \overline{\tau}\in E^-.
\]
\end{itemize}
Let us recall that, by definition, the extensions of the Zorich length cocycle $Z^{-1}$ and heights cocycle $Z^*$
as cocycles defined over $\widehat{\Z}: \widehat{X}\to \widehat{X}$ are both constant on the fibers  $\{\pi\}\times \{\lambda\} \times \Theta_\pi$   of $p:\widehat{X}\to X$.

\noindent Similarly, let us say that the cocycle $B_E: \widehat{X}\to SL(d,\mathbb{Z})$ obtained inducing a cocycle $B: \widehat{X}\to SL(d,\mathbb{Z})$ over $\widehat{\mathcal{Z}}^{-1}:\widehat{X}\to \widehat{X}$ is $\lambda$-\emph{constant} if it is constant along the fibers $\{\pi\}\times E^+ \times \{\tau\} $ of the projection $p'$ restricted to $E$, i.e.:
\begin{itemize}
\item[($\lambda$-c)] for {every} $\tau \in E^-$,
\[
B_E (\pi, \lambda,\tau )= B_E(\pi,\overline{\lambda}, {\tau})  \qquad  \text{for\ {all}}\ \lambda,\, \overline{\lambda}\in E^+.
\]
\end{itemize}
\noindent If we consider the cocycles $W^{-1}:=Z^{-1}\circ \widehat{\Z}^{-1}$ and  $W^{*}:=Z^{*}\circ \widehat{\Z}^{-1}$, the local product structure (and in particular the  Markovian structure described by \eqref{backwardMarkov1} and \eqref{eq:backwardMarkov2}) shows that also $W^{-1}$ and $W^*$ are constant on the fibers of the projection $p'$.

We will also need the following observation, which  we state as a Lemma:

\begin{lemma}\label{lemma:induced}
Let $\widetilde{W}$ being $W^{-1}$ or $W^*$.
If $W_E$  is  $\lambda$-constant as a cocycle  over $\widehat{\mathcal{Z}}_E^{-1}$ for a product set $E=\{\pi\}\times E^+\times E^-$
 and  $E'\subset E$ is saturated by fibers of $p'$, i.e.~of the form $E'=\{\pi\}\times E^+ \times Y$ for some measurable 
 $Y\subset E^-\subset \Theta_\pi$, then the induced cocycle  $W_{E'}$ over $\widehat{\mathcal{Z}}_{E'}^{-1} $ is again $\lambda$-constant. 
\end{lemma}
\begin{proof}
This follows since given any two $\lambda$, $\overline{\lambda}\in \Delta_d$,
the visit times to $E'\subset E$ (which is saturated  by $p'$-fibers) are equal (i.e.~$\widehat{\Z}_E^{-n}(\pi,\overline{\lambda}, \tau)\in E'$ iff $\widehat{\Z}_E^{-n}(\pi,{\lambda}, \tau)\in E'$).
It then follows by definition of $\widetilde{W}_{E'}$ (and since $\widetilde{W}_E$ is $\lambda$-constant) that   $\widetilde{W}_{E'}(\pi,\lambda, \tau)= \widetilde{W}_{E'}(\pi,\overline{\lambda}, \tau)$, i.e.~that ($\lambda$-c) holds and $W_{E'}$ is also $\lambda$-constant.
\end{proof}}

\subsubsection{Backward growth under the local product structure}\label{sec:newLemmasproofs}
In view of the local product structure and the locally constant nature of the cocycles we are studying (described in the previous \S~\ref{sec:localproductstructure} and \S~\ref{sec:locallyconstantcocycles}),  the Lemmas  stated in \S~\ref{sec:cocyclesgrowth} can be specialized as follows.

\begin{lemma}[backward controlled growth for products]\label{expgrowthlemmaRV}
{ Let $A$ be a  cocycle over the induced cocycle $\widehat{\mathcal{Z}}_E$ of $\widehat{\mathcal{Z}}:\widehat{X}\to \widehat{X}$ to a product set $E = \{\pi\}\times E^+\times E^- \subset \widehat{X}$, where $\pi$ is an irreducible permutation.  Assume that the cocycle $A\circ \widehat{\Z}_E^{-1}$ is $\lambda$-constant and $A, A^{-1}$ are both integrable. } 
For {every}
 $\epsilon>0$ there exists a measurable subset  $Y_1 \subset E^-\subset \Theta_\pi$, invariant under rescalings and with $Leb(Y_1)>0$, and a constant $\overline{C}_1>0$ such that for all $(\pi,\lambda, \tau)\in \widehat{X}$ (forward) recurrent under $\widehat{\mathcal{Z}}_E$ to the product set
$\{\pi\}\times E^+ \times Y_1\subset E$
along the sequence $\{n_k \}_{k\in \mathbb{N}}$ we have
\be\label{expgrowth}
\frac{ \ln \| A^{(n, n_k ) } (\pi,\lambda, \tau) \|}{n_k - n} \leq \overline{C}_1, \qquad  \forall \, 0 \leq n < n_k .
\ee
(where we use the notation \eqref{eq:product} for the cocycle matrix products).
\end{lemma}
\noindent The proof is similar to  the proof of Lemma 2.1 in \cite{Ul:abs} (here stated  as Lemma~\ref{lemma1}), with the additional insight coming from the fiber structure of the inverse cocycle described in \S~\ref{sec:locallyconstantcocycles}.
\begin{proof}
Consider the cocycle $W:= A\circ \widehat{\mathcal{Z}}_E^{-1}$, which is, by assumption, $\lambda$-constant.
Since the cocycle $A $ and hence also $W$ is integrable, by Oseledets theorem, (using the notation introduced in \S~\ref{sec:cocyclesbasics})
 the  functions $  (\pi,\lambda, \tau) \mapsto \ln \| W_{\widehat{\mathcal{Z}}_E^{-1}}^{m} (\pi,\lambda, \tau)  \| / m $
 converge a.e.\ as $m\to+\infty$.
 By Egorov's theorem, 
 there exists a set $E_1\subset E$ such that $\mu_{\widehat{\mathcal{Z}}}^{(1)}(E_1)>0$ and  the convergence is uniform, so that $\ln \| W_{\widehat{\mathcal{Z}}_E^{-1}}^{m} (\pi,\lambda, \tau) \| \leq \overline{C}_1  m  $ for some $\overline{C}_1>0$ and all $(\pi,\lambda, \tau )\in E_1$ and all $m\geq 0$.
   Notice now that since the cocycle $W$ is
 $\lambda$-constant,
  the set $E_1$ is saturated by fibers of $p'$, i.e.~it is a product  $E_1= \{\pi\} \times E^+\times Y_1$ for some measurable subset $Y_1\subset E^-\subset \Theta_\pi$.  Moreover, it is clearly invariant by rescalings $\tau\mapsto t\tau$, $t>0$ (since the cocycle $W$ does not change if $\tau$ is replaced by $t\tau$ for some $t>0$).

If 
$n_k\in \mathbb{N}$ is such that  $\widehat{\Z}_E^{n_k} (\pi, \lambda, \tau) \in E_1$, 
we have that $\ln \| W_{\widehat{\mathcal{Z}}_E^{-1}}^{m}  (\widehat{\Z}_E^{n_k} (\pi, \lambda, \tau))\| \leq \overline{C}_1 m$
for all $m\geq 0$.
Hence, since  by \eqref{eq:backwardcocycle} (using the notation introduced in \S~\ref{sec:cocyclesbasics})
\begin{equation}\label{eq:cocyclerel}
W_{\widehat{\mathcal{Z}}_E^{-1}}^{m} \left( \widehat{\mathcal{Z}}_E^{n_k} (\pi,\lambda, \tau ) \right)
=  A^{-m} \left( \widehat{\mathcal{Z}}_E^{n_k} (\pi,\lambda, \tau ) \right) =  A_{n_k-m}A_{n_k-m+1} \cdots A_{n_k-1}  = A^{(n_k-m, n_k)}.
\end{equation}
Thus,  we get $\ln \Vert A^{(n_k-m, n_k)}\Vert \leq \overline{C}_1 m$ for {all} $m\geq 0$, which, changing
indices by $n=n_k-m$, gives the desired ~(\ref{expgrowth}) for  all $0\leq n\leq n_k$.

  To conclude we only have to check that $Leb(Y_1)>0$. For that, notice that the pushforward measure $p'_\ast\, \mu_{\widehat{\mathcal{Z}}}^{(1)}$ on $\Theta_\pi$ given by $p'_\ast \mu_{\widehat{\mathcal{Z}}}^{(1)} (U) :=  \mu_{\widehat{\mathcal{Z}}}^{(1)} \left((p')^{-1} U\right)$  is absolutely continuous with respect to the Lebesgue measure on $\Theta_\pi$. Thus, since $\mu_{\widehat{\mathcal{Z}}}^{(1)} (E)>0$ by construction, it follows by Fubini that $p'_\ast \mu_{\widehat{\mathcal{Z}}}^{(1)}(E)>0$ and hence that  $Leb(Y_1)>0$. This concludes the proof.
\end{proof}

\begin{lemma}[backward subexponential growth for matrices]\label{subexpgrowthlemmaRV}
{For $A$ and $\pi$  as in Lemma \ref{expgrowthlemmaRV},} for each $\epsilon>0$ there exists a  measurable $Y_2 \subset \Theta_\pi$ with $Leb(Y_2) > 0$
and a constant $\overline{C}_2>0$ such that if $(\pi,\lambda, \tau)\in \widehat X$ is recurrent under $\widehat{\mathcal{Z}}_E$ to $\{\pi\} \times E^+\times Y_2$ along the sequence $\{n_k \}_{k\in \mathbb{N}}$, we have
\be\label{subexpgrowth}
 \| A_{n_k - n } (\pi,\lambda, \tau) \|  \leq \overline{C}_2 e^{\epsilon n }, \qquad \forall \,  0\leq n \leq n_k. 
 \ee
\end{lemma}
{
\begin{proof}
Recall that since $\widehat{\Z}_E^{-1}$ is ergodic, if $f$ is integrable, the functions $\{ f\circ \widehat{\Z}_E^{-m}/(m+1)  \}_{m\in \mathbb{N}}$ converge to zero   for a.e.~$(\pi,\lambda, \tau) \in E$ 
 and hence, by Egorov theorem, are eventually uniformly less than $\epsilon$ on some positive measure set for $m\geq \overline{m}$. Since as already remarked in the proof of Lemma~\ref{expgrowthlemmaRV},
 $W$ is integrable as a cocycle over $\widehat{\mathcal{Z}}_E^{-1}$, applying this observation to $f= \ln \| W \| $, we can find a positive measure set $E_2\subset E$ and $\overline{C}_2>0$ (in order to bound also $ \| W (\widehat{\Z}_E^{-m}(\pi,\lambda, \tau)) \|$ for $(\pi,\lambda, \tau) \in E_2$ and    $0\leq m \leq \overline{m}$) such that if $(\pi,\lambda, \tau) \in E_2$, we have
 \[
 \norm{ W (\widehat{\Z}_E^{-m}  (\pi,\lambda, \tau))} \leq \overline{C}_2 e^{\epsilon (m+1)},\quad  \forall\ m \geq 0.
 \]
 When $\widehat{\mathcal{Z}}_E^{n_k} (\pi, \lambda, \tau) \in E_2$, since (recalling the definition of $W = A\circ \widehat{\mathcal{Z}}_E^{-1}$) we have that
 \[
 W_{-m} (\widehat{\mathcal{Z}}_E^{n_k} (\pi, \lambda, \tau)) ) = A_{-m}\circ (\widehat{\mathcal{Z}}_E^{-1}  (\widehat{\mathcal{Z}}_E^{n_k} (\pi, \lambda, \tau)) ) = A_{n_k-(m+1)}  (\pi, \lambda, \tau),
 \]
we get (\ref{subexpgrowth}).
We then reason as in the proof of Lemma~\ref{expgrowthlemmaRV} to conclude that, since $W$ is locally constant on the fibers of $p'$, the set $E_2$ has the form $\{\pi\} \times E^+ \times Y_2$ for some measurable subset $Y_2\subset E^-\subset\Theta_\pi$ and  that $Leb(Y_2)>0$. This concludes the proof.
\end{proof}}

\subsubsection{Adaptation of the proof of the trimmed Birkhoff sums bounds.}\label{sec:adaptation}
Let us now sketch how this specialized form of backward controlled growth Lemmas can be used to deduce Proposition~\ref{prop:accelerationset}.
In the proof of Proposition 4.1 in \cite{Ul:abs},
we show that Birkhoff sums of the derivative display a certain form of \emph{cancellations}, i.e. the closest visits to the positive and negative parts compensate. The good times $(n_k)$ such that the orbits along a tower are sufficiently well distributed for  these cancellations happen, are obtained accelerating and inducing several times to subsets obtained via the two Lemmas 2.2 and 2.3 of \cite{Ul:abs} (reproduced here as Lemma~\ref{lemma1} and ~\ref{lemma2}).  We show that, using the new  Lemmas~\ref{expgrowthlemmaRV} and Lemma~\ref{subexpgrowthlemmaRV} instead, the inducing set has the desired product structure.

\begin{proof}[Proof of Proposition~\ref{prop:accelerationset}]
To produce the inducing set, in \cite{Ul:abs} the  cocycle $Z^*$ is first accelerated by considering returns to a 
set $\mathcal{K}\subset \widehat{X}$ {which, as remarked in \cite{Ul:abs}, can be chosen of the form\footnote{{It is important for us in this proof that the compact set has this specific form (that corresponds to a \emph{finite} cylinder in the symbolic coding of Rauzy-Veech induction as a Markov chain) rather than being a general compact set.}} $\mathcal{K}=\{\pi\}\times\Delta_{B_\gamma}\times \Theta_{B_\gamma}$, where $B_\gamma$
 is a positive matrix of the Zorich cocycle associated to a neat path $\gamma$ (refer to 
  \eqref{def:simplex}  and  \eqref{def:Theta} for the definitions of  $\Delta_{B_\gamma}$ and $\Theta_{B_\gamma}$ respectively),  so that $B^*_\gamma= Z^{(N)}$ for some integer $N>0$ (so $N$ gives the number of matrices of the Zorich cocycle that are produced to get $B_\gamma$).

The induced cocycle (which gives a balanced acceleration)  is called $A$. Then, Lemma~2.1 is used in the proof of Lemma~4.3 of \cite{Ul:abs}, to produce an inducing set $E_B\subset \widehat{X}$ and a new acceleration $B=A_{E_B}$ given by returns to $E_B$. Then, the new cocycle $B$  is accelerated again at the very beginning of Proposition~4.1 of \cite{Ul:abs}, applying this time Lemma~2.2 of  \cite{Ul:abs} to get a positive measure  subset $E_C\subset E_B$ and a new accelerated cocycle $C=B_{E_C}$ given by returns to $E_C$. The new accelerated cocycle $C$ is finally accelerated one more time at the beginning of the proof of Proposition 4.2 of \cite{Ul:abs}, where Lemma~2.2 again gives a positive measure set $E_D\subset E_C\subset \widehat{X}$ so that the sequence of returns to $E_D$ gives the desired trimmed derivative Birkhoff sums upper bounds.

{Let us modify the construction as follows. Consider the height cocycle $Z^*$  over $(\widehat{X}, \widehat{\mu}_{\widehat{\Z}} , \widehat{\mathcal{Z}} )$, which is integrable and $\tau$-constant by definition. Therefore, the cocycle $W:= Z^* \circ \widehat{\mathcal{Z}}^{-1}$ is a cocycle over $(\widehat{X}, \widehat{\mu}_{\widehat{\Z}}^{(1)} , \widehat{\mathcal{Z}}^{-1} )$ which is also integrable and, as discussed in \S~\ref{sec:locallyconstantcocycles}, is $\lambda$-constant. If we accelerate this cocycle to consider returns to the compact set $\mathcal{K}:=\{\pi\}\times\Delta_{B_\gamma}\times \Theta_{B_\gamma}$, the induced cocycle will not be in general  $\lambda$-constant. Therefore,
 we consider instead the acceleration which corresponds to visits to
$$\widehat{\Z}^{{N}} (\mathcal{K})= \widehat{\Z}^{{N}} (\{\pi\}\times\Delta_{B_\gamma}\times \Theta_{B_\gamma})=\{\pi\}\times\Delta_d\times  \Theta_{B_\gamma B_\gamma},$$
where the above relation holds by the Markovian property \eqref{bisided_Markov} and the definition of $N$ (which gives the number of Zorich matrices produced to get $B_\gamma$).

For any $Y_0\subset  \Theta_{B_\gamma B_\gamma}$, which is invariant under rescaling and of positive measure let
\[\mathcal{K}(Y_0):=\widehat{\Z}^{{-N}}(\{\pi\}\times\Delta_d\times Y_0)=\{\pi\}\times\Delta_{B_\gamma}\times B_\gamma^*(Y_0).\]

Let us consider two induced cocycles $A:=Z_{\mathcal{K}(Y_0)}$ and $\widetilde{A}:=Z_{\widehat{\Z}^{{N}}(\mathcal{K}(Y_0))}$.
As $\widehat{\Z}^{{-N}}\mathcal{K}(Y_0)=\{\pi\}\times\Delta_d\times Y_0$,
in view of Lemma~\ref{lemma:induced}, the cocycle $\widetilde{W}:=\widetilde{A}\circ \widehat{\mathcal{Z}}_{\widehat{\Z}^{{N}}(\mathcal{K}(Y_0))}^{-1}$ over $\widehat{\Z}_{\widehat{\Z}^{{N}}(\mathcal{K}(Y_0))}^{-1}$ is $\lambda$-constant. By Lemma~\ref{expgrowthlemmaRV} or Lemma~\ref{subexpgrowthlemmaRV} respectively, there exists a subset $Y_1\subset Y_0$ invariant under rescaling and of positive measure such that if $\{n_k\}_{k\in\N}$ is the sequence of returns 
 to $\widehat{\Z}^{{N}}(\mathcal{K}(Y_1))$, then
\begin{equation*}\label{eq:til}
\|\widetilde{A}^{(n,n_k)}\|\leq e^{\overline{C}_1(n_k-n)}\quad\text{or}\quad \|\widetilde{A}_{n_k-n}\| \leq\overline{C}_2e^{\epsilon n}\quad\text{resp.}\quad \forall \,  0\leq n \leq n_k.
\end{equation*}
Note that $\{n_k\}_{k\in\N}$ is also the sequence of returns for $\widehat{\mathcal{Z}}_{\mathcal{K}(Y_0)}$ to $\mathcal{K}(Y_1)$. Since $A_n=B^*_\gamma\,\widetilde{A}_n\,(B^*_\gamma)^{-1}$ and $A^{(n,n_k)}=B^*_\gamma\,\widetilde{A}^{(n,n_k)}\,(B^*_\gamma)^{-1}$, we have
\begin{equation*}\label{eq:ntil}
\|{A}^{(n,n_k)}\|\leq \|B^*_\gamma\|\|(B^*_\gamma)^{-1}\|\,e^{\overline{C}_1(n_k-n)}\quad\text{or}\quad \|{A}_{n_k-n}\| \leq\|B^*_\gamma\|\|(B^*_\gamma)^{-1}\|\overline{C}_2\,e^{\epsilon n}\quad \text{resp.}
\end{equation*}
for all   $0\leq n \leq n_k$
Next, when  Lemma~\ref{lemma1}  or, respectively, Lemma~\ref{lemma2} are used in \cite{Ul:abs}, we  replace these generic Lemmas with, respectively,   Lemmas~\ref{expgrowthlemmaRV} and Lemma~\ref{subexpgrowthlemmaRV}. More precisely, first consider the induced cocycles
$A:=Z_{\mathcal{K}}$ and $\widetilde{A}:=Z_{\widehat{\Z}^{{N}}(\mathcal{K})}$ with $\mathcal{K}=\{\pi\}\times\Delta_{B_\gamma}\times \Theta_{B_\gamma}$ and $\widehat{\Z}^{{N}}(\mathcal{K})=\{\pi\}\times\Delta_d\times \Theta_{B_\gamma B_\gamma}$.
Next,  Lemma~\ref{expgrowthlemmaRV} applied to $\widetilde{A}$ gives $Y_B\subset  \Theta_{B_\gamma B_\gamma} \subset \Theta_\pi$ and the  set $E_B:= \{\pi\}\times \Delta_{B_\gamma}\times B^*_\gamma (Y_B)=\mathcal{K}(Y_B)$ such that $\|A^{(n,n_k)}\|\leq e^{\overline{C}_B(n_k-n)}$ for $0\leq n<n_k$, where $(n_k)$ are return times for $\widehat{\mathcal Z}_{\mathcal{K}}$ to $\mathcal{K}(Y_B)=E_B$.

Next, consider the induced cocycles
$B:=Z_{\mathcal{K}(Y_B)}$ and $\widetilde{B}:=Z_{\widehat{\Z}^{{N}}(\mathcal{K}(Y_B))}$.
Then  Lemma~\ref{subexpgrowthlemmaRV} applied to $\widetilde{B}$ gives $Y_C\subset  Y_B$ and the  set $E_C:= \{\pi\}\times \Delta_{B_\gamma}\times B^*_\gamma (Y_C)=\mathcal{K}(Y_C)$ such that $\|B_{n_k-n}\|\leq \overline{C}_Ce^{\epsilon n}$ for $0\leq n<n_k$, where $(n_k)$ are return times for $\widehat{\mathcal Z}_{\mathcal{K}(Y_B)}$ to $\mathcal{K}(Y_C)=E_C$.

In the last step, consider the induced cocycles
$C:=Z_{\mathcal{K}(Y_C)}$ and $\widetilde{C}:=Z_{\widehat{\Z}^{{N}}(\mathcal{K}(Y_C))}$.
Then again Lemma~\ref{subexpgrowthlemmaRV} applied to $\widetilde{C}$ gives $Y_D\subset  Y_C\subset  Y_B$ and the  set $E_D:= \{\pi\}\times \Delta_{B_\gamma}\times B^*_\gamma (Y_D)=\mathcal{K}(Y_D)$ such that $\|C_{n_k-n}\|\leq \overline{C}_De^{\epsilon n}$ for $0\leq n<n_k$, where $(n_k)$ are return times for $\widehat{\mathcal Z}_{\mathcal{K}(Y_C)}$ to $\mathcal{K}(Y_D)=E_D$.
Finally we set $Y:= B^*_\gamma(Y_D)$.

We then claim that the 
visits to the set 
 $E:=(E_D)^{(1)}= (\{\pi\}\times \Delta_{B_\gamma} \times Y)^{(1)}$ provide the sequence of times $(n_\ell)_{\ell\in\N}$ for which the conclusion of Proposition~\ref{prop:accelerationset}, namely the Birkhoff estimate \eqref{eq:propconclusion}, holds.
For this sequence,  all the arguments of the proofs in \cite{Ul:abs}, namely the effective equidistributions and cancellations arguments, can be applied \emph{verbatim}, without needing any modification, when using these new Lemmas (since the only properties used in the proofs are balance, positivity of the matrices and the backward controlled growth given in \cite{Ul:abs} by Lemmas 2.1 and 2.2, which all hold also for the new accelerations since the backward controlled growth estimate is the same now given by  Lemmas~\ref{expgrowthlemmaRV} and Lemma~\ref{subexpgrowthlemmaRV}). Thus, the Birkhoff estimate for the trimmed derivatives holds along the Rohlin towers  of the  sequence  $(n_{\ell})_{\ell\in\N}$.
%
}}\end{proof}

{
\section{
Measurability of pairs of spectrally disjoint flows}\label{sec:meas}
In this Appendix, we first discuss (in \S~\ref{sec:reduction1} and \S~\ref{sec:reduction2}) some classical reductions  which are used to induce a topology and define the period map on  locally Hamiltonian flows. The main results of this Appendix are stated and proved in \S~\ref{sec:measurability} and are used in  the proof of Theorem~\ref{thm:main_dj} in  \S~\ref{sec:reduction}  to show that the set of \emph{pairs} of spectrally disjoint locally Hamiltonian flows is measurable and to make sense of the notion of  \emph{typical} pairwise spectral disjointness (which is the main result of Theorem~\ref{thm:main_dj}).}

\subsection{
{
Identification with closed one forms and topology}\label{sec:reduction1}}
{
Let $M$ be a compact connected orientable surface equipped with a smooth (namely $C^\infty$) area form $\omega$ of total area one. By Moser's argument (see \cite{Mos}),
all area one 
smooth area forms are 
$C^{\infty}$-diffeomorphic. It follows that any locally Hamiltonian flow on $M$ is $C^\infty$-conjugate to a locally Hamiltonian flow on $M$ preserving the area form $\omega$. Denote by $\mathcal{LH}(M,\omega)$ the space of such flows.

Flows in $\mathcal{LH}(M,\omega)$
can be identified with the space of closed $1$-forms on $M$.
{Indeed, as explained in \S~\ref{sec:locHamflows}, any such closed $1$-form $\eta$ determines a flow in $\mathcal{LH}(M,\omega)$. Conversely, given a flow  $(\varphi_t)_{t\in\mathbb{R}} \in \mathcal{LH}(M,\omega)$ generated by a vector field $X$, 
letting $\eta:=\imath_X\omega=\omega( X, \,\cdot \,)$,  one gets that $\eta$ is a smooth closed 1-form such that $d\eta=0$, since $(\varphi_t)_{t\in\mathbb{R}}$ preserves $\omega$.}
Thus, equivalently,  $\mathcal{LH}(M,\omega)$  can be identified with the space of multivalued Hamiltonian functions on $M$. This identification provides a natural $C^1$ metric on $\mathcal{LH}(M,\omega)$, whose induced topology is Polish. Denote by $\mathcal{U}_{\min}(M,\omega)$ the subset of  Morse $1$-forms for which the corresponding locally Hamiltonian flow has no saddle loop {homologous to zero}. Notice that the set $\mathcal{U}_{\min}(M,\omega)$ is open in $\mathcal{LH}(M,\omega)$ {and thus inherits a Polish topology.}

\subsection{{Reduction to a fixed singularity set}}\label{sec:reduction2}
Every flow in $\mathcal{U}_{\min}(M,\omega)$ has $\kappa:=2g-2$ Morse saddle fixed points (by the Poincar\'e-Hopf Theorem). Let $\Sigma\subset M$ be a finite subset with $\kappa$ elements. In view of \cite{MiVi}, for any two $\kappa$-element subsets of $M$, there exists a $C^\infty$ diffeomorphism of $M$ preserving $\omega$ and mapping one subset onto the other. It follows that any flow from  $\mathcal{U}_{\min}(M,\omega)$ is $C^\infty$ conjugate to a flow in  $\mathcal{U}_{\min}(M,\omega)$ such that its set of fixed points is equal to $\Sigma$. Denote by  $\mathcal{U}_{\min}(M,\omega,\Sigma)$ the set of such flows, which is in turn a closed subset of the set $\mathcal{U}_{\min}(M,\omega)$. Let $n:=2g+\kappa-1=4g-3$. Then $n$ is the dimension of the relative $1$-homology space $H_1(M,\Sigma,\R)$. Let $\gamma_1, \dots, \gamma_n\in H_1(M, \Sigma, \mathbb{Z})$ be a basis of $H_1(M, \Sigma, \mathbb{R})$.}

\smallskip

{\subsection{Measurability of periods of spectrally disjoint pairs}\label{sec:measurability}
Consider the subset $\mathcal{SD}\subset  \mathcal{U}_{\min}(M,\omega)\times \mathcal{U}_{\min}(M,\omega)$ of \emph{pairs} of flows $\big((\varphi_t)_{t\in\R},(\psi_t)_{t\in\R}\big)$ such that $(\varphi_t)_{t\in\R}$ and $(\psi_t)_{t\in\R}$ are spectrally disjoint.
{The first result of this Appendix is the following:
\begin{proposition}\label{thm:DSGd}
The subset $\mathcal{SD}$ of
\emph{spectrally disjoint pairs}
is a $G_\delta$ subset in the product $\mathcal{U}_{\min}(M,\omega)\times \mathcal{U}_{\min}(M,\omega)$ with the product topology.
\end{proposition}}
{
\begin{proof}
Denote by $P(\R)$ the space of all probability Borel measures on $\R$ equipped with the
weak topology. As it was shown in \cite{Bu-Ma}, the set
\[\mathcal{O}:=\{(\mu,\nu)\in P(\R)\times P(\R): \mu\text{ and }\nu\text{ are orthogonal}\}\]
is a $G_\delta$ subset of $P(\R)\times P(\R)$. Take any $f,g\in L^2_0(M,\omega)$. If $\|f\|_{L^2}=\|g\|_{L^2}=1$, then the maps
{\begin{align*}
\mathcal{U}_{\min}(M,\omega) &\to  P(\R)
\qquad \qquad \text{and} \ &
\mathcal{U}_{\min}(M,\omega) \to P(\R)\\
(\varphi_t)_{t\in\R} & \mapsto \sigma^{(\varphi_t)_{t\in\R}}_f
 \qquad &
 (\psi_t)_{t\in\R}  \mapsto \sigma^{(\psi_t)_{t\in\R}}_g
\end{align*}
(where $\sigma^{(\varphi_t)_{t\in\R}}_f$ denotes the spectral measure of $f$ with respect to the flow $(\varphi_t)_{t\in\R}$)
are continuous when $\mathcal{U}_{\min}(M,\omega)$ is endowed with the Polish topology introduced in \S~\ref{sec:reduction1}}. It follows that for any $f,g\in L^2_0(M,\omega)$, the
set
\[\mathcal{O}(f,g):=\{((\varphi_t)_{t\in\R},(\psi_t)_{t\in\R})\in \mathcal{U}_{\min}(M,\omega)\times \mathcal{U}_{\min}(M,\omega):(\sigma^{(\varphi_t)_{t\in\R}}_f,\sigma^{(\psi_t)_{t\in\R}}_g)\in\mathcal{O}\}\]
is a $G_\delta$ subset. Let $\{f_i:i\geq 1\}$ be a dense subset in $L^2_0(M,\omega)$. Then the set
\[\bigcap_{i,j\geq 1}\mathcal{O}(f_i,f_j)\]
is a $G_\delta$ subset of $\mathcal{U}_{\min}(M,\omega)\times \mathcal{U}_{\min}(M,\omega)$. In view of Lemma~21 in \cite{Le96}, for any finite Borel measure $\mu$ on $\R$ and any flow $(\varphi_t)_{t\in\R}$, the set
\[\{f\in L_0^2(M,\omega):\sigma^{(\varphi_t)_{t\in\R}}_f\perp\mu\}\]
is closed in $L_0^2(M,\omega)$. It follows that
\[\bigcap_{i,j\geq 1}\mathcal{O}(f_i,f_j)=\mathcal{SD},\]
which completes the proof.
\end{proof}
}
{Consider now the set  $\underline{\mathcal{SD}}$ of \emph{pairs of periods} such that all corresponding pairs of flows belong to $\mathcal{SD}$ (which is used also in the proof of Theorem~\ref{thm:main_dj} \S~\ref{sec:reduction}), namely the set $\underline{\mathcal{SD}}$ given by:}
{\[\underline{\mathcal{SD}}=\{(\bar{v},\bar{w})\in \R^n\times \R^n: (Per\times Per)^{-1}(\{(\bar{v},\bar{w})\})\subset\mathcal{SD}\},\]}
where
\begin{equation}\label{eq:prodPer}
Per\times Per:\mathcal{U}_{\min}(M,\omega,\Sigma)\times \mathcal{U}_{\min}(M,\omega,\Sigma)\to\mathbb R^n\times\mathbb R^n,
\end{equation}
 is the product of the period map defined in \S~\ref{sec:measure}, see \eqref{def:Per}.

{The main result of this Appendix (that  is used to prove typicality of pairs of spectrally disjoint flows in the proof of Theorem~\ref{thm:main_dj}, see \S~\ref{sec:reduction}) is the following:}}
{
\begin{lemma}\label{lem:measSD}
The set $\underline{\mathcal{SD}}\subset\R^n\times \R^n$ is Lebesgue measurable.
\end{lemma}
\begin{proof}
Let us first show  that
\[
\underline{\mathcal{SD}}=((Per\times Per)(\mathcal{SD}^c))^c.\]
Indeed, we have that
\begin{gather*}
(\bar{v},\bar{w})\in ((Per\times Per)(\mathcal{SD}^c))^c\Leftrightarrow(\bar{v},\bar{w})\notin(Per\times Per)(\mathcal{SD}^c)\\
\Leftrightarrow (Per\times Per)^{-1}(\{(\bar{v},\bar{w})\})\cap \mathcal{SD}^c=\emptyset\Leftrightarrow
(Per\times Per)^{-1}(\{(\bar{v},\bar{w})\})\subset \mathcal{SD}.
\end{gather*}
In view of Proposition~\ref{thm:DSGd}, $\mathcal{SD}^c$ is a Borel subset in $\mathcal{U}_{\min}(M,\omega,\Sigma)\times \mathcal{U}_{\min}(M,\omega,\Sigma)$.
{Since the period map $Per$ (defined in \eqref{def:Per}) is continuous, also the product map $Per\times Per$ in \eqref{eq:prodPer}
 is continuous.}
As $\mathcal{U}_{\min}(M,\omega,\Sigma)\times \mathcal{U}_{\min}(M,\omega,\Sigma)$ is a Polish topological space, the continuous image $(Per\times Per)(\mathcal{SD}^c)$ is an analytic set in $\R^n\times\R^n$. By the classical Lusin-Souslin theorem, it is universally measurable, in particular, it is Lebesgue measurable. This gives the Lebesgue measurability of  $((Per\times Per)(\mathcal{SD}^c))^c$ as well, which completes the proof.
\end{proof}}}
%

\subsection*{Acknowledgements}
The authors would like to thank the Banach Center  in B\k{e}dlewo, Poland, since the initial discussions which sparked this project took place during the
 Conference ``Ergodic aspects of modern dynamics'' in honour of Mariusz Lema\'nczyk on his 60th birthday. 
 The authors are also grateful to Przemys{\l}aw Berk for several useful discussions {and to the anonymous referees for their comments which helped clarify some points and improve the presentation. In particular, we thank  referees for suggesting 
  the nicely phrased summary of the strategy included in $\S~\ref{sec:stratsing}$ and the need of adding the content in  Appendix~\ref{sec:measurability}}.
 C.~U.\ acknowledges the SNSF Consolidator Grant $TMCG-2\textunderscore 213663$ 
 as well as the SNSF Grant {\it Ergodic and spectral properties of surface flows} No.~$200021\textunderscore 188617$.
A.~K.\ was partially supported by the NSF grant DMS-2247572.  This research was partially supported by the Narodowe Centrum Nauki Grant 2022/45/B/ST1/00179. 


\end{document}